\documentclass[11pt,letterpaper]{amsart}

\usepackage[T1]{fontenc}
\usepackage{lmodern}
\usepackage{microtype}
\microtypesetup{expansion=true,protrusion=true}

\usepackage{amsmath,amssymb,mathtools,mathrsfs}

\usepackage{booktabs}
\usepackage{tikz-cd}

\usepackage{enumitem}
\usepackage{needspace}

\usepackage{xcolor}
\usepackage[
    colorlinks=true,
    linkcolor=blue!45!black,
    citecolor=blue!45!black,
    urlcolor=blue!55!black
]{hyperref}

\usepackage{aliascnt}
\usepackage[nameinlink,noabbrev]{cleveref}

\allowdisplaybreaks
\numberwithin{equation}{section}

\newtheorem{theorem}{Theorem}[section]

\newaliascnt{proposition}{theorem}
\newtheorem{proposition}[proposition]{Proposition}
\aliascntresetthe{proposition}

\newaliascnt{lemma}{theorem}
\newtheorem{lemma}[lemma]{Lemma}
\aliascntresetthe{lemma}

\newaliascnt{corollary}{theorem}
\newtheorem{corollary}[corollary]{Corollary}
\aliascntresetthe{corollary}

\newaliascnt{conjecture}{theorem}

\aliascntresetthe{conjecture}

\newaliascnt{question}{theorem}

\aliascntresetthe{question}

\theoremstyle{definition}

\newaliascnt{definition}{theorem}
\newtheorem{definition}[definition]{Definition}
\aliascntresetthe{definition}

\newaliascnt{example}{theorem}
\newtheorem{example}[example]{Example}
\aliascntresetthe{example}

\newaliascnt{convention}{theorem}

\aliascntresetthe{convention}

\theoremstyle{remark}

\newaliascnt{remark}{theorem}
\newtheorem{remark}[remark]{Remark}
\aliascntresetthe{remark}

\newaliascnt{warning}{theorem}

\aliascntresetthe{warning}

\newcommand{\ZZ}{\mathbb{Z}}
\newcommand{\NN}{\mathbb{Z}_{\geq 0}}
\newcommand{\QQ}{\mathbb{Q}}
\newcommand{\RR}{\mathbb{R}}
\newcommand{\CC}{\mathbb{C}}
\newcommand{\kk}{\Bbbk}
\newcommand{\PP}{\mathbb{P}}

\newcommand{\cO}{\mathcal{O}}
\newcommand{\cE}{\mathcal{E}}
\newcommand{\cF}{\mathcal{F}}
\newcommand{\cG}{\mathcal{G}}
\newcommand{\cK}{\mathcal{K}}
\newcommand{\cL}{\mathcal{L}}
\newcommand{\cM}{\mathcal{M}}
\newcommand{\cU}{\mathcal{U}}
\newcommand{\cV}{\mathcal{V}}
\newcommand{\cC}{\mathcal{C}}
\newcommand{\cW}{\mathcal{W}}

\newcommand{\cS}{\mathcal{S}}
\newcommand{\cQ}{\mathcal{Q}}

\newcommand{\RV}{\operatorname{RV}_{\kk}}
\newcommand{\RVC}{\operatorname{RV}_{\CC}}

\newcommand{\Supp}{\operatorname{Supp}}
\newcommand{\Newt}{\operatorname{Newt}}
\newcommand{\rk}{\operatorname{rk}}
\newcommand{\ctop}{c_{\mathrm{top}}}
\newcommand{\length}{\ell}
\newcommand{\Dcomp}{\mathfrak{D}}
\newcommand{\Nrm}{\mathrm{N}}
\newcommand{\Sch}{\mathfrak{S}}
\newcommand{\Groth}{\mathfrak{G}}

\newcommand{\Span}{\operatorname{span}}
\newcommand{\Sym}{\operatorname{Sym}}
\newcommand{\Proj}{\operatorname{Proj}}
\newcommand{\Fl}{\operatorname{Fl}}

\title[Chern flow and Chern moment algebras]{Chern flow and Chern moment algebras}

\author[K.-H. Nguyen-Dang]{Khai-Hoan Nguyen-Dang}
\address{
Morningside Center of Mathematics,
Chinese Academy of Sciences,
Beijing 100190, China
}
\email{khaihoann@gmail.com}

\author[Z. Wang]{Zhenpeng Wang}
\address{
Department of Mathematics,
The University of Hong Kong,
Hong Kong SAR, China
}
\email{u3011717@connect.hku.hk}



\hypersetup{
    pdftitle={
        Chern flow and Chern moment algebras
    },
    pdfauthor={
        Khai-Hoan Nguyen-Dang and Zhenpeng Wang
    },
    pdfsubject={
        Chern-flow geometry, Chern moment algebras,
        and Hodge structures for geometric and finite matroid sources
    },
    pdfkeywords={
        Chern flow,
        Bott--Samelson variety,
        Schubert polynomial,
        Grothendieck polynomial,
        Lascoux polynomial,
        Lorentzian polynomial,
        volume polynomial,
        Newton polytope
    }
}

\date{}

\newcommand{\cR}{\mathcal{R}}
\newcommand{\CH}{\operatorname{CH}}

\newcommand{\im}{\operatorname{im}}
\DeclareMathOperator{\Ann}{Ann}
\DeclareMathOperator{\Tor}{Tor}
\DeclareMathOperator{\coker}{coker}
\DeclareMathOperator{\cl}{cl}
\hypersetup{bookmarksopen=true,bookmarksopenlevel=1,bookmarksnumbered=true,bookmarksdepth=2}

\begin{document}

\begin{abstract}
We construct realizable-volume models over every field for the factorial
normalizations of homogeneous Lascoux, Lascoux-atom, and positive
Grothendieck packets, including their minimal homogenizations and layers.
In particular, the construction realizes the factorial normalizations of
all Schubert and key polynomials and of the minimal sign-corrected
homogeneous Grothendieck polynomials.
The normalized polynomials are Lorentzian, and the ordinary supports are
the lattice points of integral generalized polymatroids. On a
Bott--Samelson tower, row and co-row filtrations assemble the local factors
into globally generated bundles; a creation-state graph absorbs the
remaining kernel factors by Chern flow.

We also construct intrinsic algebras of joint Chern moments. Positive
inverse-Chern presentations give these algebras Hard Lefschetz and
Hodge--Riemann relations, and supply source-level Hodge completions of the
packets. For globally generated tropical toric bundles in the sense of
Kaveh--Manon, finite generating witnesses and matroid duality provide the
presentations required by Larson--Partida's theorem, without
representability. These constructions yield joint Chern-number
inequalities, nonvanishing polymatroids, and equality criteria.
\end{abstract}

\maketitle
\clearpage
\tableofcontents
\clearpage

\section{Introduction}\label{sec:introduction}
Schubert polynomials connect the geometry of flag varieties with the
combinatorics of permutations. They represent Schubert classes in
cohomology, while key polynomials, or type-$A$ Demazure characters, refine
Schur polynomials by retaining the action of a Borel subgroup
\cite{LS82,Demazure74,RS95}. Their $K$-theoretic counterparts include
Grothendieck and Lascoux polynomials; Demazure and Lascoux atoms give
further nonsymmetric refinements \cite{FK94,BSW20}. Combinatorial formulas describe their coefficients and, in the
$K$-theoretic cases, account for alternating signs between degrees.
A different question asks how the coefficients are related to one another. Do they satisfy log-concavity inequalities? Which exponent
vectors occur, and do they fill all lattice points of their convex hull?

Intersection theory suggests a common approach to these questions.
Mixed intersections of nef divisors satisfy Hodge-index inequalities,
and their nonvanishing is governed by exchange properties. Br\"and\'en
and Huh's theory of Lorentzian polynomials expresses both phenomena in
terms of a homogeneous polynomial \cite[Theorems~2.25 and~4.6]{BH20}.
This led Huh, Matherne, M\'esz\'aros, and St.~Dizier to conjecture
Lorentzianity for the factorial normalizations of Schubert and key
polynomials and for suitable homogeneous Grothendieck polynomials
\cite[Conjectures~15, 21--23]{HMMSD22}. Related conjectures of Monical,
Tokcan, and Yong concern saturated Newton polytopes for Demazure atoms
and the ordinary $K$-theoretic families \cite{MTY19}.

The geometric origin of a polynomial does not by itself give the
mixed-intersection formula needed here. In particular, representing a
Schubert class on a flag variety is different from realizing the
\emph{coefficients} of its polynomial representative as intersections
of semiample divisors on one integral variety. We construct such
realizations for complete homogeneous families, so that the resulting
inequalities compare different homogeneous components as well as
coefficients within a single component. We then study the algebra of
Chern moments underlying the construction and give a parallel Hodge
theory for globally generated tropical toric bundles.

\subsection{Volume realizations and the main polynomial theorem}
\label{subsec:main-results}\label{subsec:resolved-problems}
The normalization in the statement is dictated by the intersection
expansion. For $P=\sum_{|\alpha|=d}c_\alpha x^\alpha$, put
\[
 \Nrm(P)=\sum_{|\alpha|=d}c_\alpha x^{[\alpha]},
 \qquad x^{[\alpha]}=\frac{x^\alpha}{\alpha!}.
\]
If $D_1,\ldots,D_n$ are divisor classes on a projective $d$-fold $X$, then
\[
 \frac1{d!}\int_X\left(\sum_i x_iD_i\right)^d
 =\sum_{|\alpha|=d}\left(\int_XD^\alpha\right)x^{[\alpha]}.
\]
We write $\RV$ for nonnegative rational multiples of these polynomials
when $X$ is integral projective over $\kk$ and the $D_i$ are semiample
Cartier classes, meaning that some positive multiple is globally generated.
Membership means an actual realization, rather than a limit of realizations. The polynomial coefficients are in $\QQ$ even
when the geometric field has positive characteristic.

For an inhomogeneous positive-parameter polynomial
\[
 Q^+(x;\beta)=\sum_{k=0}^rQ_k(x)\beta^k,
\]
with each $Q_k$ homogeneous of degree $d+k$, we call
$\sum_kQ_k(x)z^{L-k}$, $L\ge r$, a \emph{homogeneous packet}.
The variable $z$ records the excess degree $k$ and places all components
in the same total degree. The least choice $L=r$ gives the minimal
packet; a larger $L$ is useful for keeping track of an operator word.

Fix $n\ge1$, and pad partitions by zeros to length $n$. Let $s_i$
interchange $x_i$ and $x_{i+1}$, and set
\[
 \partial_i f=\frac{f-s_if}{x_i-x_{i+1}},\qquad
 \pi_i f=\partial_i(x_if).
\]
The homogeneous operators used below are
\begin{align*}
 \mathscr K_i^{(z)}f&=\partial_i\bigl(x_i(x_{i+1}+z)f\bigr),&
 \overline{\mathscr K}_i^{(z)}&=\mathscr K_i^{(z)}-z\,\mathrm{id},\\
 \mathscr D_i^{(z)}f&=\partial_i\bigl((x_{i+1}+z)f\bigr).
\end{align*}
They satisfy the braid relations; in a product, the rightmost factor
acts first. With $w_0$ the longest permutation and
$\delta=(n-1,n-2,\ldots,0)$, define
\begin{equation}\label{eq:intro-packets}
 \begin{aligned}
 H_{w,\lambda}(x,z)&=\mathscr K_w^{(z)}x^\lambda,\\
 \overline H_{w,\lambda}(x,z)&=\overline{\mathscr K}_w^{(z)}x^\lambda,\\
 F_w(x,z)&=\mathscr D_{w^{-1}w_0}^{(z)}x^\delta.
 \end{aligned}
\end{equation}
These are the Lascoux, Lascoux-atom, and positive Grothendieck packets,
respectively. For the ordinary Grothendieck polynomial, write
\[
 \begin{gathered}
 G_w=\sum_{k\ge0}G_w^{(\length(w)+k)},\qquad
 A_{w,k}=(-1)^kG_w^{(\length(w)+k)},\\
 \widetilde G_w(x,z)=\sum_{k=0}^{r_w}A_{w,k}(x)z^{r_w-k},
 \end{gathered}
\]
where $r_w=\max\{k:A_{w,k}\ne0\}$ and the superscripts denote
homogeneous components. The sign convention makes $A_{w,k}$
coefficientwise nonnegative. All operator conventions are specified in
\cref{subsec:operator-conventions,sec:ordinary-packets}.

Our first main result realizes these three packets over the original
geometric field, with no restriction on its characteristic.
\begin{theorem}\label{thm:main}
For every field $\kk$, partition $\lambda$ with at most $n$ parts,
and permutation $w\in S_n$,
\[
 \Nrm_{x,z}(H_{w,\lambda}),\qquad
 \Nrm_{x,z}(\overline H_{w,\lambda}),\qquad
 \Nrm_{x,z}(F_w)\quad\text{belong to }\RV.
\]
The factorial normalizations of their nonzero minimal homogeneous
packets and all their homogeneous layers also belong to $\RV$.
In particular, this includes all key polynomials, Demazure atoms,
Schubert polynomials, the minimal sign-corrected packet
$\widetilde G_w$, and every sign-corrected component $A_{w,k}$.

Every nonzero normalized polynomial just listed is Lorentzian and has
support equal to the integral base set of a polymatroid algebraic over
$\kk$. In characteristic zero it admits a smooth integral projective
realization by semiample Cartier divisors.
The support of every ordinary Grothendieck, Lascoux, or Lascoux-atom
polynomial is $M^\natural$-convex and is the full lattice-point set of
its integral generalized-polymatroid Newton polytope.
\end{theorem}

The layer statements follow from the full-packet realization: for
$Q=\sum_kQ_k(x)z^{L-k}$,
\[
 \left.\partial_z^{L-k}\Nrm_{x,z}(Q)\right|_{z=0}=\Nrm_x(Q_k).
\]
Differentiation also removes common homogenizing padding. These operations
preserve actual realizability \cite[Theorem~1.3 and Remark~1.14]{GHMSSW25};
distinct excess degrees prevent cancellation in ordinary signed supports.

A small example shows why normalization is essential.
\begin{example}\label{ex:normalization-essential}
For $n=2$ and $\lambda=(2,0)$,
\begin{align*}
 H_{s_1,(2,0)}(x,y,z)&=z(x^2+xy+y^2)+xy(x+y),\\
 \Nrm(H_{s_1,(2,0)})&=\frac{z(x+y)^2+xy(x+y)}2.
\end{align*}
Differentiating the normalized packet in $z$ gives $(x+y)^2/2$,
whereas setting $z=0$ gives $xy(x+y)/2$. The unnormalized quadratic
layer $x^2+xy+y^2$ has Hessian eigenvalues $3$ and $1$ and is not
Lorentzian. Since it is the $z$-derivative of the raw packet, the raw
packet is not Lorentzian either. Normalization changes the coefficient
inequalities, but not the support.
\end{example}

\Cref{thm:main} proves the normalized assertions of
\cite[Conjectures~15, 21--23]{HMMSD22} and the support assertions of
\cite[Conjectures~3.14, 5.5--5.7]{MTY19}. For ordinary Grothendieck
polynomials it also gives the augmentation, coordinate-interval, and
generalized-polymatroid assertions in
\cite[Conjectures~1.1--1.4]{MSSD25}; see \cref{cor:ordinary-packet-support}.
Schubert and key supports were already known to be saturated and
generalized-permutahedral \cite{FinkMeszarosStDizier18}; the assertion
here concerns their coefficients and a simultaneous intersection model.
Related results treat Schubert complements \cite{FanGuoLiu24},
vexillary and fireworks Grothendieck supports
\cite{HafnerEtAl24,ChouSetiabrata25}, zero-one polynomials
\cite{ChenFanYe26}, and top-degree components
\cite{PanYu24,Yu25,Haf26}. Our support theorem includes the homogeneous
and Castelnuovo--Mumford support conjectures of \cite{MSD20};
special families of the maximal-degree case are treated in \cite{Haf26}. We do not address the
double-Schubert assertion following Conjecture~15 of \cite{HMMSD22},
the distinct double-Schubert support problem of \cite{CastilloEtAl23},
or the separate coefficient and schubitope questions in \cite{MSSD25}.

\subsection{The geometric construction}\label{subsec:geometric-strategy}
\label{subsec:canonical-row}
The divided-difference formulas on Bott--Samelson varieties are
classical \cite{Magyar98,Hudson14,Oetjen21}. Their local factors,
however, involve quotient and kernel lines that need not be globally
generated. Applying positivity separately at each step therefore does
not suffice. The construction groups these factors into globally
generated bundles before taking their top Chern classes.

The simplest instance of the kernel argument is a pair of compatible
quotients. Suppose that $\cF_j$ is generated of rank $j$, for
$j=1,2,3$, and that
\[
 0\longrightarrow\cK_2\longrightarrow\cF_2\longrightarrow\cF_1
   \longrightarrow0,\qquad
 0\longrightarrow\cK_3\longrightarrow\cF_3\longrightarrow\cF_2
   \longrightarrow0.
\]
The line bundles $\cK_2,\cK_3$ need not be generated, but Whitney's
formula gives
\begin{equation}\label{eq:intro-two-edge-flow}
 c_1(\cF_1)c_1(\cK_2)c_1(\cK_3)
 =c_2(\cF_2)c_1(\cK_3)=c_3(\cF_3).
\end{equation}
The initial top Chern factor allows the two kernel factors to be
absorbed into a generated bundle. No Chern class has been inverted.

For a general collection of such sequences, a directed edge
$e:v\to w$ represents
$0\to\cE_e\to\cF_w\to\cF_v\to0$.
If $\kappa_e$ counts its kernel factors and $r_v$ counts the supplied
vertex top Chern factors, define
\[
 b_v=r_v+\sum_{t(e)=v}\kappa_e-\sum_{s(e)=v}\kappa_e.
\]
When every $b_v$ is nonnegative, the product is the top Chern class of
$\bigoplus_v\cF_v^{\oplus b_v}$. This is the Chern-flow identity
of \cref{thm:chern-flow}. On a flag tower, vertices must retain both
prefix rank and creation time: equal ranks do not identify the
bundles or make their quotient maps composable.

The other factors are controlled by a particular reduced word. For
$u\in S_n$, the inverse-Lehmer word concatenates
$B_{n-1}\cdots B_1$, where
\[
 c_a=\#\{b>a:u^{-1}(a)>u^{-1}(b)\},\qquad
 B_a=(a,a+1,\ldots,a+c_a-1).
\]
Along each increasing row, the newly created quotient lines filter a
bundle that is an actual quotient of a generated split tail.
The old upper roots used for atoms filter a second generated bundle,
defined at the start of the row. These are the row and co-row
constructions of \cref{prop:row-bundle,prop:corow}. For the
Grothendieck packet, the staircase terminal weight balances the graph.
More generally, profiles satisfying the internal balance inequalities have
the explicit minimal terminal completion in
\cref{thm:weighted-balance}.

Reciprocal complementation in a finite exponent box, followed by variable
reversal, turns the desired packet into such a top-Chern pushforward
$\Xi(h)$ on $B=\prod_i\PP^{m_i}$. The incidence construction of
\cref{thm:incidence-dual} then realizes
\[
 \Dcomp_{\mathbf m}\Nrm(\Xi)
 =\sum_\gamma[h^\gamma]\Xi\;x^{[\mathbf m-\gamma]},
\]
which is the original factorial normalization, up to reversal.
The construction uses an integral evaluation-kernel projective bundle
and exact volume preservation, not an integral general zero locus.
Generating-kernel incidence also appears in
\cite[Proposition~6.2]{CidRuiz25}; here the row/co-row surjections and
creation-state balance identify the nonsymmetric coefficient arrays.
The related skew Schur constructions of \cite{NguyenDang26} are not proof
inputs. The example $w=2143$ in \cref{app:worked-example} follows the
packet through the tower, flow, and incidence realization.

\subsection{Chern moments and the source algebra}\label{subsec:main-hodge}
A volume model supplies numerical inequalities, but it need not be
recovered from the scalar coefficient polynomial. To retain the
intersection algebra, we work over $\CC$ with
\[
 A_X=\bigoplus_j\bigl(H^{2j}(X,\RR)\cap H^{j,j}(X,\CC)\bigr)
\]
for a smooth integral projective $d$-fold $X$.
For generated bundles $\cE_1,\ldots,\cE_m$, define the Chern-moment
functional and its quotient by
\[
 \begin{gathered}
 \Lambda_C(az^\alpha)=\int_Xa\prod_i c_{\alpha_i}(\cE_i)
       \quad(\deg a+|\alpha|=d),\\
 R_X(\cE_1,\ldots,\cE_m)=A_X[z_1,\ldots,z_m]/\Ann(\Lambda_C).
 \end{gathered}
\]
The functional is zero in other degrees, and its annihilator consists
of $p$ such that $\Lambda_C(pq)=0$ for every $q$.
The variables record Chern indices: their moments are not obtained by
replacing each $z_i$ with $c_1(\cE_i)$.

The resulting algebra has the full Lefschetz and Hodge--Riemann
structure, not only the degree-one signature that yields log-concavity.
\begin{theorem}\label{thm:main-hodge}
The algebra $R_X(\cE_1,\ldots,\cE_m)$ has Poincar\'e duality, Hard
Lefschetz, and Hodge--Riemann relations for the image of
\[
 \left\{L+\sum_i t_i z_i:L\in\mathcal K_X,\ t_i>0\right\},
\]
where $\mathcal K_X$ is the K\"ahler cone. If $r_i=\rk\cE_i$ and
$0\ne\theta=\prod_i c_{r_i}(\cE_i)$, the weighted quotient
$A_X/\Ann_{A_X}(\theta)$, with degree $[a]\mapsto\int_X\theta a$,
has the same properties for the image of $\mathcal K_X$, in formal
dimension $d-\sum_i r_i$. For the fixed coefficient algebra $A_X$ with its integration map,
the moment algebra, degree map, and labelled generators depend
only on the classes $c_a(\cE_i)\in A_X^a$, and not on the chosen
spaces of generating sections.
\end{theorem}
The proof is given in \cref{hd-thm:geometric}; intrinsicness follows from
\cref{hd-prop:intrinsic-PD}.

The proof uses the same evaluation kernels as the volume construction.
Their projective-bundle relations are the inverse Chern series, and
their tautological classes lie on the boundary of a Lefschetz cone.
Successive descent, in the form recalled in
\cite[\S4.3, Lemma~4.6]{AHL26}, identifies the resulting quotient
with the intrinsic moment algebra. The terminal one-variable inverse-Chern
quotient also appears in \cite[Proposition~3.4]{LP26}.
The contribution here is the joint marked construction, its geometric
certificates, and its compatibility with the packet sources.

This belongs to a wider study of Hodge--Riemann relations for
characteristic classes. Ross and Toma proved such relations for Schur
classes of ample vector bundles and obtained Chern-number
log-concavity \cite[Theorems~1.1 and~1.3]{RossToma23}.
Lu and Zheng develop higher-bidegree Hodge--Riemann polynomials under
specified cohomological vanishing hypotheses \cite{LuZheng25}.
Our statement instead concerns a duality quotient defined by all
Chern moments of generated bundles. It does not assert all-degree
Hodge--Riemann relations for every individual Chern class on the
original variety. For a canonical-row source,
adjoining ample divisor directions that span this quotient's degree-one
part produces a polynomial whose apolar algebra is the quotient itself; see
\cref{hd-cor:canonical-completion,cor:canonical-Hodge-channels}.
The distinction is necessary: \cref{ex:apolar-obstruction} gives an
actual volume polynomial whose smaller apolar algebra fails
Hodge--Riemann in middle degree.

\subsection{Finite matroid certificates and consequences}
\label{subsec:main-tropical}
The evaluation-kernel argument has a combinatorial counterpart.
For tropical toric bundles in the sense of Kaveh--Manon \cite{KM24},
there need not be an algebraic bundle from which to form an evaluation
sequence. Nevertheless, finitely many generating witnesses give a
weighted matroid presentation. Passing to the dual matroid produces
the inverse-Chern relations required by the projective-bundle theorem
of Larson and Partida \cite[Theorem~3.9]{LP26}.

Here is the resulting joint statement. The moment algebra is defined
as above, with the real Chow algebra of the smooth projective toric
variety as coefficient algebra.
\begin{theorem}\label{thm:main-tropical}
Let $\Sigma$ be a smooth projective fan of dimension $d$, let
$A=A^\bullet(X_\Sigma)_\RR$, and let
$\cE_1,\ldots,\cE_m$ be globally generated Kaveh--Manon bundles on
$\Sigma$. Their Chern moment algebra has Poincar\'e duality, Hard
Lefschetz, and Hodge--Riemann relations for the image of
\[
 \left\{L+\sum_i t_i z_i:L\in\operatorname{Amp}(X_\Sigma)_\RR,
                         \ t_i>0\right\}.
\]
For nef classes $H_1,\ldots,H_q\in A^1$, the polynomial
\[
 \sum_{|\alpha|+|\beta|=d}
 \deg_A\left(\prod_i c_{\alpha_i}(\cE_i)\prod_jH_j^{\beta_j}\right)
 \frac{x^\alpha y^\beta}{\alpha!\beta!}
\]
is Lorentzian, with zero allowed. In particular, for one bundle of
rank $r\le d$ and an ample class $H$, the nonnegative numbers
$a_i=\deg_A(c_i(\cE)H^{d-i})$ satisfy
$a_i^2\ge a_{i-1}a_{i+1}$ for $1\le i<r$.
No matroid representability hypothesis is required.
\end{theorem}
The proof is given in \cref{km-thm:joint-Hodge,km-thm:log-concavity}.

The one-bundle assertion is \cite[Conjecture~1.7]{KM24}.
The essential comparison is between two different descriptions of
positivity: a witness basis maximizes weight in the finite matroid,
so its complement minimizes weight in the dual. The corresponding
equivariant Chern series multiply to a product of global characters;
only after passing to ordinary Chow does this become the inverse-Chern
identity. Thus the construction uses neither representability nor an
asserted exact sequence of tropical bundles. In the nonrepresentable
case it produces a Hodge algebra, not a geometric volume model.

The geometric and tropical constructions also describe Chern
nonvanishing. For a set $S$ of bundle indices, put
$\rho_C(S)=\max\{j:c_j(\bigoplus_{i\in S}\cE_i)\ne0\}$.
We prove that $\rho_C$ is a polymatroid rank and that a product
$\prod_i c_{\alpha_i}(\cE_i)$ is nonzero exactly when
$\sum_{i\in S}\alpha_i\le\rho_C(S)$ for every $S$
(\cref{nv-thm:polymatroid}). In the geometric case the associated clone
matroid is algebraic (\cref{nv-thm:clone}).
The coefficient inequalities include comparisons between excess
layers, while their equality cases are expressed as vanishing in a
descended pairing (\cref{pr-cor:Chern-rigidity}).
The auxiliary functoriality and base-change results, together with minimal
polarized completions, are collected in
\cref{sec:moment-functoriality}.

\subsection{Organization}
Part~\ref{part:volume} proves the volume theorem independently of the
moment-algebra theory. \Cref{sec:prelim} gives the realization tools;
\cref{sec:canonical,sec:flow} constructs the generated bundles; and
\cref{sec:ordinary-packets} identifies the three packets and their common
consequences, ending with $2143$.
Part~\ref{part:hodge} constructs the moment algebras and source
completions, proves the finite-matroid comparison, and treats
nonvanishing and coefficient inequalities.
The appendices contain the local reciprocal calculations, weighted and
relative extensions, flow-cone refinements, and functorial algebra.

The actual volume realizations use the generated envelopes and
the exact preservation results of \cite{GHMSSW25}; the later
Hodge theory supplies additional structure on the source
algebras and is not used to prove Theorem~\ref{thm:main}.

\part{Volume realizations}\label{part:volume}

\section{Volume polynomials and Chern-class realizations}\label{sec:prelim}
\subsection{Divided powers and discrete convexity}
Coefficient polynomials and ordinary differential operators are over
$\QQ$ (over $\RR$ for signatures), independently of the geometric
field; normalization therefore does not invert factorials in positive
characteristic. Chow groups and operational Chow rings have rational
coefficients unless stated otherwise. A generated bundle means a globally generated
vector bundle.

For $\alpha\in\NN^n$, write $|\alpha|=\sum_i\alpha_i$ and
$x^{[\alpha]}=x^\alpha/\alpha!$. The operators $\Nrm_{x,z}$ and
$\Nrm_x$ normalize all displayed variables and only the $x$-variables,
respectively. Ordinary derivatives satisfy
\begin{equation}\label{eq:divided-derivative}
 \partial^\beta x^{[\alpha]}=
 \begin{cases}x^{[\alpha-\beta]},&\beta\le\alpha,\\0,&\text{otherwise}.
 \end{cases}
\end{equation}
Thus normalized differential contractions introduce no factorials.
For a cap $\mathbf m\in\NN^n$, define on
$\RR[x]_{\le\mathbf m}=\Span_\RR\{x^{[\alpha]}:0\le\alpha\le\mathbf m\}$
the involution
\begin{equation}\label{eq:finite-complement}
 \Dcomp_{\mathbf m}(x^{[\alpha]})=x^{[\mathbf m-\alpha]}.
\end{equation}
This is the divided-power form of Poincar\'e duality for the Chow ring
$\QQ[h_1,\ldots,h_n]/(h_i^{m_i+1})_i$ of $\prod_i\PP^{m_i}$.
For $f(x,z)$ supported in $0\le\alpha\le\mathbf m$, $0\le q\le r$, put
\begin{align}
 \mathcal R_{\mathbf m,r}f&=x^{\mathbf m}z^rf(x^{-1},z^{-1}),
 \label{eq:full-reciprocal}\\
 \Dcomp_{\mathbf m,r}(x^{[\alpha]}z^{[q]})
 &=x^{[\mathbf m-\alpha]}z^{[r-q]}.
 \label{eq:full-finite-dual}
\end{align}
Then
\begin{equation}\label{eq:finite-reciprocal-normalization}
 \Dcomp_{\mathbf m,r}\Nrm_{x,z}(f)=\Nrm_{x,z}(\mathcal R_{\mathbf m,r}f),
 \qquad
 \Dcomp_{\mathbf m,r}\Nrm_{x,z}(\mathcal R_{\mathbf m,r}f)=\Nrm_{x,z}(f).
\end{equation}
Omit $r,z$ when $z$ is absent. Enlarging the cap gives
\begin{equation}\label{eq:box-enlargement-translation}
 \begin{aligned}
 \Dcomp_{\mathbf m+\nu,r+s}\Nrm_{x,z}(f)
 &=T_{\nu,s}(\Dcomp_{\mathbf m,r}\Nrm_{x,z}(f)),\\
 T_{\nu,s}(x^{[\alpha]}z^{[q]})&=x^{[\alpha+\nu]}z^{[q+s]}.
 \end{aligned}
\end{equation}
Reciprocal calculations may use Laurent polynomials; the specified
finite boxes ensure polynomial source and target expressions.

We regard coefficient polynomials as real polynomials when discussing
Lorentzianity. A complex volume model implies this numerical property by
mixed Hodge--Riemann; the geometric field in a realization is recorded
separately.
A finite set $B\subseteq\NN^n$ of constant coordinate sum is
$M$-convex if, whenever $\alpha_i>\beta_i$ for $\alpha,\beta\in B$,
there is $j$ with $\alpha_j<\beta_j$ and both
$\alpha-e_i+e_j,\beta+e_i-e_j\in B$ \cite{Murota03}.

We use the closed class of Lorentzian polynomials, including its boundary cases.
\begin{definition}\label{def:lorentzian}
In degree $d\ge2$, a homogeneous polynomial with nonnegative
coefficients is \emph{Lorentzian} if it is a coefficientwise limit of
homogeneous polynomials with strictly positive coefficients whose
coordinate derivatives of order $d-2$ have nondegenerate Hessians of
signature $(1,n-1)$. In degrees zero and one every nonnegative-coefficient
polynomial is Lorentzian. Zero is included in every degree.
\end{definition}

The support--Hessian characterization of Br\"and\'en and Huh is particularly convenient for coefficient arrays.
\begin{proposition}\label{thm:BH-criterion}
A nonzero homogeneous nonnegative-coefficient polynomial $f$ of degree
$d$ is Lorentzian exactly when $\Supp(f)$ is $M$-convex and every
$\partial_x^\beta f$, $|\beta|=d-2$, has a Hessian with at most one
positive eigenvalue.
\end{proposition}
This is \cite[Theorem~2.25]{BH20}; the Hessian condition is vacuous
for $d\le1$. For $F=\sum_{|\alpha|=d}c_\alpha x^\alpha$ and
$f=\Nrm(F)$, these Hessians are
\begin{equation}\label{eq:coefficient-hessian}
 H_\beta(F)=(c_{\beta+e_i+e_j})_{1\le i,j\le n},
\end{equation}
with out-of-range coefficients zero.

We also need the corresponding language for supports of varying total degree.
Write $\Newt(f)=\operatorname{conv}(\Supp(f))$; $f$ has a saturated
Newton polytope (SNP) when $\Supp(f)=\Newt(f)\cap\ZZ^n$.
For finite nonempty $S\subseteq\NN^n$ and
$D\ge\max_{\alpha\in S}|\alpha|$, set
\begin{equation}\label{eq:slack-homogenization}
 \widehat S_D=\{(\alpha,D-|\alpha|):\alpha\in S\}.
\end{equation}
The set $S$ is $M^\natural$-convex if $\widehat S_D$ is $M$-convex.
Increasing $D$ merely translates the slack coordinate, so the choice
does not matter. Equivalently, $S$ is the lattice-point set of an
integral generalized polymatroid \cite{Murota03}.

Slack homogenization also controls signed specializations of packets.
\begin{lemma}\label{lem:graded-packet}
Let $P_k\in\RR[x]_{d+k}$, $0\le k\le r\le L$, and suppose
$Q=\sum_kP_k(x)z^{L-k}\ne0$ has $M$-convex support. Then
$S=\bigcup_k\Supp(P_k)$ is $M^\natural$-convex and
\begin{equation}\label{eq:SNP-dehomogenization}
 S=\operatorname{conv}(S)\cap\ZZ^n.
\end{equation}
Moreover, $\Supp(\sum_ka_kP_k)=S$ if $a_k\ne0$ on every nonzero
layer, and $\sum_kP_k(x)\beta^k$ has saturated Newton polytope in
$\ZZ^{n+1}$.
\end{lemma}
\begin{proof}
The support of $Q$ is the slack homogenization of $S$ in total degree
$d+L$. The affine lattice embedding $a\mapsto(a,d+L-|a|)$ identifies
the two convex hulls. Since an $M$-convex set contains every lattice
point of its base polytope \cite{Murota03}, $S$ is saturated.
Distinct $x$-degrees preclude cancellation in the signed sum.
The affine unimodular map $(a,k)\mapsto(a,L-k)$ identifies the
parameter support with $\Supp(Q)$.
\end{proof}

\subsection{Realizable volume polynomials}
We distinguish actual intersection realizations from their coefficientwise limits.
\begin{definition}\label{def:RV}
A homogeneous $f\in\QQ[x_1,\ldots,x_n]_d$ belongs to $\RV$ if
\begin{equation}\label{eq:RV-definition}
 f(x)=\frac q{d!}\int_X\left(\sum_ix_iD_i\right)^d
\end{equation}
for $q\in\QQ_{\ge0}$, an integral projective $d$-fold $X/\kk$, and
semiample Cartier classes $D_i$. The choice $q=0$ includes zero.
\end{definition}
For $\kk=\CC$ write $\RVC$. Its polynomials are Lorentzian by mixed
Hodge--Riemann for nef classes \cite[Theorem~4.6]{BH20}.

The exact preservation results of \cite{GHMSSW25} supply the operations
used below.
\begin{proposition}\label{prop:RV-closure}
Over every field $\kk$, variable permutations, nonnegative rational
linear substitutions, ordinary coordinate derivatives, and full
polarization preserve $\RV$. Common divided-power translation is
reversible:
\[
 \sum_\alpha c_\alpha x^{[\alpha]}\in\RV
 \quad\Longleftrightarrow\quad
 \sum_\alpha c_\alpha x^{[\alpha+\gamma]}\in\RV.
\]
In particular, for $P_k\in\QQ[x]_{d+k}$ and $0\le r\le L$, put
$Q=\sum_{k=0}^rP_kz^{L-k}$ and
$\widetilde P=\sum_{k=0}^rP_kz^{r-k}$. Then
\begin{align}
 \partial_z^{L-r}\Nrm_{x,z}(Q)
   &=\Nrm_{x,z}(\widetilde P),\label{eq:common-minimal-packet}\\
 \left.\partial_z^{L-k}\Nrm_{x,z}(Q)\right|_{z=0}
   &=\Nrm_x(P_k).\label{eq:common-layer-extraction}
\end{align}
Thus actual realizability of $\Nrm_{x,z}(Q)$ passes to the minimal
packet, every layer, and every larger homogenizing padding.
In characteristic zero every nonzero member of $\RV$ has a smooth
integral projective realization with semiample Cartier classes.
\end{proposition}
\begin{proof}
Substitutions replace divisors by nonnegative rational combinations;
clear denominators using the scalar allowed in \cref{def:RV}.
Derivatives are covered by \cite[Theorem~1.3]{GHMSSW25}; forward
translation is the cone construction after its Definition~1.2, and
ordinary differentiation reverses it. Polarization is its
Proposition~4.1, with the arbitrary-field formulations in Remark~1.14.
The extraction identities follow termwise from
\eqref{eq:divided-derivative}.
Finally, pull a realization back along a projective resolution
\cite{Hironaka64}. Semiample classes remain semiample, and the
projection formula preserves every mixed intersection.
\end{proof}

\label{subsec:algebraic-support}
An integral polymatroid rank $\rho$ on $[n]$ is algebraic over $\kk$
if fields $\kk\subseteq K_i\subseteq L$ satisfy
$\rho(S)=\operatorname{trdeg}_{\kk}K_S$, where $K_S$ is their
compositum and $K_\varnothing=\kk$. Its integral base set is
\[
 B(\rho)=\{\alpha\in\NN^n:\alpha(S)\le\rho(S)\ (S\subseteq[n]),
                       \ |\alpha|=\rho([n])\}.
\]
\begin{proposition}\label{thm:algebraic-polymatroid-consequence}
If $F\ne0$ is homogeneous and $\Nrm(F)\in\RV$, then
$\Supp(F)=B(\rho)$ for an integral polymatroid algebraic over $\kk$.
After full polarization in disjoint blocks $E_i$ dominating the
$x_i$-exponents, the support is the base set of an algebraic matroid:
\[
 \widetilde B=\{S\subseteq\bigsqcup_iE_i:
       (|S\cap E_1|,\ldots,|S\cap E_n|)\in\Supp(F)\}.
\]
\end{proposition}
\begin{proof}
Normalization preserves support, so the first assertion is
\cite[Proposition~5.4]{GHMSSW25}. Polarization preserves actual volume
and gives the displayed multiaffine support. Apply the same result;
the resulting polymatroid has singleton ranks at most one, hence is
a matroid.
\end{proof}
A linear representation over $F$ consists of subspaces $U_i$ with
$\rho(S)=\dim_F\sum_{i\in S}U_i$. Fields generated by linear forms give
an algebraic representation. The distinction between $F$ and its
extensions is retained in \cref{prop:weighted-support-charzero}.

\subsection{Operator and intersection-theoretic conventions}\label{subsec:operator-conventions}
Let $s_i$ interchange $x_i,x_{i+1}$, and set
$\partial_i f=(f-s_if)/(x_i-x_{i+1})$, $\pi_i f=\partial_i(x_if)$.
The local homogeneous operators are
\begin{align}
 \Gamma_i^{(\epsilon,z)}f&=\partial_i(x_{i+1}^\epsilon(x_i+z)f),
 &&\epsilon\in\NN,\label{eq:operator-dictionary-gamma}\\
 \Pi_i^{(z)}f&=\Gamma_i^{(0,z)}f=\partial_i((x_i+z)f),
 \label{eq:operator-dictionary-pi}\\
 \mathscr K_i^{(z)}f&=\partial_i(x_i(x_{i+1}+z)f),&
 \mathscr D_i^{(z)}f&=\partial_i((x_{i+1}+z)f).
 \label{eq:operator-dictionary-source}
\end{align}
Here $\mathscr K,\mathscr D$ are the Lascoux and positive Grothendieck
source forms, and $\Pi,\Gamma^{(1,z)}$ their no-kernel and all-kernel
reciprocal forms. Word identities use the common finite boxes in
\cref{sec:key,app:operator-conjugation}.
\begin{remark}\label{warn:derivative-notation}
In operator sections $\partial_i$ is a divided difference; ordinary
derivatives are $\partial_{x_i}$, $\partial_z$, or $\partial_x^\gamma$.
The ordinary differential-preserver results of \cite{GHMSSW25} do not
assert volume preservation by Demazure or opposite-Demazure operators.
\end{remark}
For a word $\mathbf i=(i_1,\ldots,i_\ell)$, set
\begin{equation}\label{eq:operator-order}
 T_{\mathbf i}=T_{i_1}\cdots T_{i_\ell}\quad\text{in }\operatorname{End}(\QQ[x,z]).
\end{equation}
The rightmost factor acts first, matching pushforward from the top of
the tower. With $\delta=(n-1,n-2,\ldots,0)$,
$\Sch_w=\partial_{w^{-1}w_0}x^\delta$; equivalently
$\Sch_{w_0}=x^\delta$ and $\Sch_{ws_i}=\partial_i\Sch_w$ when
$w(i)>w(i+1)$.

For a weak composition $\alpha$, let $\alpha^+$ be decreasing and
$w_\alpha$ the minimal-length place permutation taking $\alpha^+$ to
$\alpha$ (preserving the order of equal parts). Then
$\kappa_\alpha=\pi_{w_\alpha}x^{\alpha^+}$, independently of the
reduced word by braid relations \cite{Demazure74,RS95}.
The positive Grothendieck convention is
$\Groth_w^+(x;\beta)=\sum_{k\ge0}A_{w,k}(x)\beta^k$,
$A_{w,k}=(-1)^kG_w^{(\length(w)+k)}$. Its recursion uses
$\mathscr D_i^{(z)}$ from \eqref{eq:operator-dictionary-source}, as
recalled in \cref{sec:groth}.

We now fix the projective-bundle convention used in the geometric constructions.
We use quotient projectivization $\PP_X(\cE)=\Proj_X(\Sym^\bullet\cE)$,
with $p^*\cE\twoheadrightarrow\cO(1)$. For $\rk\cE=r$ and
$\xi=c_1(\cO(1))$,
\begin{equation}\label{eq:projective-bundle-push}
 p_*(\xi^{r-1+m})=h_m(\cE)\quad(m\ge0),\qquad
 \sum_{m\ge0}h_m(\cE)t^m=c_{-t}(\cE)^{-1}.
\end{equation}
Write $\int_X\theta=\deg(\theta\cap[X])$ for top-degree operational
classes, and let $K^0(X)$ denote the Grothendieck group of vector
bundles; see \cite[\S\S3.2--3.3]{Fulton98}.

For a proper $\rho:X\to B$ to a finite projective-space box, the
notation $\rho_*\ctop(\cV)$ means the unique $\Xi\in A^c(B)$ with
$\Xi\cap[B]=\rho_*(\ctop(\cV)\cap[X])$, via the smooth-target
cap-product isomorphism. Singular sources are read cycle-theoretically;
smooth quotient-flag towers use their Chow rings. Over a singular base,
projective-bundle formulas are identities on Chow homology, with
operational classes acting by cap product.

\subsection{Incidence and total Chern classes}\label{sec:incidence}

Fix an arbitrary geometric field $\kk$ and put
\[
 B=\prod_{i=1}^N\PP^{m_i}_{\kk},\qquad
 h_i=c_1(\cO_{\PP^{m_i}}(1)),\qquad \mathbf m=(m_1,\ldots,m_N).
\]
The finite dual of a top-Chern pushforward will give the desired volume
polynomial. Since $B$ is smooth, each pushed-forward class has a unique
reduced representative in the Chow box $0\le\gamma\le\mathbf m$.

\begin{proposition}\label{thm:incidence-dual}
Let $\rho:X\to B$ be a projective morphism from an integral projective
variety of dimension $d_X$.  Let $\cV$ be a globally generated vector bundle
of rank $R$ on $X$, and write
\begin{equation}\label{eq:Xi-def}
 \Xi(h):=\rho_*c_R(\cV)
 =\sum_{\gamma\le\mathbf m}\xi_\gamma h^\gamma
 \in A^c(B),
 \qquad c=\dim B-d_X+R.
\end{equation}
If $\Xi=0$, the conclusion below is understood to be the zero polynomial.
Otherwise $0\le c\le\dim B$, and
\begin{equation}\label{eq:incidence-conclusion}
 \Dcomp_{\mathbf m}\Nrm(\Xi)
 =\sum_{\gamma\le\mathbf m}\xi_\gamma x^{[\mathbf m-\gamma]}
 \in\RV.
\end{equation}
If $X$ is smooth, the incidence variety used in the proof is smooth.
If $\operatorname{char}(\kk)=0$ and $\Xi\ne0$, the polynomial in
\eqref{eq:incidence-conclusion} also admits a smooth integral projective
realization.
\end{proposition}

\begin{proof}
The case $\Xi=0$ is immediate. Otherwise choose a generating sequence
\begin{equation}\label{eq:global-generation-sequence}
 0\longrightarrow\cK\longrightarrow W\otimes\cO_X
 \longrightarrow\cV\longrightarrow0,
\end{equation}
adding zero generators so that $s=\rk\cK>0$. Put
$M=\dim W-1=s+R-1$, $p:Z=\PP_X(\cK^\vee)\to X$, and
$H=c_1(\cO_Z(1))$. The dual generating sequence makes
$\cK^\vee$ globally generated, hence $H$ semiample. The projective
bundle $Z$ is integral and projective, and smooth when $X$ is smooth.
Its evaluation embedding in $X\times\PP(W^\vee)$ identifies it with
the universal zero incidence. Moreover,
\begin{equation}\label{eq:incidence-dimension}
 \dim Z=d_X+M-R,\qquad \dim Z-M=d_X-R=\dim B-c.
\end{equation}
Since $\Xi\ne0$, one has $0\le c\le\dim B$.

The identity $c_{-t}(\cK^\vee)c_t(\cV)=1$ and the projective-bundle
formula \eqref{eq:projective-bundle-push} give directly
\[
 p_*(H^M)=p_*(H^{s-1+R})=h_R(\cK^\vee)=c_R(\cV).
\]
For $\widetilde\rho=\rho p$ and $|\alpha|=d_X-R$, the projection
formula therefore yields
\begin{equation}\label{eq:incidence-coeff}
 \int_Z\widetilde\rho^*(h^\alpha)H^M
 =\int_Xc_R(\cV)\rho^*(h^\alpha)
 =\int_B\Xi(h)h^\alpha=\xi_{\mathbf m-\alpha},
\end{equation}
where coefficients outside the Chow box are zero. All displayed
divisors are semiample, so
\[
 V_Z(t,y)=\frac1{(\dim Z)!}
 \int_Z\left(\sum_i t_i\widetilde\rho^*h_i+yH\right)^{\dim Z}
 \in\RV.
\]
Consequently,
\begin{equation}\label{eq:incidence-volume-extraction}
 \left.\partial_y^M V_Z(t,y)\right|_{y=0}
 =\sum_{|\alpha|=d_X-R}\xi_{\mathbf m-\alpha}t^{[\alpha]}
 =\Dcomp_{\mathbf m}\Nrm(\Xi).
\end{equation}
Derivative and restriction closure \cite[Theorem~1.3 and Remark~1.14]{GHMSSW25}
prove actual realizability over $\kk$. Apply \cref{prop:RV-closure} for
the characteristic-zero smooth realization. The argument uses only
projective-bundle pushforward and the projection formula, so it also
applies cycle-theoretically when $X$ is singular.
\end{proof}

\begin{remark}\label{rem:zero-locus}
A general generated-bundle section need not have an integral zero
scheme: a general section of $\cO_{\PP^1}(2)$ gives two reduced points
over an algebraically closed field, though it may be integral over a
nonclosed field. We instead use the integral universal incidence bundle
$\PP_X(\cK^\vee)$ and extract a volume minor in the required degree.
\end{remark}

The evaluation-kernel construction also applies to products of
arbitrary Chern classes of globally generated bundles. The same
projective-bundle mechanism underlies the normalized Lorentzian
statement in \cite[Proposition~6.2]{CidRuiz25}. We record the
realizable-volume form needed here: combined with the exact
preservation theorem of \cite{GHMSSW25}, it gives an actual
realization over the original field.

\begin{proposition}\label{thm:total-chern-arbitrary-field}
Let $X$ be an integral projective variety of dimension $d$ over the field $\kk$, and let $\cE_1,\ldots,\cE_n$ be globally generated vector bundles on $X$, of ranks $r_1,\ldots,r_n$.  Define
\[
 T_{X,\cE_\bullet}(x_1,\ldots,x_n)
 =\sum_{\substack{\alpha\in\mathbb Z_{\geq0}^n\\|\alpha|=d}}
 \deg\!\left(\prod_{i=1}^n c_{\alpha_i}(\cE_i)\cap[X]\right)x^\alpha.
\]
If it is nonzero, $\Nrm(T_{X,\cE_\bullet})$ is a realizable-volume
polynomial over $\kk$, with a smooth integral projective realization
by semiample Cartier divisors in characteristic zero.
\end{proposition}

\begin{proof}
Choose generating sequences
$0\to \cK_i\to W_i\otimes\cO_X\to \cE_i\to0$, adding zero generators so
that $s_i=\rk \cK_i>0$. Let $p_i:\PP_X(\cK_i^\vee)\to X$, put
$b_i=s_i-1$, $b=(b_1,\ldots,b_n)$, and set
\[
 Y=\prod_{i,X}\PP_X(\cK_i^\vee),\qquad D=d+\sum_i b_i.
\]
Write $\xi_i$ for the tautological first Chern class on the $i$th
factor and for its pullback to $Y$. Dualizing the generating sequence
gives $W_i^\vee\otimes\cO_X\twoheadrightarrow \cK_i^\vee$; after
pullback, its composite with the tautological quotient generates
$\cO_{\PP_X(\cK_i^\vee)}(1)$. Thus the $\xi_i$ are globally generated
and $Y$, being an iterated projective bundle over $X$, is integral
and projective.

The projective-bundle formula \eqref{eq:projective-bundle-push} gives
$(p_i)_*\xi_i^{b_i+a}=h_a(\cK_i^\vee)=c_a(\cE_i)\quad(a\ge0)$,
whereas lower powers push to zero. The second equality follows from
$c_{-t}(\cK_i^\vee)c_t(\cE_i)=1$.
Proper pushforward commutes with flat pullback along the other
projective-bundle factors \cite[Tag~02RG]{Stacks}. Iterating this
identity and the projection formula therefore yields
\begin{equation}\label{pI-eq:total-chern-shifted-volume}
 \frac1{D!}\int_Y\left(\sum_i x_i\xi_i\right)^D
 =\Nrm\!\left(x^bT_{X,\cE_\bullet}(x)\right).
\end{equation}
Apply $\partial_x^b$ and \eqref{eq:divided-derivative}.
The differential-preserver theorem
\cite[Theorem~1.3 and Remark~1.14]{GHMSSW25} gives the desired
realization over $\kk$. In characteristic zero use
\cref{prop:RV-closure}. All calculations take place on Chow homology,
so singular integral $X$ and the zero polynomial are included.
\end{proof}

The packet constructions below supply particular generated bundles to
which this proposition applies. When $X$ is a complex variety, the
resulting normalized total-Chern polynomial is Lorentzian by its complex
volume realization.

\section{Canonical-row towers}\label{sec:canonical}
The local Gysin formulas are classical; the positivity argument depends
on grouping their factors before applying incidence. We construct the
tower, fix its pushforward order, and prove generation for the two
rowwise bundles that will be used by the three packet families.

\subsection{The quotient-flag tower}

Let $u\in S_n$.  For $1\le a<n$, define
\begin{equation}\label{eq:inverse-lehmer}
 c_a=c_a(u)
 :=\#\{b>a:u^{-1}(a)>u^{-1}(b)\},
\end{equation}
and let
\begin{equation}\label{eq:canonical-block}
 B_a=(a,a+1,\ldots,a+c_a-1).
\end{equation}
Empty blocks are omitted.  The canonical-row word is
\begin{equation}\label{eq:canonical-word}
 \mathbf r(u)=B_{n-1}B_{n-2}\cdots B_1.
\end{equation}

The row order is chosen to give a reduced expression with the required unchanged-prefix property.
\begin{lemma}\label{lem:canonical-reduced}
The word $\mathbf r(u)$ is a reduced expression for $u$.
\end{lemma}

\begin{proof}
Use one-line notation and right multiplication, so $s_i$ exchanges
positions $i,i+1$. Induct downwards on $a$. Before row $a$, the larger
values occupy positions $a+1,\ldots,n$ in their relative order in $u$,
and $a$ remains in position $a$. Of those larger values, exactly the
first $c_a$ precede $a$ in $u$. Multiplication by
$s_a\cdots s_{a+c_a-1}$ moves $a$ across that initial segment without
changing their relative order. This proves the induction and produces
$u$. The length is $\sum_ac_a=\length(u)$, the inversion number of
$u^{-1}$, so the word is reduced.
\end{proof}

For example, $c_2=c_1=2$ for $u=3412$, and
$\mathbf r(3412)=(2,3\mid1,2)$. The bar separates the rows and has no
algebraic effect. Its creation-state graph is displayed in
\cref{ex:creation-state-3412}.

The tower is built over a product of projective spaces whose dimensions
record the exponent bounds. Fix nonnegative integers
$m_1,\ldots,m_n,m_z$, and put
\begin{equation}\label{eq:base-B}
 B=\prod_{i=1}^n\PP^{m_i}\times\PP^{m_z}.
\end{equation}
Let
$\cL_i=\operatorname{pr}_i^*\cO_{\PP^{m_i}}(1), \qquad \cM=\operatorname{pr}_z^*\cO_{\PP^{m_z}}(1)$,
and write
\[
 h_i=c_1(\cL_i),
 \qquad h_z=c_1(\cM),
 \qquad
 \cE=\bigoplus_{i=1}^n\cL_i.
\]
The initial quotient flag is
\begin{equation}\label{eq:initial-quotient-flag}
 \cF_j^{(0)}=\cL_1\oplus\cdots\oplus\cL_j,
 \qquad
 \cF_{j+1}^{(0)}\twoheadrightarrow \cF_j^{(0)}.
\end{equation}
We set $\cF_0^{(0)}=0$ and $\cF_n^{(0)}=\cE$.  At every later stage the same
notation denotes a genuine quotient flag
$\cF_n\twoheadrightarrow\cdots\twoheadrightarrow \cF_0$.

Write $\mathbf r(u)=(i_1,\ldots,i_\ell)$.  Starting with $X_0=B$, construct inductively $X_t\to X_{t-1}$ as follows.  Pull all previously defined bundles to $X_{t-1}$ and put
\begin{equation}\label{eq:Kt-def}
 \cK_t=\ker\bigl(\cF_{i_t+1}^{(t-1)}\longrightarrow \cF_{i_t-1}^{(t-1)}\bigr).
\end{equation}
The map in \eqref{eq:Kt-def} is surjective between vector bundles of ranks $i_t+1$ and $i_t-1$, so $\cK_t$ is a rank-two vector bundle.  Let
\begin{equation}\label{eq:Ptower}
 \pi_t:X_t=\PP_{X_{t-1}}(\cK_t)\longrightarrow X_{t-1}
\end{equation}
with the quotient convention, and write its tautological sequence as
\begin{equation}\label{eq:tautological-SQ}
 0\longrightarrow \cS_t\longrightarrow\pi_t^*\cK_t
 \longrightarrow \cQ_t\longrightarrow0.
\end{equation}
Define the new rank-$i_t$ quotient by
\begin{equation}\label{eq:new-Fi}
 \cF_{i_t}^{(t)}=\pi_t^*\cF_{i_t+1}^{(t-1)}/\cS_t,
\end{equation}
and keep all other prefix bundles unchanged after pullback.  Since $\cS_t$ is
a line subbundle of $\pi_t^*\cK_t$, the quotient in \eqref{eq:new-Fi} is
locally free.  The local modification is summarized by the commutative
diagram with exact rows
\[
\begin{tikzcd}[column sep=3.1em]
0 \arrow[r]
 & \pi_t^*\cK_t \arrow[r,hook] \arrow[d,two heads]
 & \pi_t^*\cF_{i_t+1}^{(t-1)} \arrow[r,two heads]
     \arrow[d,two heads]
 & \pi_t^*\cF_{i_t-1}^{(t-1)} \arrow[r] \arrow[d,equal]
 & 0 \\
0 \arrow[r]
 & \cQ_t \arrow[r,hook]
 & \cF_{i_t}^{(t)} \arrow[r,two heads]
 & \pi_t^*\cF_{i_t-1}^{(t-1)} \arrow[r]
 & 0.
\end{tikzcd}
\]
The two nontrivial vertical maps are the quotients by $\cS_t$.  Thus the maps to and
from the unchanged neighbouring quotients make the new collection a quotient
flag, and there are exact sequences
\begin{equation}\label{eq:S-exact}
 0\longrightarrow \cS_t\longrightarrow \pi_t^*\cF_{i_t+1}^{(t-1)}
 \longrightarrow \cF_{i_t}^{(t)}\longrightarrow0
\end{equation}
and
\begin{equation}\label{eq:Q-exact}
 0\longrightarrow \cQ_t\longrightarrow \cF_{i_t}^{(t)}
 \longrightarrow \pi_t^*\cF_{i_t-1}^{(t-1)}\longrightarrow0.
\end{equation}
Every time-indexed $\cF_j^{(t)}$ is a quotient of the pullback of $\cE$,
hence is globally generated.  The final tower
$\rho:X_u=X_\ell\longrightarrow B$
is a smooth irreducible projective variety.

\begin{remark}
\label{rem:relative-bott-samelson}
Let $\Fl_B(\cE)$ be the relative full quotient-flag bundle, and let
$\Fl_{B,i}(\cE)$ be the partial flag bundle obtained by omitting the
rank-$i$ quotient.  The split flag \eqref{eq:initial-quotient-flag} gives a section
$f_0:B\to\Fl_B(\cE)$.  Inductively, if
$f_{t-1}:X_{t-1}\to\Fl_B(\cE)$ is the current quotient flag and
$i=i_t$, then $X_t$ fits into the Cartesian square
\[
\begin{tikzcd}[column sep=4.5em,row sep=2.4em]
 X_t \arrow[r,"f_t"] \arrow[d,"\pi_t"']
   & \Fl_B(\cE) \arrow[d,"\varpi_i"] \\
 X_{t-1} \arrow[r,"\varpi_i\circ f_{t-1}"']
   & \Fl_{B,i}(\cE),
\end{tikzcd}
\]
where $\varpi_i$ forgets rank $i$. Its fibre parametrizes the
intermediate quotients by $\PP(\ker(\cF_{i+1}\to \cF_{i-1}))$, with
universal quotient \eqref{eq:new-Fi}. Thus each geometric fibre of
$X_u\to B$ is the usual Bott--Samelson variety \cite{Magyar98}.
We retain the actual extensions \eqref{eq:old-ranktwo-extension},
without splitting them or replacing the tower by a toric Bott model
as in \cite{JeongKimLee25}. This interpretation is not a proof input.
\end{remark}

Put
$s_t=c_1(\cS_t), \qquad q_t=c_1(\cQ_t)$.

For each stage $t$ and $1\le j\le n$, let
\begin{equation}\label{eq:successive-line-bundles}
 \cR_j^{(t)}:=\ker\bigl(\cF_j^{(t)}\longrightarrow \cF_{j-1}^{(t)}\bigr),
 \qquad r_j^{(t)}:=c_1(\cR_j^{(t)}).
\end{equation}
These are actual line bundles, not merely formal roots.  If the $t$th
modification has colour $i=i_t$, then \eqref{eq:S-exact} and
\eqref{eq:Q-exact} identify
\begin{equation}\label{eq:line-root-update}
 \cR_i^{(t)}=\cQ_t,
 \qquad \cR_{i+1}^{(t)}=\cS_t,
 \qquad \cR_j^{(t)}=\pi_t^*\cR_j^{(t-1)}\quad(j\ne i,i+1).
\end{equation}
Before the modification, the rank-two bundle $\cK_t$ is itself an extension of
the two successive old quotient lines:
\begin{equation}\label{eq:old-ranktwo-extension}
 0\longrightarrow \cR_{i+1}^{(t-1)}\longrightarrow \cK_t
 \longrightarrow \cR_i^{(t-1)}\longrightarrow0.
\end{equation}
Indeed, the map $\cK_t\to \cR_i^{(t-1)}$ is induced by
$\cF_{i+1}^{(t-1)}\twoheadrightarrow \cF_i^{(t-1)}$; its kernel is
$\cR_{i+1}^{(t-1)}$, and local lifting proves surjectivity.  Thus the two
labelled Chern roots of $\cK_t$ on $X_{t-1}$ are $r_i^{(t-1)}$ and
$r_{i+1}^{(t-1)}$; after pullback to $X_t$, the tautological modification
replaces the ordered pair by $(q_t,s_t)$.

For $0\le a\le b\le\ell$, write
\begin{equation}\label{eq:tower-projections}
 p_{b,a}=\pi_{a+1}\circ\cdots\circ\pi_b:X_b\longrightarrow X_a,
 \qquad p_{a,a}=\mathrm{id}_{X_a}.
\end{equation}
A class defined on $X_a$ is pulled to $X_b$ by $p_{b,a}^*$; pullback
symbols are suppressed only when no stage ambiguity can arise.

\subsection{Gysin formulas and operator words}

We use $\Gamma_i^{(\epsilon,z)}$ from
\eqref{eq:operator-dictionary-gamma}, with $\epsilon\in\NN$ and
$z=h_z$ after evaluation.

The rank-two projective-bundle formula gives both the local identity and
its time-indexed form on the tower.
\begin{lemma}\label{lem:gysin}
Let $\cK$ have rank two and formal Chern roots $a,b$, and write
$0\to\cS\to\pi^*\cK\to\cQ\to0$ on $\PP(\cK)$ in the quotient
convention, with $s=c_1(\cS)$ and $q=c_1(\cQ)$. Then
\begin{equation}\label{eq:ranktwo-localization}
 \pi_*\Psi(q,s)=\frac{\Psi(a,b)-\Psi(b,a)}{a-b}.
\end{equation}
In particular, for every $\epsilon\in\NN$,
\begin{equation}\label{eq:gysin-identity}
 \pi_*\bigl(s^\epsilon(q+h_z)\Phi(q,s)\bigr)
 =\left.\partial_i\bigl(x_{i+1}^\epsilon(x_i+h_z)
              \Phi(x_i,x_{i+1})\bigr)\right|_{(x_i,x_{i+1})=(a,b)}.
\end{equation}
On the quotient-flag tower, writing
$\Phi(r^{(t)},h_z)=\Phi(r_1^{(t)},\ldots,r_n^{(t)},h_z)$, this becomes
\begin{equation}\label{eq:one-step-intertwining}
 (\pi_t)_*\!\left[s_t^\epsilon(q_t+h_z)\Phi(r^{(t)},h_z)\right]
 =\bigl(\Gamma_{i_t}^{(\epsilon,z)}\Phi\bigr)(r^{(t-1)},h_z).
\end{equation}
\end{lemma}
\begin{proof}
Since $s=\pi^*c_1(\cK)-q$, it suffices to check polynomials in $q$.
The projective-bundle formula gives $\pi_*1=0$ and, for $m\ge1$,
\[
 \pi_*q^m=h_{m-1}(\cK)
       =\sum_{j=0}^{m-1}a^{m-1-j}b^j=\frac{a^m-b^m}{a-b}.
\]
This proves \eqref{eq:ranktwo-localization}; its symmetric right-hand
side is a polynomial in $c_1(\cK),c_2(\cK)$, so no splitting is
assumed. Substitution proves \eqref{eq:gysin-identity}.
For \eqref{eq:one-step-intertwining}, the updated root pair is
$(q_t,s_t)$ by \eqref{eq:line-root-update}, the old pair consists of
the two roots of $\cK_t$, and all other variables are pulled back.
\end{proof}

The following projection-formula argument will be used for quotient,
weighted, and atom factors. It fixes the word order once for all three
constructions. Write
$\operatorname{ev}_t(\Phi)=\Phi(r^{(t)},h_z)$.

\begin{lemma}\label{lem:time-indexed-pushforward}
On the quotient-flag tower, let $A_t\in A^*(X_t)$ and let $T_t$ be
linear operators on $\QQ[x,z]$ such that, for every polynomial $\Phi$,
\[
 (\pi_t)_*(A_t\operatorname{ev}_t(\Phi))
       =\operatorname{ev}_{t-1}(T_t\Phi).
\]
For a polynomial $\Phi$, put
$\Phi_t=T_{t+1}\cdots T_\ell\Phi$, with $\Phi_\ell=\Phi$.
Then, for $0\le t\le\ell$,
\begin{align}
 &(p_{\ell,t})_*\!\left[
  \left(\prod_{a=1}^{\ell}p_{\ell,a}^*A_a\right)
            \operatorname{ev}_\ell(\Phi)\right]\notag\\
 &\qquad=
  \left(\prod_{a=1}^{t}p_{t,a}^*A_a\right)
             \operatorname{ev}_t(\Phi_t).
 \label{eq:master-descending-induction}
\end{align}
In particular, at $t=0$ the output is
$\operatorname{ev}_0(T_1\cdots T_\ell\Phi)$.
\end{lemma}
\begin{proof}
For $t=\ell$ the identity is tautological. Assuming it at $t\ge1$,
every factor with $a<t$ is pulled back from $X_{t-1}$. The projection
formula moves those factors outside $(\pi_t)_*$, and the assumed
local identity replaces $A_t\operatorname{ev}_t(\Phi_t)$ by
$\operatorname{ev}_{t-1}(T_t\Phi_t)
=\operatorname{ev}_{t-1}(\Phi_{t-1})$.
Composition of pushforwards proves the identity at $t-1$.
\end{proof}

Let $\mu=(\mu_1,\ldots,\mu_n)$ be a partition, padded by zero, and put
\begin{equation}\label{eq:dj-def}
 d_j=\mu_j-\mu_{j+1},
 \qquad \mu_{n+1}=0.
\end{equation}
Let $\epsilon=(\epsilon_1,\ldots,\epsilon_\ell)\in\NN^\ell$ and use
\eqref{eq:operator-order}:
\begin{equation}\label{eq:P-operator-word}
 P_{u,\epsilon,\mu}(x,z)
 :=\Gamma_{i_1}^{(\epsilon_1,z)}\cdots
 \Gamma_{i_\ell}^{(\epsilon_\ell,z)}x^\mu.
\end{equation}
Thus $\Gamma_{i_\ell}$ acts first.  Since
$\Gamma_i^{(\epsilon,z)}$ raises total degree by $\epsilon$,
\begin{equation}\label{eq:P-degree}
 P_{u,\epsilon,\mu}\in\QQ[x,z]_{|\mu|+\sum_{t=1}^{\ell}\epsilon_t}.
\end{equation}

The terminal monomial is a product of top Chern classes of prefix bundles, so the local formulas give a single Chow identity.
\begin{proposition}\label{prop:operator-chow}
On the canonical-row tower $X_u$,
\begin{equation}\label{eq:operator-chow}
 P_{u,\epsilon,\mu}(h,h_z)
 =\rho_*\left[
 \prod_{t=1}^\ell p_{\ell,t}^*(q_t+h_z)
 \prod_{t=1}^{\ell}p_{\ell,t}^*s_t^{\epsilon_t}
 \prod_{j=1}^n c_j(\cF_j^{(\ell)})^{d_j}
 \right].
\end{equation}
\end{proposition}

\begin{proof}
Since the final root lines are the successive quotients of the flag,
\begin{equation}\label{eq:terminal-monomial}
 x^\mu(r^{(\ell)})
 =\prod_{j=1}^n(r_1^{(\ell)}\cdots r_j^{(\ell)})^{d_j}
 =\prod_{j=1}^nc_j(\cF_j^{(\ell)})^{d_j}.
\end{equation}
Apply \cref{lem:time-indexed-pushforward} with
$A_t=s_t^{\epsilon_t}(q_t+h_z)$,
$T_t=\Gamma_{i_t}^{(\epsilon_t,z)}$, and $\Phi=x^\mu$.
Its local hypothesis is \cref{lem:gysin}.
At $t=0$, the initial roots are $h_1,\ldots,h_n$; the lemma and
\eqref{eq:terminal-monomial} give \eqref{eq:operator-chow}.
\end{proof}

\subsection{Generated row and co-row bundles}

Fix a nonempty canonical row
$B_a=(a,a+1,\ldots,a+c_a-1)$.
Put $e(a,t)=\sum_{r>a}c_r+t$ for $0\le t\le c_a$.
For $t\ge1$, the quotient line created at this global stage is
$\cQ_{a,t}:=\cQ_{e(a,t)}$, with $q_{a,t}=c_1(\cQ_{a,t})$.
On $X_{e(a,t)}$, immediately after the first $t$ modifications of the row,
define
\begin{equation}\label{eq:Ga-t}
 \cG_{a,t}=\ker\bigl(\cF_{a+t-1}^{(e(a,t))}
                    \longrightarrow \cF_{a-1}^{(e(a,t))}\bigr),
 \qquad \cG_{a,0}=0.
\end{equation}
In the filtration below, the earlier bundle $\cG_{a,t-1}$ is pulled back
to $X_{e(a,t)}$. Each completed bundle $\cG_a=\cG_{a,c_a}$ is thereafter
kept fixed under pullback; it is not redefined using later modified
prefixes.

Compatible quotient maps supply both the row filtration and the surjection that proves global generation.
\begin{proposition}\label{prop:row-bundle}
For $1\le t\le c_a$, there are exact sequences
\begin{equation}\label{eq:row-filtration-exact}
 0\longrightarrow \cQ_{a,t}\longrightarrow \cG_{a,t}
 \longrightarrow \cG_{a,t-1}\longrightarrow0.
\end{equation}
At the end of the row, the bundle $\cG_a:=\cG_{a,c_a}$ is globally generated; more precisely, there is a natural surjection
\begin{equation}\label{eq:row-surjection}
 \cL_a\oplus\cdots\oplus\cL_n\twoheadrightarrow \cG_a.
\end{equation}
Moreover,
\begin{equation}\label{eq:row-topchern}
 c_{c_a}(\cG_a\otimes\cM)
 =\prod_{t=1}^{c_a}(q_{a,t}+h_z).
\end{equation}
\end{proposition}

\begin{proof}
At the $t$th step of row $B_a$, \eqref{eq:Q-exact} is
$0\to\cQ_{a,t}\to\cF_{a+t-1}\to\cF_{a+t-2}\to0$.
Taking kernels of the compatible maps to the unchanged $\cF_{a-1}$
gives \eqref{eq:row-filtration-exact}; local lifting proves its
surjectivity. Earlier bundles are pulled back to this stage.

All preceding row starts exceed $a$, so the quotient
$\cE\to\cF_{a-1}=\bigoplus_{i<a}\cL_i$ is the original split quotient.
At row completion, $\cE\twoheadrightarrow\cF_{a+c_a-1}$ is compatible
with it. Restriction to its kernel gives the surjection
$\bigoplus_{i\ge a}\cL_i\twoheadrightarrow\cG_a$: a local lift in
$\cE$ of a section of $\cG_a$ necessarily maps to zero in $\cF_{a-1}$.
This proves global generation. Tensoring the row filtration by the
generated line bundle $\cM$ and applying Whitney's formula gives
\eqref{eq:row-topchern}. The completed row is subsequently kept fixed
under pullback.
\end{proof}

\begin{remark}\label{warn:individual-Q}
The line $\cQ_t$ is a subbundle of the globally generated bundle $\cF_{i_t}^{(t)}$ in \eqref{eq:Q-exact}; it need not itself be globally generated.  Positivity enters only after all quotient roots in one canonical row are assembled into $\cG_a$.
\end{remark}

\label{subsec:corow-bundles}
Atoms use the upper root present before a modification, rather than the
quotient root created by it. These old upper roots are controlled by a
second generated bundle, defined at the start of each row.

Write
$\mathbf r(u)=B_{n-1}B_{n-2}\cdots B_1, \qquad B_a=(a,a+1,\ldots,a+c_a-1)$.
For a nonempty row $B_a$, let
\[
 b_a=1+\sum_{r>a}c_r,
 \qquad Y_a=X_{b_a-1},
\]
so that $Y_a$ is the tower immediately before row $B_a$ is processed.
Write
\[
 \cF_j^{\langle a\rangle}=\cF_j^{(b_a-1)},
 \qquad
 \cR_j^{\langle a\rangle}
 =\ker(\cF_j^{\langle a\rangle}\to \cF_{j-1}^{\langle a\rangle}).
\]

\begin{definition}\label{def:corow}
For $0\le t\le c_a$, define on $Y_a$
\begin{equation}\label{eq:corow-def}
 \cC_{a,t}:=
 \ker\bigl(\cF_{a+t}^{\langle a\rangle}
 \longrightarrow \cF_a^{\langle a\rangle}\bigr),
 \qquad \cC_{a,0}=0,
\end{equation}
and put $\cC_a=\cC_{a,c_a}$.
\end{definition}

The analogous kernel filtration at row start is generated by the smaller split tail.
\begin{proposition}\label{prop:corow}
For $1\le t\le c_a$, there is an exact sequence
\begin{equation}\label{eq:corow-filtration}
 0\longrightarrow \cR_{a+t}^{\langle a\rangle}
 \longrightarrow \cC_{a,t}
 \longrightarrow \cC_{a,t-1}\longrightarrow0.
\end{equation}
Moreover, there is a natural surjection
\begin{equation}\label{eq:corow-surjection}
 \cL_{a+1}\oplus\cdots\oplus\cL_n
 \twoheadrightarrow \cC_{a,t}.
\end{equation}
Thus every $\cC_{a,t}$ is globally generated, and
\begin{equation}\label{eq:corow-chern}
 c_{c_a}(\cC_a\otimes\cM)
 =\prod_{t=1}^{c_a}
 \bigl(c_1(\cR_{a+t}^{\langle a\rangle})+h_z\bigr).
\end{equation}
\end{proposition}

\begin{proof}
Taking kernels in the compatible quotient maps
$\cF_{a+t}^{\langle a\rangle}\twoheadrightarrow
 \cF_{a+t-1}^{\langle a\rangle}\twoheadrightarrow
 \cF_a^{\langle a\rangle}$
gives \eqref{eq:corow-filtration}, with surjectivity by local lifting.
All preceding rows start strictly above $a$, so
$\cE\to\cF_a^{\langle a\rangle}=\bigoplus_{i\le a}\cL_i$
is the original split quotient. Restricting the compatible surjection
$\cE\twoheadrightarrow\cF_{a+t}^{\langle a\rangle}$ to its kernel
therefore gives \eqref{eq:corow-surjection}, by the same lifting
argument as in \cref{prop:row-bundle}. Hence each $\cC_{a,t}$ is
generated. Tensor the filtration by $\cM$ and apply Whitney's formula
to obtain \eqref{eq:corow-chern}.
\end{proof}

To apply this filtration to atoms, we must identify its factors with the roots present at the correct geometric stage.
\begin{lemma}\label{lem:old-upper}
At the $t$th stage inside row $B_a$, whose global stage is
$e=b_a+t-1$ and whose colour is $i_e=a+t-1$, one has
\begin{equation}\label{eq:old-upper-static}
 \cR_{i_e+1}^{(e-1)}
 =p_{e-1,b_a-1}^*\cR_{a+t}^{\langle a\rangle},
\end{equation}
where $p_{e-1,b_a-1}:X_{e-1}\to X_{b_a-1}$ is the tower projection.
\end{lemma}

\begin{proof}
Before stage $e$, the modifications already performed inside the current
row have colours $a,a+1,\ldots,a+t-2$.  A colour-$c$ modification changes
only the successive root lines with indices $c$ and $c+1$.  None of these
stages changes the root of index $a+t=i_e+1$, proving
\eqref{eq:old-upper-static}.
\end{proof}

\section{Chern flow and generated packet bundles}\label{sec:flow}

The row and co-row bundles of \cref{sec:canonical} account for the
quotient and old upper factors. Kernel factors require terminal top
Chern classes. Each exact-sequence edge contributes
$[\cF_{t(e)}]-[\cF_{s(e)}]=[\cE_e]$ in $K^0$; collecting these
contributions gives a generated direct sum whenever the resulting
vertex multiplicities are nonnegative. We first formulate this identity,
then apply it to the creation states of the tower.

\subsection{Chern flow and exact sequences}

\begin{proposition}
\label{thm:chern-flow}
Let $Y$ be a separated scheme of finite type over a field, assume all
bundles below have constant rank, and let $\Gamma$ be a finite directed multigraph.  Each
vertex $v$ carries a globally generated vector bundle $\cF_v$.  For every
ordinary edge $e:v\to w$, fix a short exact sequence
\begin{equation}\label{eq:flow-exact}
 0\longrightarrow \cE_e\longrightarrow \cF_w\longrightarrow \cF_v
 \longrightarrow0
\end{equation}
and an integral multiplicity $\kappa_e\in\NN$.  Fix also root multiplicities
$r_v\in\NN$, and define the weighted divergence
\begin{equation}\label{eq:vertex-divergence}
 b_v:=r_v+\sum_{e:\,t(e)=v}\kappa_e
          -\sum_{e:\,s(e)=v}\kappa_e.
\end{equation}
Assume $b_v\ge0$ for every vertex.  Then there is an equality in $K^0(Y)$,
\begin{equation}\label{eq:flow-K0}
 \left[\bigoplus_v\cF_v^{\oplus b_v}\right]
 =
 \left[
 \bigoplus_v\cF_v^{\oplus r_v}
 \oplus
 \bigoplus_e\cE_e^{\oplus\kappa_e}
 \right].
\end{equation}
Consequently the total Chern classes agree:
\begin{equation}\label{eq:flow-total-chern}
 c\!\left(\bigoplus_v\cF_v^{\oplus b_v}\right)
 =
 \prod_vc(\cF_v)^{r_v}
 \prod_ec(\cE_e)^{\kappa_e}.
\end{equation}
Taking the component in the common top degree gives
\begin{equation}\label{eq:flow-topchern}
 \ctop\!\left(\bigoplus_v\cF_v^{\oplus b_v}\right)
 =
 \left(\prod_v c_{\rk \cF_v}(\cF_v)^{r_v}\right)
 \left(\prod_e c_{\rk \cE_e}(\cE_e)^{\kappa_e}\right).
\end{equation}
The bundle on the left is globally generated and is determined canonically
by the vertexwise divergence.  Moreover,
\begin{equation}\label{eq:divergence-rank}
 \sum_v b_v\rk \cF_v
 =
 \sum_v r_v\rk \cF_v+\sum_e\kappa_e\rk \cE_e.
\end{equation}
No acyclicity assumption is required.
\end{proposition}

\begin{proof}
Summing $[\cE_e]=[\cF_{t(e)}]-[\cF_{s(e)}]$ with multiplicities gives
\[
 \sum_v b_v[\cF_v]=\sum_vr_v[\cF_v]+\sum_e\kappa_e[\cE_e],
\]
which is \eqref{eq:flow-K0}; its ranks give
\eqref{eq:divergence-rank}. Apply the multiplicative operation
$c_t([\cE]-[\cF])=c_t(\cE)c_t(\cF)^{-1}$ in $1+tA^*(Y)[[t]]$
\cite[\S3.2]{Fulton98}, interpreted operationally on Chow homology.
The inverse exists since the constant term is one. This proves
\eqref{eq:flow-total-chern}. Both sides have equal rank, so its top
component takes the top class from every factor, giving
\eqref{eq:flow-topchern}. Nonnegative $b_v$ make the left-hand bundle
an actual generated direct sum.
\end{proof}

For line kernels of multiplicity one, \eqref{eq:flow-topchern} reads
$\prod_e c_1(\cE_e)\prod_v c_{\rk\cF_v}(\cF_v)^{r_v}
=\ctop(\bigoplus_v\cF_v^{\oplus b_v})$.
The identity requires no path decomposition. Exterior-power identities,
pathwise exact sequences, and the flow cone are treated in
\cref{subsec:flow-cone-geometry}; the divergence bundle is fixed,
whereas a path matching need not be canonical.

\begin{remark}\label{warn:time-indices}
Vertices are actual bundles, not ranks. Distinct rank-$j$ flag states
may be nonisomorphic; collapsing their creation times can spuriously
compose unrelated exact sequences. Every graph below retains those times.
\end{remark}

\subsection{Creation states and balance conditions}

The graph must remember when a prefix bundle is actually created, while
identifying the passive pullback copies that occur between two creations.
This is encoded as follows.

\begin{definition}\label{def:creation-equivalence}
For a word $\mathbf i=(i_1,\ldots,i_\ell)$ and
$j\in\{1,\ldots,n\}$, declare
\[
 (j,s)\sim(j,t)
 \quad\Longleftrightarrow\quad
 i_q\ne j\text{ for every }
 \min\{s,t\}<q\le\max\{s,t\}.
\]
Write $\sigma_j(t)=[j,t]$ for the equivalence class.  Its unique
creation-time representative is
\[
 \bigl(j,\tau_j(t)\bigr),\qquad
 \tau_j(t)=\max\bigl(\{0\}\cup\{q\le t:i_q=j\}\bigr).
\]
On the quotient-flag tower of \cref{sec:canonical}, if
$s\le t$ and $(j,s)\sim(j,t)$, then
$\cF_j^{(t)}$ is canonically the pullback of $\cF_j^{(s)}$.  Hence the pullback
to the final tower of the bundle represented by $\sigma_j(t)$ is
independent of the representative.
\end{definition}

\begin{definition}\label{def:creation-state-graph}
Let $\mathbf i=(i_1,\ldots,i_\ell)$ be a word in
$\{1,\ldots,n-1\}$, let
$\epsilon=(\epsilon_1,\ldots,\epsilon_\ell)\in\NN^\ell$, and let
$\mathbf d=(d_1,\ldots,d_n)\in\NN^n$.  For $1\le j\le n$, put
\begin{equation}\label{eq:creation-times}
\begin{gathered}
 T_j=\{0\}\cup\{t\in[\ell]:i_t=j\},\\
 \tau_j(t)=\max\bigl(T_j\cap\{0,1,\ldots,t\}\bigr),\qquad
 \lambda_j=\tau_j(\ell).
\end{gathered}
\end{equation}
The vertex $v_{j,r}$, for $r\in T_j$, is the rank-$j$ state created at
time $r$; in particular, $v_{j,0}$ is the initial state.  At stage $t$ the
current rank-$j$ state is $v_{j,\tau_j(t)}$.  There is deliberately no
separate vertex for a passive pullback at a time not belonging to $T_j$.

The rooted directed multigraph
$\mathfrak C(\mathbf i,\epsilon,\mathbf d)$ has vertex set
$\{0\}\sqcup\{v_{j,r}:1\le j\le n,\ r\in T_j\}$.
At stage $t$ it has $\epsilon_t$ parallel copies of the ordinary edge
\begin{equation}\label{eq:abstract-flow-edge}
 e_t:v_{i_t,t}\longrightarrow
 v_{i_t+1,\tau_{i_t+1}(t-1)}.
\end{equation}
For every $j$ it has $d_j$ root edges
\begin{equation}\label{eq:abstract-root-edges}
 e_{j,a}:0\longrightarrow v_{j,\lambda_j}
 \qquad(1\le a\le d_j).
\end{equation}
Degrees are counted with edge multiplicity. The graph is
\emph{balanced} if, at every non-root vertex, the total incoming
multiplicity (including root edges) is at least the outgoing ordinary
multiplicity.
\end{definition}

Every ordinary edge in \eqref{eq:abstract-flow-edge} increases the rank by
one.  Thus $\mathfrak C(\mathbf i,\epsilon,\mathbf d)$ is acyclic and the
root has no incoming edge.  The creation-state convention is exactly what
makes two consecutive edges correspond to composable exact sequences.

\begin{example}
\label{ex:creation-state-3412}
For $u=3412$, the canonical word is
$\mathbf r(u)=(2,3\mid1,2)$.  Thus
$T_1=\{0,3\},\qquad T_2=\{0,1,4\},\qquad T_3=\{0,2\},\qquad T_4=\{0\}$.
With the all-kernel profile and staircase terminal multiplicities
$\mathbf d=(1,1,1,0)$, the nonisolated part of the rooted graph is
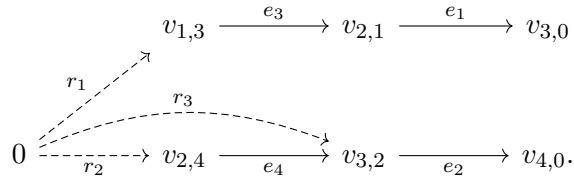
\[
\begin{tikzcd}[column sep=3.8em,row sep=2.8em]
 & v_{1,3} \arrow[r,"e_3"]
 & v_{2,1} \arrow[r,"e_1"]
 & v_{3,0} \\
 0 \arrow[ur,dashed,"r_1"] \arrow[r,dashed,"r_2"']
   \arrow[rr,bend left=22,dashed,"r_3"]
 & v_{2,4} \arrow[r,"e_4"']
 & v_{3,2} \arrow[r,"e_2"']
 & v_{4,0}.
\end{tikzcd}
\]
Solid edges represent exact sequences; dashed edges record terminal
top-Chern factors. The two rank-two states, like the two rank-three
states, must remain distinct to preserve composability.
\Cref{app:worked-example} follows this graph through the packet and
incidence construction.
\end{example}

We now determine the balance condition needed in the Grothendieck
application.  Let $\mathbf r(u)=(i_1,\ldots,i_\ell)$ be the canonical-row
word defined in \cref{sec:canonical}, and take $\epsilon_t=1$ for every
$t$.  For $1\le j<n$, let $m_j$ be the number of occurrences of colour
$j$ and put
\[
 L_j=
 \begin{cases}
 \max\{t:i_t=j\},&m_j>0,\\
 -\infty,&m_j=0,
 \end{cases}
 \qquad L_0=-\infty.
\]
Define
\begin{equation}\label{eq:epsilon-reservoir}
 \varepsilon_j(u)=
 \mathbf 1_{\{m_j>0\text{ and }L_j>L_{j-1}\}}.
\end{equation}

\begin{lemma}\label{thm:reservoir}
The graph $\mathfrak C(\mathbf r(u),\mathbf 1,\mathbf d)$ is balanced
at every non-root vertex if and only if
\begin{equation}\label{eq:reservoir-criterion}
 d_j\ge\varepsilon_j(u)
 \qquad(1\le j<n).
\end{equation}
No condition is required in rank $n$.
\end{lemma}

\begin{proof}
Write $T_j=\{0=t_{j,0}<\cdots<t_{j,m_j}\}$.
Initial states have no outgoing edge; each noninitial state has exactly
one, created at its own colour-$j$ step. A next occurrence $t_{j,k+1}$
of $j$ lies in a later row, whose strictly smaller start forces colour
$j-1$ immediately before it. Since no intervening step creates rank $j$,
$\tau_j(t_{j,k+1}-2)=t_{j,k}$, so that colour-$(j-1)$ edge enters the
previous state. Every nonfinal noninitial state is thus balanced.
The final state at $L_j$, when $m_j>0$, receives an ordinary edge
exactly when a colour-$(j-1)$ step occurs after $L_j$, equivalently
$L_{j-1}>L_j$. Otherwise its sole outgoing edge requires $d_j\ge1$.
This is \eqref{eq:reservoir-criterion}. If $m_j=0$ the state is initial,
and rank $n$ has no outgoing edge.
\end{proof}

In particular, the staircase differences $d_j=1$ for $j<n$ always
suffice, since $\varepsilon_j(u)\in\{0,1\}$.

\subsection{Top-Chern realizations of balanced packets}

We now combine the row bundles with the creation-state flow on the final tower.  For a
vertex $v_{j,r}$, with $r\in T_j$, define
\begin{equation}\label{eq:vertex-bundle}
 \widehat{\cF}_{j,r}:=p_{\ell,r}^*\cF_j^{(r)}.
\end{equation}
This is the pullback to $X_u$ of the rank-$j$ bundle $\cF_j^{(r)}$ on $X_r$.
For a creation-state vertex $v=v_{j,r}$, abbreviate
$\widehat{\cF}_v:=\widehat{\cF}_{j,r}$.  If $r=\tau_j(t)$, then the pullback of
the current bundle $\cF_j^{(t)}$ to $X_u$ is canonically isomorphic to
$\widehat{\cF}_{j,r}$; passive pullbacks therefore do not create additional
vertices.

For every stage $t$, pull \eqref{eq:S-exact} to $X_u$.  Independently of
any profile, this gives the literal exact sequence
\begin{equation}\label{eq:realized-flow-edge}
 0\longrightarrow \cS_t\longrightarrow
 \widehat{\cF}_{i_t+1,\tau_{i_t+1}(t-1)}
 \longrightarrow \widehat{\cF}_{i_t,t}\longrightarrow0.
\end{equation}
In $\mathfrak C(\mathbf r(u),\epsilon,\mathbf d)$ this sequence is
used with multiplicity $\epsilon_t$, contributing $s_t^{\epsilon_t}$.
The $d_j$ root edges at $v_{j,\lambda_j}$ contribute
$c_j(\widehat{\cF}_{j,\lambda_j})^{d_j}
=c_j(\cF_j^{(\ell)})^{d_j}$. Thus every edge is realized by the
literal exact sequence required by \cref{thm:chern-flow}, and all
vertex bundles are globally generated.

\begin{theorem}\label{thm:canonical-packet}
Let $u\in S_n$, write $\mathbf r(u)=(i_1,\ldots,i_\ell)$, and let
$\mu=(\mu_1\ge\cdots\ge\mu_n\ge0)$ be a partition. Put
$\mu_{n+1}=0$ and $d_j=\mu_j-\mu_{j+1}$ for $1\le j\le n$.
For $\epsilon\in\NN^\ell$, assume that
$\mathfrak C(\mathbf r(u),\epsilon,\mathbf d)$ is balanced.
There is a globally generated vector bundle $\cV_{u,\epsilon,\mu}$
on $X_u$ such that
\begin{equation}\label{eq:canonical-topchern-push}
 \rho_*\ctop(\cV_{u,\epsilon,\mu})=P_{u,\epsilon,\mu}(h,h_z).
\end{equation}
For every finite box $\mathbf m=(m_1,\ldots,m_n,m_z)$ containing
the support, consequently,
\begin{equation}\label{eq:canonical-finite-dual-RV}
 \Dcomp_{\mathbf m}\Nrm_{x,z}(P_{u,\epsilon,\mu})\in\RV.
\end{equation}
This applies in particular to $\epsilon=\mathbf0$ and any partition
$\mu$, and to $\epsilon=\mathbf1$ with $\mu=\delta$.
\end{theorem}

\begin{proof}
For each creation-state vertex $v$, let $r_v$ be its root-edge
multiplicity and put
$b_v=r_v+\operatorname{indeg}_{\rm ord}(v) -\operatorname{outdeg}_{\rm ord}(v)$.
Balance says exactly that $b_v\ge0$.  Applying \cref{thm:chern-flow} to the
realized edge sequences \eqref{eq:realized-flow-edge} gives
\[
 \prod_{t=1}^{\ell}s_t^{\epsilon_t}
 \prod_{j=1}^n c_j(\cF_j^{(\ell)})^{d_j}
 =\ctop(\cU_{u,\epsilon,\mu}),
 \qquad
 \cU_{u,\epsilon,\mu}:=\bigoplus_v\widehat{\cF}_v^{\oplus b_v}.
\]
This globally generated bundle is canonical: it is independent of every
choice of path matching.  By \cref{prop:row-bundle},
\[
 \prod_{t=1}^\ell(q_t+h_z)
 =\prod_{a:c_a>0}c_{c_a}(\cG_a\otimes\cM),
\]
and each $\cG_a\otimes\cM$ is globally generated.  Set
\begin{equation}\label{eq:V-canonical}
 \cV_{u,\epsilon,\mu}
 =\cU_{u,\epsilon,\mu}
 \oplus\bigoplus_{a:c_a>0}(\cG_a\otimes\cM).
\end{equation}
Its top Chern class is exactly the integrand in
\eqref{eq:operator-chow}.  The pushforward identity
\eqref{eq:canonical-topchern-push} follows from
\cref{prop:operator-chow}.

The divergence rank identity \eqref{eq:divergence-rank} gives
\[
 \rk \cU_{u,\epsilon,\mu}
 =\sum_{j=1}^n j d_j+\sum_{t=1}^\ell\epsilon_t
 =|\mu|+\sum_{t=1}^\ell\epsilon_t.
\]
The row bundles have total rank $\sum_ac_a=\ell$, and hence
\begin{equation}\label{eq:rank-codimension-bookkeeping}
 \rk \cV_{u,\epsilon,\mu}
 =\ell+|\mu|+\sum_t\epsilon_t,
 \qquad \dim(X_u/B)=\ell.
\end{equation}
Thus pushforward lowers the codimension by $\ell$ and produces a class of
degree $|\mu|+\sum_t\epsilon_t$, in agreement with
\eqref{eq:P-degree}.

Finally, fix a finite box containing the support of
$P_{u,\epsilon,\mu}$ and choose the corresponding product of projective
spaces as the base $B$.  Because the support lies in the Chow box, the
polynomial $P_{u,\epsilon,\mu}(h,h_z)$ is already its unique reduced
representative in $A^*(B)$.  The construction above therefore gives the
required top-Chern pushforward of that class, and
\cref{thm:incidence-dual} yields \eqref{eq:canonical-finite-dual-RV}.
For $\epsilon=\mathbf0$ balance is automatic; for
$\epsilon=\mathbf1$ and $\mu=\delta$ it follows from
\cref{thm:reservoir}, since $d_j=1$ for $j<n$.
\end{proof}

The cap $\mathbf m$ includes the $z$-coordinate. By
\eqref{eq:rank-codimension-bookkeeping}, the nonzero finite dual has degree
$|\mathbf m|-|\mu|-\sum_t\epsilon_t=\dim X_u-\rk\cV_{u,\epsilon,\mu}$:
this is the dimension left after the top-Chern insertion.

\section{Ordinary \texorpdfstring{type-$A$}{type-A} polynomial packets}\label{sec:ordinary-packets}
\label{sec:ordinary-type-A-packets}
The top-Chern models of \cref{sec:canonical,sec:flow} realize finite
complements. We now identify those complements with the normalized
source packets. The three local correspondences are
\[
\begin{array}{c|c|c}
\text{source operator}&\text{reciprocal operator}&\text{generated factors}\\ \hline
\mathscr K_i^{(z)}&\Pi_{n-i}^{(z)}&\text{quotient rows}\\
\overline{\mathscr K}_i^{(z)}&\overline\Pi_{n-i}^{(z)}&\text{co-rows}\\
\mathscr D_i^{(z)}&\Gamma_{n-i}^{(1,z)}&\text{rows and rooted flow}
\end{array}
\]
Here $\overline\Pi_i^{(z)}=(x_{i+1}+z)\partial_i$.
The reciprocal correspondence includes variable reversal and matching finite
caps; the precise identities are proved below. After these three
conversions, \cref{subsec:packet-extraction} treats their layer and
support consequences together. No general preservation theorem for
finite complementation is used.

\subsection{Homogeneous Lascoux packets and key polynomials}\label{sec:key}

We now apply the quotient-root part of the canonical-row construction.  No
kernel-root factors occur in this section, so the flow condition is automatic.

For $1\le i<n$, define
\begin{equation}\label{eq:K-operator-key-section}
 \mathscr K_i^{(z)}
 :=z\pi_i+x_ix_{i+1}\partial_i
 =\partial_i\bigl(x_i(x_{i+1}+z)\,\cdot\,\bigr).
\end{equation}
The second equality follows because $x_ix_{i+1}$ is $s_i$-invariant and
$\partial_i(x_if)=\pi_i f$.  If
\[
 \overline\pi_i^{(\beta)}f
 :=\partial_i\bigl(x_i(1+\beta x_{i+1})f\bigr),
 \qquad
 \mathfrak d_i^{(\beta)}f
 :=\partial_i\bigl((1+\beta x_{i+1})f\bigr),
\]
then
\[
 \mathscr K_i^{(z)}=z\,\overline\pi_i^{(z^{-1})},
 \qquad
 \mathscr D_i^{(z)}=z\,\mathfrak d_i^{(z^{-1})}.
\]
Both operator families $\overline\pi_i^{(\beta)}$ and
$\mathfrak d_i^{(\beta)}$ satisfy the Coxeter braid and
distant-commutation relations; for the Demazure--Lascoux family see
\cite[\S2.2, equations (2.2a)--(2.2c)]{BSW20}, and for the Grothendieck
family see \cite{LS82,FK94}. Their homogeneous forms
$\mathscr K_i^{(z)}$ and $\mathscr D_i^{(z)}$ therefore satisfy the same
relations.  In particular, their products along a reduced word are
well-defined.

A common exponent bound is needed to transport the operator identities through finite reciprocals.
\begin{lemma}\label{lem:finite-box}
Let $M,r\ge0$, and suppose that every monomial of $f(x,z)$ has
$x_j$-exponent at most $M$ for every $j$ and $z$-exponent at most $r$.
Then every monomial of $\mathscr K_i^{(z)}f$ and
$\mathscr D_i^{(z)}f$ has the same $x$-exponent bounds and $z$-exponent at
most $r+1$.
\end{lemma}

\begin{proof}
For $p,q\ge0$ one has
\begin{equation}\label{eq:monomial-divided-difference}
 \partial_i(x_i^px_{i+1}^q)=
 \begin{cases}
 \displaystyle\sum_{a=0}^{p-q-1}x_i^{p-1-a}x_{i+1}^{q+a},&p>q,\\[2mm]
 0,&p=q,\\[1mm]
 \displaystyle-\sum_{a=0}^{q-p-1}x_i^{p+a}x_{i+1}^{q-1-a},&p<q.
 \end{cases}
\end{equation}
Now
\[
 \mathscr K_i^{(z)}(x_i^px_{i+1}^qz^b)
 =z^b\partial_i(x_i^{p+1}x_{i+1}^{q+1})
  +z^{b+1}\partial_i(x_i^{p+1}x_{i+1}^{q}),
\]
and
\[
 \mathscr D_i^{(z)}(x_i^px_{i+1}^qz^b)
 =z^b\partial_i(x_i^px_{i+1}^{q+1})
  +z^{b+1}\partial_i(x_i^px_{i+1}^{q}).
\]
Formula \eqref{eq:monomial-divided-difference} shows term by term that no
$x_i$- or $x_{i+1}$-exponent exceeds $\max\{p,q\}\le M$; the other
variables are unchanged, and the $z$-exponent increases by at most one.
\end{proof}

For $M,r\ge0$, let
\begin{equation}\label{eq:reciprocal-Mr}
 \mathcal R_{M,r}f(x,z)
 =(x_1\cdots x_n)^M z^r
 f(x_1^{-1},\ldots,x_n^{-1},z^{-1}),
\end{equation}
and let $\omega$ be the variable-reversal involution
\begin{equation}\label{eq:omega-reversal}
 (\omega f)(x_1,\ldots,x_n,z)=f(x_n,\ldots,x_1,z).
\end{equation}
When $f$ is supported in the box $(M,\ldots,M;r)$, the operator
$\mathcal R_{M,r}$ is an involution.

Define also
\begin{equation}\label{eq:Pi-operator}
 \Pi_i^{(z)}f=\partial_i\bigl((x_i+z)f\bigr)
 =f+(x_{i+1}+z)\partial_i f.
\end{equation}

The following common transport statement treats all three families. Set
$\overline{\mathscr K}_i^{(z)}=\mathscr K_i^{(z)}-z\,\mathrm{id}$ and
$\overline\Pi_i^{(z)}=(x_{i+1}+z)\partial_i$.
\begin{lemma}\label{lem:reciprocal-transport}
On the indicated finite boxes, the local identities are
\begin{align}
 \mathcal R_{M,r+1}\mathscr K_i^{(z)}\mathcal R_{M,r}
   &=1-(x_i+z)\partial_i,\label{eq:local-lascoux-conjugation}\\
 \mathcal R_{M,r+1}\overline{\mathscr K}_i^{(z)}\mathcal R_{M,r}
   &=-(x_i+z)\partial_i,\label{eq:local-atom-reciprocal}\\
 \mathcal R_{M,r+1}\mathscr D_i^{(z)}\mathcal R_{M,r}
   &=-\mathscr K_i^{(z)}.\label{eq:local-D-conjugation}
\end{align}
Reversal sends these right-hand sides, respectively, to
\begin{align}
 \omega\bigl(1-(x_i+z)\partial_i\bigr)\omega
   &=\Pi_{n-i}^{(z)},\label{eq:local-lascoux-reversal}\\
 \omega\bigl(-(x_i+z)\partial_i\bigr)\omega
   &=\overline\Pi_{n-i}^{(z)},\label{eq:atom-reversal}\\
 \omega(-\mathscr K_i^{(z)})\omega
   &=\Gamma_{n-i}^{(1,z)}.\label{eq:local-D-reversal}
\end{align}
For a word $(i_1,\ldots,i_s)$ in any one source family $T_i$, let
$T_i^\vee=\mathcal R_{M,q+1}T_i\mathcal R_{M,q}$, the corresponding
right-hand side above, independent of $q$. Then
\begin{equation}\label{eq:capped-word}
 \mathcal R_{M,r+s}T_{i_1}\cdots T_{i_s}\mathcal R_{M,r}
       =T_{i_1}^\vee\cdots T_{i_s}^\vee.
\end{equation}
The reciprocal families $\Pi_i^{(z)},\overline\Pi_i^{(z)}$, and
$\Gamma_i^{(1,z)}$ satisfy the braid and distant-commutation relations.
\end{lemma}
\begin{proof}
The local $\mathscr K$ and $\mathscr D$ calculations are given in
\cref{app:operator-conjugation}. The atom operator has the same finite-box
bound as $\mathscr K_i^{(z)}$, since it differs by $z\,\mathrm{id}$.
Its reciprocal identity follows by subtracting
$\mathcal R_{M,r+1}z\mathcal R_{M,r}=1$ from the first identity.
Reversal uses $\omega\partial_i\omega=-\partial_{n-i}$ and
$\omega(x_i)=x_{n+1-i}$. For the word identity insert matching caps:
\[
 \mathcal R_{M,r+s}T_{i_1}\cdots T_{i_s}\mathcal R_{M,r}
 =\prod_{a=1}^s
  \bigl(\mathcal R_{M,r+s-a+1}T_{i_a}\mathcal R_{M,r+s-a}\bigr).
\]
Thus the word order is unchanged. Each Coxeter relation compares words
of equal length, so the same outer caps transport the source braid
relations. For atoms these follow from the braid relations for
$\varpi_i^{(\beta)}-\mathrm{id}$
\cite[\S2.2, discussion following equations (2.2)]{BSW20}.
\end{proof}

Let $\lambda=(\lambda_1,\ldots,\lambda_n)$ be a partition, put
$M=\lambda_1$, and let $w\in S_n$.  Define
\begin{equation}\label{eq:H-w-lambda}
 H_{w,\lambda}(x,z)=\mathscr K_w^{(z)}x^\lambda.
\end{equation}
By \cref{lem:finite-box}, every monomial of $H_{w,\lambda}$ lies in the box
$(M,\ldots,M;\length(w))$.  Put
\begin{equation}\label{eq:mu-reciprocal-key}
 \mu=(M-\lambda_n,M-\lambda_{n-1},\ldots,M-\lambda_1),
 \qquad u=w_0ww_0.
\end{equation}
The partition $\mu$ is dominant.

The complemented dominant monomial is therefore the terminal weight of the reciprocal word.
\begin{proposition}\label{prop:lascoux-reciprocal}
One has
\begin{equation}\label{eq:lascoux-reciprocal-normal-form}
 \omega\mathcal R_{M,\length(w)}H_{w,\lambda}
 =\Pi_u^{(z)}x^\mu.
\end{equation}
Here $\Pi_u^{(z)}$ may be computed along any reduced expression for $u$; in
particular, it may be computed along the canonical-row word $\mathbf r(u)$.
\end{proposition}

\begin{proof}
Apply \eqref{eq:capped-word} with $T_i=\mathscr K_i^{(z)}$ to a reduced word
$w=s_{i_1}\cdots s_{i_\ell}$, with initial $z$-cap zero.
Conjugating by $\omega$ changes each local factor to
$\Pi_{n-i_a}^{(z)}$ by \eqref{eq:local-lascoux-reversal}.
The resulting word represents $u=w_0ww_0$, and
$\omega\mathcal R_{M,0}x^\lambda=x^\mu$.
\Cref{lem:reciprocal-transport} permits replacement by $\mathbf r(u)$.
\end{proof}

The reciprocal formula and the generated quotient rows now give the Lascoux volume realization.
\begin{proposition}\label{thm:lascoux-RV}
For every partition $\lambda$ and every $w\in S_n$,
\begin{equation}\label{eq:lascoux-RV}
 \Nrm_{x,z}\bigl(H_{w,\lambda}(x,z)\bigr)\in\RV.
\end{equation}
If $\operatorname{char}(\kk)=0$ and $H_{w,\lambda}$ is nonzero, its factorial
normalization admits a smooth integral projective realization.
\end{proposition}

\begin{proof}
Apply \cref{thm:canonical-packet} to the canonical-row word of $u$ in
\eqref{eq:mu-reciprocal-key}, with terminal partition $\mu$ and
$\epsilon_t=0$ at every step.  By \cref{thm:canonical-packet}, the graph is
balanced.  Since $\Gamma_i^{(0,z)}=\Pi_i^{(z)}$, we obtain
$\Dcomp_{(M,\ldots,M;\ell)} \Nrm\bigl(\Pi_u^{(z)}x^\mu\bigr)\in\RV, \qquad \ell=\length(w)$.
By \cref{prop:lascoux-reciprocal} and the coefficientwise identity
\begin{equation}\label{eq:lascoux-finite-dual-identity}
 \Dcomp_{(M,\ldots,M;\ell)}
 \Nrm\bigl(\omega\mathcal R_{M,\ell}H_{w,\lambda}\bigr)
 =\omega\Nrm(H_{w,\lambda}),
\end{equation}
the right-hand side is a realizable-volume polynomial.  Variable reversal preserves
$\RV$, proving \eqref{eq:lascoux-RV}.  Smooth realization follows from
\cref{prop:RV-closure}.
\end{proof}

Let $\alpha$ be a weak composition, put $\lambda=\alpha^+$, and set
$w=w_\alpha$, where $w_\alpha$ is the minimal sorting permutation fixed
above.  Define
\begin{equation}\label{eq:positive-lascoux}
 \Omega_\alpha^+(x;\beta)
 :=\overline\pi_w^{(\beta)}x^\lambda
 =\sum_{k=0}^{r_\alpha}\Omega_{\alpha,k}(x)\beta^k.
\end{equation}
Here
$r_\alpha:=\deg_\beta\Omega_\alpha^+(x;\beta) =\max\{k:\Omega_{\alpha,k}\ne0\}$.
The constant coefficient is the nonzero key polynomial
$\Omega_{\alpha,0}=\kappa_\alpha$, so this maximum is well defined.
Every $\Omega_{\alpha,k}$ is homogeneous of degree $|\lambda|+k$.
Define the ordinary Lascoux polynomial by
\begin{equation}\label{eq:ordinary-lascoux}
 \Omega_\alpha(x):=\Omega_\alpha^+(x;-1).
\end{equation}
This is the convention of Monical--Tokcan--Yong: their local operator is
$\tau_i(f)=\partial_i(x_i(1-x_{i+1})f)$
\cite[Conjecture~5.6]{MTY19}.  Since
$\mathscr K_i^{(z)}=z\overline\pi_i^{(z^{-1})}$,
\begin{equation}\label{eq:H-lascoux-dehomogenization}
 H_{w,\lambda}(x,z)
 =z^{\length(w)}\Omega_\alpha^+(x;z^{-1})
 =\sum_{k=0}^{r_\alpha}\Omega_{\alpha,k}(x)
 z^{\length(w)-k}.
\end{equation}
In particular, $r_\alpha\le\length(w)$.  Put
\begin{equation}\label{eq:minimal-lascoux-packet}
 \widetilde\Omega_\alpha(x,z)
 :=\sum_{k=0}^{r_\alpha}\Omega_{\alpha,k}(x)z^{r_\alpha-k}.
\end{equation}

\subsection{Homogeneous Lascoux atoms and Demazure atoms}\label{sec:atoms}

The reciprocal Lascoux packet uses the quotient roots produced by a row.
For atoms the local reciprocal factor uses instead the \emph{old upper
root} present immediately before each modification.  These roots admit a
parallel rowwise positivity construction.

Recall the Lascoux operator
$\varpi_i^{(\beta)}f :=\partial_i\bigl(x_i(1+\beta x_{i+1})f\bigr) =\overline\pi_i^{(\beta)}f$
from \cref{sec:key}, and define its atom counterpart by
$\bar\varpi_i^{(\beta)}:=\varpi_i^{(\beta)}-\mathrm{id}$.
The operators $\bar\varpi_i^{(\beta)}$ satisfy the Coxeter braid and
distant-commutation relations and give the standard Lascoux atoms
\cite[\S2.2, discussion following equations (2.2)]{BSW20}.
Define the homogeneous atom operators
\begin{equation}\label{eq:atom-operators}
 \overline{\mathscr K}_i^{(z)}
 :=\mathscr K_i^{(z)}-z\,\mathrm{id}
 =z\bar\varpi_i^{(z^{-1})},
 \qquad
 \overline\Pi_i^{(z)}
 :=\Pi_i^{(z)}-\mathrm{id}
 =(x_{i+1}+z)\partial_i.
\end{equation}
Their products along reduced words are well defined. Subtracting
$z\,\mathrm{id}$ from $\mathscr K_i^{(z)}$ does not change the finite-box
bound in \cref{lem:finite-box}.

Let $\lambda=(\lambda_1\ge\cdots\ge\lambda_n)$, put $M=\lambda_1$, and
let $w\in S_n$ have length $\ell$.  Set
\begin{equation}\label{eq:atom-packet}
 \overline H_{w,\lambda}(x,z)
 :=\overline{\mathscr K}_w^{(z)}x^\lambda,
\end{equation}
and define
\begin{equation}\label{eq:atom-mu-u}
 \mu=(M-\lambda_n,M-\lambda_{n-1},\ldots,M-\lambda_1),
 \qquad u=w_0ww_0.
\end{equation}
Then $\overline H_{w,\lambda}$ is supported in
$(M,\ldots,M;\ell)$.

The local atom identity gives the following complete reciprocal expression.
\begin{proposition}\label{prop:atom-normal}
One has
\begin{equation}\label{eq:atom-normal}
 \omega\mathcal R_{M,\ell}\overline H_{w,\lambda}
 =\overline\Pi_u^{(z)}x^\mu.
\end{equation}
The right-hand side may be evaluated along the canonical-row word
$\mathbf r(u)$.
\end{proposition}

\begin{proof}
Insert matching reciprocal caps along a reduced word for $w$ and
apply \cref{lem:reciprocal-transport}. Reversal cancels the local minus
sign and changes the word to one for $u=w_0ww_0$, while
$\omega\mathcal R_{M,0}x^\lambda=x^\mu$. Braid independence, transported
by the same caps on equal-length words, permits the canonical-row word.
\end{proof}

The atom factor is pulled back from the preceding stage, so the local
formula and its iteration can be stated together.
\begin{proposition}\label{prop:atom-chow}
At a stage $t$ of colour $i=i_t$,
\begin{equation}\label{eq:atom-gysin}
 (\pi_t)_*\!\left[
 \pi_t^*(r_{i+1}^{(t-1)}+h_z)\Phi(r^{(t)},h_z)\right]
 =\bigl(\overline\Pi_i^{(z)}\Phi\bigr)(r^{(t-1)},h_z).
\end{equation}
For $\mathbf r(u)=(i_1,\ldots,i_\ell)$ and
$d_j=\mu_j-\mu_{j+1}$, consequently,
\begin{equation}\label{eq:atom-chow}
 \overline\Pi_u^{(z)}x^\mu(h,h_z)
 =\rho_*\!\left[
  \prod_{t=1}^{\ell}p_{\ell,t-1}^*(r_{i_t+1}^{(t-1)}+h_z)
  \prod_{j=1}^n c_j(\cF_j^{(\ell)})^{d_j}\right].
\end{equation}
\end{proposition}
\begin{proof}
The projection formula moves the prefactor in
\eqref{eq:atom-gysin} outside the pushforward; then
\eqref{eq:ranktwo-localization} gives
$(r_{i+1}^{(t-1)}+h_z)\partial_i\Phi$.
Apply \cref{lem:time-indexed-pushforward} with this pulled-back
factor, $T_t=\overline\Pi_{i_t}^{(z)}$, and $\Phi=x^\mu$.
The terminal monomial is \eqref{eq:terminal-monomial}, and
$p_{\ell,t}^*\pi_t^*=p_{\ell,t-1}^*$ retains its historical factor.
\end{proof}

Let $\widehat{\cC}_a$ denote the pullback of $\cC_a$ from $Y_a$ to $X_u$.

The co-row filtrations assemble these historical factors into the required generated bundle.
\begin{proposition}\label{thm:atom-top-chern}
The bundle
\begin{equation}\label{eq:atom-bundle}
 \cW_{u,\mu}^{\mathrm{at}}
 :=\bigoplus_{a:c_a>0}(\widehat{\cC}_a\otimes\cM)
 \oplus\bigoplus_{j=1}^n(\cF_j^{(\ell)})^{\oplus d_j}
\end{equation}
is globally generated and satisfies
\begin{equation}\label{eq:atom-top-chern}
 \rho_*\ctop(\cW_{u,\mu}^{\mathrm{at}})
 =\overline\Pi_u^{(z)}x^\mu(h,h_z).
\end{equation}
Its rank is $\ell+|\mu|$.
\end{proposition}

\begin{proof}
By \cref{lem:old-upper}, the old upper roots appearing in
\eqref{eq:atom-chow} during row $B_a$ are the pullbacks of
$\cR_{a+1}^{\langle a\rangle},\ldots,
\cR_{a+c_a}^{\langle a\rangle}$.  Hence \cref{prop:corow} gives
\[
 \prod_{t=b_a}^{b_a+c_a-1}(r_{i_t+1}^{(t-1)}+h_z)
 =c_{c_a}(\widehat{\cC}_a\otimes\cM).
\]
The terminal factors are top Chern classes of globally generated final
prefix bundles.  Multiplicativity under direct sum proves
\eqref{eq:atom-top-chern}.  Finally,
\[
 \rk\cW_{u,\mu}^{\mathrm{at}}
 =\sum_ac_a+\sum_jjd_j=\ell+|\mu|.
\]
\end{proof}

Incidence and finite complementation then give the normalized atom packet.
\begin{proposition}\label{thm:atom-volume}
For every partition $\lambda$ and every $w\in S_n$,
\begin{equation}\label{eq:atom-rvc}
 \Nrm_{x,z}\bigl(\overline H_{w,\lambda}(x,z)\bigr)\in\RV.
\end{equation}
If $\operatorname{char}(\kk)=0$ and $\overline H_{w,\lambda}$ is nonzero, its
factorial normalization admits a smooth integral projective realization.
\end{proposition}

\begin{proof}
By \cref{thm:atom-top-chern,thm:incidence-dual}, the normalized
finite complement of $\overline\Pi_u^{(z)}x^\mu$ in the box
$(M,\ldots,M;\ell)$ lies in $\RV$.
By \cref{prop:atom-normal} this complement is
$\omega\Nrm_{x,z}(\overline H_{w,\lambda})$.
Variable reversal preserves $\RV$, and \cref{prop:RV-closure} gives
smooth realization in characteristic zero.
\end{proof}

Let $\alpha$ be a weak composition, put $\lambda=\alpha^+$, and set
$w=w_\alpha$, using the same minimal sorting permutation.  Define
\begin{equation}\label{eq:positive-atom}
 \overline L_\alpha^+(x;\beta)
 :=\bar\varpi_w^{(\beta)}x^\lambda
 =\sum_{k=0}^{r_\alpha^{\mathrm{at}}}
 \overline L_{\alpha,k}(x)\beta^k.
\end{equation}
Every $\overline L_{\alpha,k}$ is homogeneous of degree $|\lambda|+k$, and
\begin{equation}\label{eq:atom-packet-expansion}
 \overline H_{w,\lambda}(x,z)
 =z^{\length(w)}\overline L_\alpha^+(x;z^{-1})
 =\sum_k\overline L_{\alpha,k}(x)z^{\length(w)-k}.
\end{equation}
The ordinary Lascoux atom is
$\overline L_\alpha(x)=\overline L_\alpha^+(x;-1)$.
Its constant $\beta$-coefficient is the nonzero Demazure atom
$\mathcal A_\alpha$ \cite[\S2.2]{BSW20}. Thus the maximal excess
\[
 r_\alpha^{\mathrm{at}}:=\deg_\beta\overline L_\alpha^+(x;\beta)
   =\max\{k:\overline L_{\alpha,k}\ne0\}
\]
is well defined. Each atom operator is affine in $\beta$, so
$0\le r_\alpha^{\mathrm{at}}\le\length(w)$. Set
$r=r_\alpha^{\mathrm{at}}$, regard layers outside $0\le k\le r$ as zero,
and define the minimal homogeneous atom packet by
\begin{equation}\label{eq:minimal-atom-packet}
 \widetilde{\overline L}_\alpha(x,z)
 =\sum_{k=0}^r\overline L_{\alpha,k}(x)z^{r-k}.
\end{equation}

\subsection{Positive Grothendieck packets and Schubert polynomials}\label{sec:groth}

The Schubert and Grothendieck cases require the full rooted-flow mechanism.
The staircase terminal weight supplies exactly one root reservoir in every
nontrivial rank.

Recall the operators
\begin{equation}\label{eq:beta-groth-operator}
 \mathfrak d_i^{(\beta)}f
 =\partial_i\bigl((1+\beta x_{i+1})f\bigr),
\end{equation}
which satisfy the braid and distant-commutation relations
\cite{LS82,FK94}.  With
$\delta=(n-1,n-2,\ldots,1,0)$ and $v=w^{-1}w_0$, define
\begin{equation}\label{eq:positive-beta-groth}
 \Groth_w^+(x;\beta)
 :=\mathfrak d_v^{(\beta)}x^\delta.
\end{equation}
This is the positive-parameter Grothendieck polynomial.  It has an expansion
\begin{equation}\label{eq:positive-groth-expansion}
 \Groth_w^+(x;\beta)=\sum_{k=0}^{r_w}A_{w,k}(x)\beta^k,
 \qquad A_{w,k}=(-1)^kG_w^{(\length(w)+k)},
\end{equation}
with nonnegative integral coefficients. This is the coefficient sign
alternation in the pipe-dream expansion of Grothendieck polynomials
\cite{FK94}; compare the conventions in \cite[\S5.1]{MTY19}.

Put
\begin{equation}\label{eq:L-and-Fw}
 L=\length(v)=\length(w_0)-\length(w),
 \qquad
 F_w(x,z)=z^L\Groth_w^+(x;z^{-1}).
\end{equation}
Since
$z\,\mathfrak d_i^{(z^{-1})} =\partial_i\bigl((x_{i+1}+z)\,\cdot\,\bigr) =\mathscr D_i^{(z)}$,
we have the homogeneous operator formula
\begin{equation}\label{eq:Fw-D-word}
 F_w(x,z)=\mathscr D_v^{(z)}x^\delta.
\end{equation}
Every term of $F_w$ has total degree $\length(w_0)$.  Moreover,
\cref{lem:finite-box} shows that $F_w$ is supported in the finite box
$(n-1,\ldots,n-1;L)$.  Finally,
\begin{equation}\label{eq:Fw-layers}
 F_w(x,z)=\sum_{k=0}^{r_w}A_{w,k}(x)z^{L-k}.
\end{equation}
In particular, $r_w\le L$.

Define
\begin{equation}\label{eq:Gamma-groth}
 \Gamma_i^{(z)}f
 :=\partial_i\bigl(x_{i+1}(x_i+z)f\bigr)
 =\Gamma_i^{(1,z)}f.
\end{equation}

Let $M=n-1$ and put
\begin{equation}\label{eq:u-groth}
 u=w_0vw_0.
\end{equation}

The staircase is fixed by complemented reversal, so the terminal weight remains unchanged.
\begin{proposition}\label{prop:groth-reciprocal}
One has
\begin{equation}\label{eq:groth-reciprocal-normal-form}
 \omega\mathcal R_{M,L}F_w=\Gamma_u^{(z)}x^\delta.
\end{equation}
The operator on the right may be evaluated along the canonical-row word of
$u$.
\end{proposition}

\begin{proof}
Apply \eqref{eq:capped-word} with $T_i=\mathscr D_i^{(z)}$ and initial $z$-cap zero to a reduced
word for $v$. By \eqref{eq:local-D-reversal}, conjugation by $\omega$
changes each factor $-\mathscr K_i^{(z)}$ to $\Gamma_{n-i}^{(z)}$.
The transformed word represents $u=w_0vw_0$, and
$\omega\mathcal R_{M,0}x^\delta=x^\delta$.
Use \cref{lem:reciprocal-transport} to choose its canonical-row word.
\end{proof}

The staircase reservoir criterion now supplies the missing positivity in the full Grothendieck construction.
\begin{proposition}\label{thm:groth-packet}
For every $w\in S_n$,
\begin{equation}\label{eq:groth-packet-RV}
 \Nrm_{x,z}(F_w)\in\RV.
\end{equation}
If $\operatorname{char}(\kk)=0$ and $F_w$ is nonzero, its factorial
normalization admits a smooth integral projective realization.
\end{proposition}

\begin{proof}
Use \cref{thm:canonical-packet} for the canonical-row word of $u$,
with $\epsilon_t=1$ and $\mu=\delta$.
The terminal differences $d_j=1$ for $1\le j<n$ satisfy the exact
reservoir criterion of \cref{thm:reservoir}.
Thus the normalized finite complement of $\Gamma_u^{(z)}x^\delta$
in the box $(M,\ldots,M;L)$ lies in $\RV$. By
\cref{prop:groth-reciprocal},
\begin{equation}\label{eq:groth-finite-dual-identity}
 \Dcomp_{(M,\ldots,M;L)}
 \Nrm\bigl(\omega\mathcal R_{M,L}F_w\bigr)
 =\omega\Nrm(F_w).
\end{equation}
Variable reversal proves the assertion; \cref{prop:RV-closure}
gives its characteristic-zero refinement.
\end{proof}

Let
\begin{equation}\label{eq:rw-def}
 r_w=\max\{k:A_{w,k}\ne0\}.
\end{equation}
The minimal sign-corrected homogeneous Grothendieck polynomial is
\begin{equation}\label{eq:tilde-G-def}
 \widetilde G_w(x,z)
 =z^{r_w}\Groth_w^+(x;z^{-1})
 =\sum_{k=0}^{r_w}A_{w,k}(x)z^{r_w-k}.
\end{equation}
By \eqref{eq:L-and-Fw},
\begin{equation}\label{eq:F-vs-tildeG}
 F_w=z^{L-r_w}\widetilde G_w.
\end{equation}
The common extraction formulas below remove this padding in divided-power
coordinates.

Recall that the ordinary Grothendieck polynomial is
\begin{equation}\label{eq:ordinary-groth-sign-layers}
 G_w(x)=\sum_{k=0}^{r_w}(-1)^kA_{w,k}(x),
\end{equation}
where $A_{w,k}$ is homogeneous of degree $\length(w)+k$.

\begin{remark}\label{warn:direct-schubert}
One cannot delete the common factors $q_t+h_z$ first and then attempt to
organize the remaining $q_ts_t$ factors into unrooted Chern paths.  The
smallest obstruction occurs for $u=312$, whose canonical word is $(2,1)$:
the proposed grouping produces a non-root-started path from the initial
rank-$1$ state toward the initial rank-$3$ state, while the integrand contains
no initial factor $c_1(\cF_1^{(0)})$.  The valid proof of
\cref{cor:packet-layers} first treats the full positive Grothendieck packet,
where the staircase terminal Chern factors supply the root reservoirs, and
then extracts the highest $z$-layer.
\end{remark}

\subsection{Common layer extraction and ordinary supports}
\label{subsec:packet-extraction}
All three conversions are now complete. Their remaining consequences
use the same divided-power identities, with the following dictionary.
For the first two rows, $\lambda=\alpha^+$ and $w=w_\alpha$.
\begin{center}
\small
\begin{tabular}{@{}llll@{}}
\toprule
Packet $Q$ & Layer $P_k$ & Initial degree $d$ & Padding $L$\\
\midrule
$H_{w,\lambda}$ & $\Omega_{\alpha,k}$ & $|\lambda|$ & $\length(w)$\\
$\overline H_{w,\lambda}$ & $\overline L_{\alpha,k}$ & $|\lambda|$ & $\length(w)$\\
$F_w$ & $A_{w,k}$ & $\length(w)$ & $\length(w_0)-\length(w)$\\
\bottomrule
\end{tabular}
\end{center}
In each row,
\[
 Q(x,z)=\sum_{k=0}^r P_k(x)z^{L-k},\qquad
 P_k\in\QQ[x]_{d+k},\qquad
 r=\max\{k:P_k\ne0\}\le L.
\]
For arbitrary $w$ and $\lambda$, the same expansion applies to a nonzero
$H_{w,\lambda}$ or $\overline H_{w,\lambda}$ by defining
$P_k=[z^{L-k}]Q$, $d=|\lambda|$, and $L=\length(w)$.
Zero packets are already included in the realization statements and
require no maximal-excess convention. Put
$\widetilde P=\sum_k P_kz^{r-k}$.

The highest homogenizing terms recover the classical operators:
\begin{equation}\label{eq:classical-highest-layers}
 \begin{gathered}
 [z^{\length(w)}]H_{w,\lambda}=\pi_wx^\lambda,\qquad
 [z^{\length(w)}]\overline H_{w,\lambda}=(\pi-\mathrm{id})_wx^\lambda,\\
 [z^L]F_w=\Sch_w,\qquad L=\length(w_0)-\length(w).
 \end{gathered}
\end{equation}
Here $(\pi-\mathrm{id})_w$ is the reduced-word product of atom
operators. For $\lambda=\alpha^+$ and $w=w_\alpha$, the first two
right-hand sides are $\kappa_\alpha$ and $\mathcal A_\alpha$.

\begin{corollary}\label{cor:packet-layers}
For every nonzero packet above, its normalized minimal packet and all
normalized layers belong to $\RV$ for every field $\kk$. The same is
true for every larger homogenizing padding. In particular,
\begin{equation}\label{eq:classical-volume}
 \Nrm(\kappa_\alpha),\qquad \Nrm(\mathcal A_\alpha),\qquad
 \Nrm(\Sch_w)\quad\text{belong to }\RV.
\end{equation}
Every nonzero normalized polynomial so obtained is Lorentzian and has
support equal to the integral base set of a polymatroid algebraic over
$\kk$. Its full polarization has the bases of an algebraic matroid as
support. In characteristic zero it has a smooth integral projective
realization.
\end{corollary}
\begin{proof}
Apply \eqref{eq:common-minimal-packet} and
\eqref{eq:common-layer-extraction} to
\cref{thm:lascoux-RV,thm:atom-volume,thm:groth-packet}.
The expansions above identify the ordinary layers; for arbitrary
$w,\lambda$, homogeneity and \cref{lem:finite-box} give the same
extraction. Taking the $z\pi_i$, $z(\pi_i-\mathrm{id})$, and
$z\partial_i$ terms proves \eqref{eq:classical-highest-layers} and
hence \eqref{eq:classical-volume}. Apply \cref{prop:RV-closure}
for padding and smooth realization, and
\cref{thm:algebraic-polymatroid-consequence} for supports and
polarization. Choosing the complex realization gives Lorentzianity
\cite[Theorem~4.6]{BH20}.
\end{proof}

Straight flagged Schur polynomials provide a further classical application.
For straight flagged Schur shapes, the empty shape gives $1$.
Otherwise, let $\lambda$ have $r$ nonempty rows and weakly increasing
flag $\mathbf b$ with $b_i\ge i$. Choose $n\ge\max\{r,b_r\}$,
pad $\lambda$ by zeros, and put $b_i=n$ for $i>r$.
The correspondence preceding \cite[Theorem~14.1]{PostnikovStanley09}
gives a $312$-avoiding $w\in S_n$ with
$s_\lambda(\mathbf b;x)=\pi_wx^\lambda$. Thus
\cref{cor:packet-layers} realizes its factorial normalization; see also
\cite{RS95}. This does not extend the assertion to arbitrary flagged
skew shapes or diagram-indexed modules.

\begin{remark}\label{rem:schubert-volume-covolume}
Huh records realizable covolume for $\Sch_w(\partial)$ over every
field \cite{Huh26}. Our \cref{cor:packet-layers} instead realizes
$\Nrm(\Sch_w)$ as volume. These are distinct assertions, not formal
renamings of one another.
\end{remark}

The same homogeneous support controls the ordinary signed polynomial.
\begin{corollary}\label{cor:ordinary-packet-support}
Let $P$ be an ordinary Lascoux polynomial, Lascoux atom, or Grothendieck
polynomial above, let $B=\Supp(P)$, and let $D$ be its maximal total
degree. Then
\[
 \widehat B_D=\{(\alpha,D-|\alpha|):\alpha\in B\}
\]
is $M$-convex. Hence $B$ is $M^\natural$-convex and
\begin{equation}\label{eq:ordinary-packet-support}
 B=\operatorname{conv}(B)\cap\ZZ^n,
\end{equation}
where $\operatorname{conv}(B)$ is an integral generalized polymatroid.
Each homogeneous layer is the full lattice-point set of its integral
constant-sum section. Moreover, every nonmaximal-degree $\alpha\in B$
has $\alpha+e_i\in B$ for some $i$, and
$\alpha\le\beta\le\gamma$ with integral $\beta$ and
$\alpha,\gamma\in B$ implies $\beta\in B$.
The positive-parameter polynomial $\sum_kP_k(x)\beta^k$ also has
saturated Newton polytope in $\ZZ^{n+1}$.
\end{corollary}
\begin{proof}
The normalized minimal packet has $M$-convex support by
\cref{cor:packet-layers}. Its support is $\widehat B_D$, because the
layer determines total $x$-degree and distinct layers cannot cancel in
$\sum_k(-1)^kP_k$. Apply \cref{lem:graded-packet} for saturation and
the parameter assertion. Constant-sum sections of integral generalized
polymatroids are integral base polytopes \cite{Murota03}.
Exchange in the slack coordinate against a maximal-degree lift gives
augmentation. Finally, in a presentation
$\operatorname{conv}(B)=\{x:p(S)\le x(S)\le b(S)\ (S\subseteq[n])\}$,
coordinatewise betweenness gives
$p(S)\le\alpha(S)\le\beta(S)\le\gamma(S)\le b(S)$ for every $S$;
saturation then gives $\beta\in B$.
\end{proof}
Together with the Demazure-atom case of \cref{cor:packet-layers}, this
proves \cite[Conjectures~3.14 and~5.5--5.7]{MTY19}. For Grothendieck
polynomials, the saturation, augmentation, interval, and polytope
assertions give \cite[Conjectures~1.1--1.4]{MSSD25}.
The full packet also compares coefficients in different excess layers;
these inequalities are recorded in \cref{subsec:hodge-consequences}.

\begin{proof}[Proof of \cref{thm:main}]
The full-packet realizations are
\cref{thm:lascoux-RV,thm:atom-volume,thm:groth-packet}.
Their layer, algebraicity, smoothness, and Lorentzian consequences are
\cref{cor:packet-layers}; their ordinary support consequences are
\cref{cor:ordinary-packet-support}.
\end{proof}

\subsection{The example \texorpdfstring{$w=2143$}{w=2143}}
\label{app:worked-example}
For $w=2143\in S_4$, let $\delta=(3,2,1,0)$. Then
$v=w^{-1}w_0=u=w_0vw_0=3412$ has inverse-Lehmer word
$\mathbf r(u)=(2,3\mid1,2)$. Over any field, this example connects
reciprocity, the tower, flow, and incidence.

With rightmost operators acting first,
\begin{align}
 F_{2143}
 &=\mathscr D_2^{(z)}\mathscr D_3^{(z)}
   \mathscr D_1^{(z)}\mathscr D_2^{(z)}x^\delta\notag\\
 &=x_1z^2\bigl[x_1x_2x_3+z(x_1x_2+x_1x_3+x_2x_3)
             +z^2(x_1+x_2+x_3)\bigr].
 \label{eq:worked-source-packet}
\end{align}
For $\mathcal R_{3,4}f=(x_1x_2x_3x_4)^3z^4f(x^{-1},z^{-1})$,
\cref{app:operator-conjugation} gives
\begin{align}
 P:=\omega\mathcal R_{3,4}F_{2143}
 &=\Gamma_2^{(1,z)}\Gamma_3^{(1,z)}
   \Gamma_1^{(1,z)}\Gamma_2^{(1,z)}x^\delta\notag\\
 &=x_1^3x_2^2x_3^2x_4
 \bigl[x_2x_3+x_2x_4+x_3x_4+z(x_2+x_3+x_4)+z^2\bigr].
 \label{eq:worked-reciprocal-packet}
\end{align}
Their degrees are six and ten. Finite complementation gives
\begin{equation}\label{eq:worked-finite-complement}
 \Dcomp_{(3,3,3,3;4)}\Nrm_{x,z}(P)
 =\omega\Nrm_{x,z}(F_{2143}).
\end{equation}

The same calculation can be followed on the quotient-flag tower.
Use $B=(\PP^3)^4\times\PP^4$ and the four tower modifications of
colours $(2,3,1,2)$. Thus $\rho:X_u\to B$ has
$\dim B=16$ and $\dim X_u=20$. The complete rows, not their
individual quotient lines, give generated bundles satisfying
\begin{align}
 (q_1+h_z)(q_2+h_z)&=c_2(\cG_2\otimes\cM),\notag\\
 (q_3+h_z)(q_4+h_z)&=c_2(\cG_1\otimes\cM).
 \label{eq:worked-row-bundles}
\end{align}
For $\epsilon=(1,1,1,1)$ and $\mu=\delta$, the terminal differences
are $(1,1,1,0)$. The creation states of
\cref{ex:creation-state-3412} give exact sequences
$0\to \cS_t\to \cF_{w_t}\to \cF_{v_t}\to0$, where
$\cF_{v_{j,t}}=\widehat{\cF}_{j,t}$ and
\[
 \begin{aligned}
 (w_1,v_1)&=(v_{3,0},v_{2,1}), & (w_2,v_2)&=(v_{4,0},v_{3,2}),\\
 (w_3,v_3)&=(v_{2,1},v_{1,3}), & (w_4,v_4)&=(v_{3,2},v_{2,4}).
 \end{aligned}
\]
Only $v_{3,0},v_{3,2},v_{4,0}$ have nonzero divergence, each equal to
one. Hence $\cU=\widehat{\cF}_{3,0}\oplus\widehat{\cF}_{3,2}
\oplus\widehat{\cF}_{4,0}$ has rank ten and
\begin{equation}\label{eq:worked-flow-identity}
 s_1s_2s_3s_4\,
 c_1(\widehat{\cF}_{1,3})c_2(\widehat{\cF}_{2,4})c_3(\widehat{\cF}_{3,2})
 =c_{10}(\cU).
\end{equation}
The generated bundle
$\cV=(\cG_2\otimes\cM)\oplus(\cG_1\otimes\cM)\oplus\cU$ has rank fourteen.
The four rank-two pushforwards yield
\begin{equation}\label{eq:worked-topchern-push}
 \rho_*c_{14}(\cV)=P(h_1,h_2,h_3,h_4,h_z),
 \qquad 14-4=10=\deg P.
\end{equation}

It remains to extract the six-dimensional volume polynomial from the incidence construction.
Choose a generating presentation $W\otimes\cO_{X_u}\twoheadrightarrow \cV$
with positive-rank kernel $\cK$, adding a trivial summand to $W$ if needed.
Put $M_W=\dim\PP(W^\vee)$, $p:Z=\PP_{X_u}(\cK^\vee)\to X_u$,
$\widetilde\rho=\rho p$, and $H=c_1(\cO_Z(1))$.
Then $Z$ is integral, $\dim Z=M_W+6$, and
\cref{thm:incidence-dual} gives
\begin{align}
 &\left.\partial_y^{M_W}\frac1{(\dim Z)!}
 \int_Z\left(\sum_{i=1}^4x_i\widetilde\rho^*h_i
             +z\widetilde\rho^*h_z+yH\right)^{\dim Z}
 \right|_{y=0}\notag\\
 &\hspace{20mm}=\Dcomp_{(3,3,3,3;4)}\Nrm_{x,z}(P)
 =\omega\Nrm_{x,z}(F_{2143}).
 \label{eq:worked-incidence-minor}
\end{align}
The divisors are semiample, so derivative, specialization, and variable
reversal closure give $\Nrm_{x,z}(F_{2143})\in\RV$ in degree six.
Its three layers are
\begin{align*}
 [z^4]F_{2143}&=x_1(x_1+x_2+x_3)=\Sch_{2143},\\
 [z^3]F_{2143}&=x_1(x_1x_2+x_1x_3+x_2x_3)=A_{2143,1},\\
 [z^2]F_{2143}&=x_1^2x_2x_3=A_{2143,2}.
\end{align*}
Divided-power differentiation extracts their factorial normalizations
without additional factorials.

\part{Chern moments and Hodge--Riemann relations}\label{part:hodge}

\section{Chern moment algebras and Hodge--Riemann relations}\label{reorg:hodge}
The volume constructions record scalar intersection numbers. We now
retain their coefficient algebra and the Chern classes themselves.
The intrinsic moment quotient requires only Poincar\'e duality;
positivity enters through an inverse-Chern projective-bundle
presentation. We separate these algebraic and positive inputs before
constructing the geometric certificate.

\subsection{Lefschetz descent}\label{sec:moment-algebras}

All algebras here are over $\RR$, graded by codimension, with ordinary
commutative multiplication. A Poincar\'e-duality algebra of formal
dimension $d$ is a finite-dimensional graded algebra
$A=\bigoplus_{j=0}^dA^j$, with $A^0=\RR$, a degree isomorphism
$\deg_A:A^d\to\RR$, and perfect pairings
$(a,b)\mapsto\deg_A(ab)$ between complementary degrees. Extend degree
by zero elsewhere.

\begin{definition}\label{hd-def:package}
An open nonempty convex cone $\mathcal K_A\subseteq A^1$ is a
\emph{Lefschetz cone} if, for $L\in\mathcal K_A$ and $0\le j\le d/2$,
$L^{d-2j}:A^j\to A^{d-j}$ is an isomorphism and
$(-1)^j\deg_A(abL^{d-2j})$ is positive definite on
$\ker(L^{d-2j+1}:A^j\to A^{d-j+1})$.
We call this the K\"ahler package. Classes in
$\overline{\mathcal K_A}$ are \emph{boundary polarizations}; this
terminology does not assert geometric nefness without a realization.
\end{definition}

We use the classical boundary descent lemma in the formulation
of \cite[\S4.3, Lemma~4.6]{AHL26}. The iterated statement follows
by successive quotients.
\begin{proposition}\label{hd-thm:descent}
Let $(A,\deg_A,\mathcal K_A)$ have the K\"ahler package of formal
dimension $d$. For $0\ne\gamma\in\overline{\mathcal K_A}$, the quotient
$A_\gamma=A/\Ann_A(\gamma)$, with degree
$\deg_\gamma([a])=\deg_A(\gamma a)$, has the induced K\"ahler package
and formal dimension $d-1$. More generally, for
$0\ne\eta=\gamma_1\cdots\gamma_p$ with
$\gamma_i\in\overline{\mathcal K_A}$,
\[
 A_\eta=A/\Ann_A(\eta),\qquad
 \deg_\eta([a])=\deg_A(\eta a)
\]
has formal dimension $d-p$ and the package for the image of
$\mathcal K_A$. For $p=0$, take $\eta=1$ and $A_\eta=A$.
\end{proposition}
\begin{proof}
The one-step assertion is \cite[\S4.3, Lemma~4.6, p.~26]{AHL26}, for
$M=A$ and $Q(a,b)=\deg_A(ab)$. Its pairing is
$\deg_A(\gamma ab)$; $\gamma\ne0$ gives degree-zero part $\RR$,
and the degree-one surjection makes the image cone open. Iterate using
$(A/\Ann_A\eta')/\Ann(\bar\gamma)\simeq A/\Ann_A(\eta'\gamma)$.
A nonzero final product makes every intermediate descent nonzero, and
projection preserves boundary membership. The empty product gives the
identity.
\end{proof}
Descent gives both the mixed signatures and the scalar degree polynomials.
\begin{corollary}\label{hd-cor:degree-polynomial}
Let $B$ have the K\"ahler package of formal dimension $n$.
Top products of boundary polarizations have nonnegative degree. For
$n\ge2$ and a product $\eta$ of $n-2$ boundary polarizations, the form
$(a,b)\mapsto\deg_B(\eta ab)$ on $B^1$ has at most one positive
eigenvalue. Consequently, for $D_1,\ldots,D_q\in\overline{\mathcal K_B}$,
\[
 V_D(x)=\frac1{n!}\deg_B\left(\Bigl(\sum_jx_jD_j\Bigr)^n\right)
\]
is Lorentzian, with zero allowed.
\end{corollary}
\begin{proof}
For interior factors, successive descent ends in a degree-zero polarized
algebra with positive degree of $1$; degree-zero Hodge--Riemann at
each step ensures nonvanishing. Approximation by
$\gamma+\varepsilon L$, $L\in\mathcal K_B$, gives boundary
nonnegativity. If $\eta\ne0$, descent to $B/\Ann_B(\eta)$ gives a
surface pairing of signature $(1,\dim B_\eta^1-1)$, and pullback
cannot increase the positive index; if $\eta=0$, the form is zero.
For $n\ge2$, replace each $D_j$ by $D_j+\varepsilon L$. The degree
polynomial then has positive coefficients and full support, and each
quadratic derivative Hessian restricts one of the preceding forms.
Apply \cite[Theorem~2.25]{BH20} and closedness as
$\varepsilon\downarrow0$. Degrees zero and one use nonnegativity.
\end{proof}

\subsection{The intrinsic Chern moment algebra}
Fix a Poincar\'e-duality algebra $A$ of formal dimension $d$ and arrays
$C_i=(C_{i,0},\ldots,C_{i,d})$ with $C_{i,a}\in A^a$ and
$C_{i,0}=1$. Set $C_{i,a}=0$ for $a>d$ and $P=A[z_1,\ldots,z_m]$,
$\deg z_i=1$. For homogeneous $a\in A^j$, define
\begin{equation}\label{hd-eq:moment}
 \Lambda_C(az^\alpha)=
 \begin{cases}
 \deg_A(a\prod_iC_{i,\alpha_i}),&j+|\alpha|=d,\\
 0,&j+|\alpha|\ne d,
 \end{cases}
\end{equation}
and extend linearly. Put
\begin{equation}\label{hd-eq:intrinsic}
 I_C=\{p\in P:\Lambda_C(pq)=0\text{ for every }q\in P\},\qquad
 R_C=P/I_C.
\end{equation}
This multiplication-annihilator is not merely $\ker\Lambda_C$ and need
not be radical.

Duality and marked uniqueness are formal consequences of the moment
functional, before any positivity assumption.
\begin{proposition}\label{hd-prop:intrinsic-PD}
The ideal $I_C$ is homogeneous and contains $P^{>d}$. The quotient
$R_C$ is a Poincar\'e-duality algebra of formal dimension $d$, with
degree $\deg_{R_C}[p]=\Lambda_C(p)$ and an injective natural map
$A\to R_C$.
If a graded duality algebra $R$ of formal dimension $d$ is generated
over $A$ by marked $w_i\in R^1$ and has moments
$\deg_R(aw^\alpha)=\deg_A(a\prod_iC_{i,\alpha_i})$, then there is a
unique graded, degree-preserving isomorphism $R_C\simeq R$ fixing
$A$ and sending $z_i\mapsto w_i$.
\end{proposition}
\begin{proof}
Homogeneity gives $P^{>d}\subseteq I_C$ and finite dimension.
A nonzero $[p]\in R_C^j$ has $\Lambda_C(pq)\ne0$ for some
$q\in P^{d-j}$, giving perfect complementary pairings. Since
$\Lambda_C|_{A^d}=\deg_A\ne0$, degrees zero and $d$ are one dimensional.
For $a\ne0$ in $A$, duality in $A$ supplies a nonzero pairing, proving
injectivity. For the final assertion the generator surjection
$P\to R$ has functional $\Lambda_C$; perfect pairings in $R$ identify
its kernel with $I_C$. Generators force uniqueness. This is
multiplication-annihilator reconstruction
\cite[Theorems~7.15--7.16]{HKM24}, with even regrading.
\end{proof}

Choose $s_i\ge1$ and $k_{i,j}\in A^j$ with
\begin{equation}\label{hd-eq:inverse}
 K_i(t)C_i(t)=1,\qquad
 K_i(t)=1+\sum_{j=1}^{s_i}k_{i,j}t^j,\qquad
 C_i(t)=\sum_{a=0}^dC_{i,a}t^a.
\end{equation}
The identity is in $A[t]$, equivalently through degree $d$ since
$A^{>d}=0$. Such finite inverses always exist algebraically after zero
padding; positivity is additional. Define
\begin{equation}\label{hd-eq:projective-ring}
 B=\frac{A[\xi_1,\ldots,\xi_m]}
 {(\xi_i^{s_i}+k_{i,1}\xi_i^{s_i-1}+\cdots+k_{i,s_i})_{i=1}^m},
 \qquad \eta=\prod_i\xi_i^{s_i-1}.
\end{equation}
The monic relations give the $A$-basis $\xi^\alpha$, $0\le\alpha_i<s_i$.
Let $\pi_*$ extract the coefficient of $\eta$ and set
$\deg_B=\deg_A\circ\pi_*$ in degree $D=d+\sum_i(s_i-1)$.

The inverse Chern series is exactly the recurrence for extraction of the top relative monomial.
\begin{lemma}\label{hd-lem:recurrence}
The algebra $B$ has Poincar\'e duality, and
\begin{equation}\label{hd-eq:pushforward}
 \pi_*\left(\prod_i\xi_i^{s_i-1+\alpha_i}\right)
 =\prod_iC_{i,\alpha_i}\qquad(\alpha_i\ge0),
\end{equation}
with no alternating signs or factorial factors.
\end{lemma}
\begin{proof}
In one variable, $p_a=\pi_*(\xi^{s-1+a})$ satisfies $p_0=1$,
$p_a=0$ for $a<0$, and
$p_a+\sum_{j=1}^{\min(a,s)}k_jp_{a-j}=0$ for $a\ge1$, by the monic
relation. Thus $\sum_ap_at^a=K(t)^{-1}=C(t)$; independent extraction
proves the joint formula. In the basis $1,\xi,\ldots,\xi^{s-1}$, the $(a,b)$-entry of
the $\pi_*(uv)$ matrix is zero for $a+b<s-1$ and one for
$a+b=s-1$, hence the matrix is invertible after reversing columns. Tensoring these perfect
$A$-pairings and composing with duality on $A$ gives duality on $B$,
in complementary degrees $j,D-j$.
\end{proof}

This recurrence identifies the intrinsic algebra with a weighted quotient of every such presentation.
\begin{proposition}\label{hd-thm:quotient}
The assignment $z_i\mapsto\xi_i$ induces
\begin{equation}\label{hd-eq:quotient}
 R_C\simeq B/\Ann_B(\eta),\qquad
 \deg_{R_C}[b]=\deg_B(\eta b).
\end{equation}
This marked graded algebra and degree do not depend on the auxiliary
ranks or projective-bundle presentation.
\end{proposition}
\begin{proof}
For the surjection $\phi:P\to B$, \cref{hd-lem:recurrence} gives
$\Lambda_C(p)=\deg_B(\eta\phi(p))$. Hence $p\in I_C$ exactly when
$\deg_B(\eta\phi(p)b)=0$ for all $b\in B$, equivalently
$\eta\phi(p)=0$ by duality. This proves the isomorphism and, by the
intrinsic definition, independence of the presentation.
\end{proof}

To use Lefschetz descent, we impose a positivity condition on the presentation rather than on the formal recurrence alone.
\begin{definition}\label{hd-def:positive-certificate}
A \emph{positive inverse-Chern certificate} is a presentation
\eqref{hd-eq:projective-ring} with the K\"ahler package on an open
convex cone $\mathcal K_B$, such that every $\xi_i$ lies in its
closure. Specified base directions used in degree polynomials must
also lie in this closure.
\end{definition}

The quotient presentation now transports the Hodge--Riemann relations to the Chern moments.
\begin{proposition}\label{hd-thm:moment-Hodge}
A positive inverse-Chern certificate gives $R_C$ the K\"ahler package
of formal dimension $d$, for the image of its cone, with the $z_i$ as
boundary polarizations.
\end{proposition}
\begin{proof}
The fibre-top basis monomial $\eta$ is nonzero. Descend by $s_i-1$
copies of each $\xi_i$ using \cref{hd-thm:descent}, then identify the
quotient and degree by \cref{hd-thm:quotient}. Its dimension is
$D-\sum_i(s_i-1)=d$.
\end{proof}

The resulting degree polynomial records joint Chern and divisor intersections.
\begin{corollary}\label{hd-cor:mixed-Chern}
If $C$ has a positive certificate and $H_1,\ldots,H_q\in A^1$ are
specified boundary directions, then
\begin{equation}\label{hd-eq:joint-polynomial}
 \sum_{|\alpha|+|\beta|=d}
 \deg_A\left(\prod_iC_{i,\alpha_i}\prod_jH_j^{\beta_j}\right)
 \frac{x^\alpha y^\beta}{\alpha!\,\beta!}
\end{equation}
is Lorentzian, or zero.
\end{corollary}
\begin{proof}
This is the degree polynomial of $z_i,H_j$ in $R_C$; apply
\cref{hd-cor:degree-polynomial}.
\end{proof}

Assume $C_{i,a}=0$ for $a>r_i$; put $R=\sum_i r_i$ and
$\theta=\prod_iC_{i,r_i}$.
At the top Chern index of every bundle, all relative directions disappear and only a weighted base algebra remains.
\begin{corollary}\label{hd-thm:terminal}
If $C$ has a positive certificate and $\theta\ne0$, then
$A/\Ann_A(\theta)$, with degree $[a]\mapsto\deg_A(\theta a)$,
has the K\"ahler package of formal dimension $d-R$. Its cone is
obtained by projecting the certificate cone, with the relative
generators acting as zero.
\end{corollary}
\begin{proof}
The out-of-range moments give $z_i^{r_i+1}=0$: each such
monomial pairs to zero with every element of $A[z]$. For
$\zeta=\prod_i z_i^{r_i}$, pairings against $a z^\beta$ vanish
when $\beta\ne0$ and equal $\deg_A(\theta a)$ when $\beta=0$.
Duality on $A$ therefore gives
$\zeta\ne0\Longleftrightarrow\theta\ne0$. Descend by the factors of
$\zeta$. All $z_i$ act as zero in $R_C/\Ann(\zeta)$, so $A$
surjects onto it. Its kernel consists of $a$ with
$\deg_A(\theta ab)=0$ for every $b\in A$, namely $\Ann_A\theta$.
The degree is as stated, and nonvanishing implies $R\le d$.
\end{proof}

\subsection{Geometric certificates and apolar reconstruction}\label{sec:geometric-lefschetz}

For a smooth integral complex projective $d$-fold $X$, use
$A_X^j=H^{2j}(X,\RR)\cap H^{j,j}(X,\CC)$, integration, and the
K\"ahler cone in $A_X^1$. Classical Lefschetz decomposition restricts
to these real diagonal Hodge classes, giving perfect pairings and the
K\"ahler package. We do not assume duality for the Chow ring of an
arbitrary smooth projective variety.

Generating sections provide a positive certificate for actual Chern arrays.
\begin{proposition}\label{hd-thm:geometric}
For globally generated bundles $\cE_1,\ldots,\cE_m$ on $X$, the moment
algebra $R_X(\cE_1,\ldots,\cE_m)$ of $C_{i,a}=c_a(\cE_i)$ has the
K\"ahler package for the image of
\[
 \left\{L+\sum_i t_i z_i:L\in\mathcal K_X,\ t_i>0\right\}.
\]
Its marked algebra and degree depend only on the Chern arrays.
If $r_i=\rk\cE_i$ and $\theta=\prod_i c_{r_i}(\cE_i)\ne0$, then
$A_X/\Ann_{A_X}(\theta)$, with degree
$[a]\mapsto\int_X\theta a$, has the package in formal dimension
$d-\sum_i r_i$ for the image of $\mathcal K_X$.
\end{proposition}
\begin{proof}
Choose $W_i\otimes\cO_X\twoheadrightarrow \cE_i$ with locally free
kernel $\cK_i$, adding zero generators so $s_i=\rk \cK_i\ge2$.
Dualization makes $\cK_i^\vee$ generated. For the quotient projective
bundles, form
$p:Z=\prod_{i,X}\PP_X(\cK_i^\vee)\to X$ and
$\xi_i=c_1(\cO_i(1))$. This is smooth, integral, and projective.
The projective-bundle formula, compatible with Hodge types, gives
\begin{equation}\label{hd-eq:geometric-relation}
 A_Z=A_X[\xi_1,\ldots,\xi_m]/
 (\xi_i^{s_i}+c_1(\cK_i)\xi_i^{s_i-1}+\cdots+c_{s_i}(\cK_i))_i.
\end{equation}
Indeed $c_j(\cK_i^\vee)=(-1)^jc_j(\cK_i)$ cancels the alternating signs.
Furthermore,
\begin{equation}\label{hd-eq:geometric-series}
 c_t(\cK_i)c_t(\cE_i)=1,\qquad
 \sum_{a\ge0}h_a(\cK_i^\vee)t^a=c_{-t}(\cK_i^\vee)^{-1}=c_t(\cE_i),
\end{equation}
so $p_{i*}(\xi_i^{s_i-1+a})=c_a(\cE_i)$ without signs or factorials.

The evaluation quotient embeds each projective bundle in
$X\times\PP(W_i^\vee)$. A pulled-back Fubini--Study form represents
$\xi_i$, is semipositive, and is positive in its relative tangent
directions. Thus $p^*\omega+\sum_it_i\omega_i$, for
$[\omega]=L\in\mathcal K_X$ and $t_i>0$, is positive: the base detects
nonvertical vectors and a projective factor detects every nonzero
vertical vector. These classes form an open cone in $A_Z^1$ since
$s_i\ge2$ makes the relative directions independent. Its closure
contains the $\xi_i$ and base nef pullbacks. This is a positive
certificate; apply \cref{hd-thm:moment-Hodge}.
Intrinsicness follows from \cref{hd-prop:intrinsic-PD,hd-thm:quotient}.
The terminal assertion follows from \cref{hd-thm:terminal}; all
relative classes then act as zero, leaving the base cone.
\end{proof}
The arbitrary-field total-Chern realization is instead
\cref{thm:total-chern-arbitrary-field}; it does not use complex cohomology.

A scalar degree polynomial recovers the full duality algebra only if
its chosen directions generate that algebra.
\begin{proposition}\label{hd-thm:apolar}
Let $H$ be a graded Poincar\'e-duality algebra of formal dimension $D$,
generated by $H^1$, and choose $\ell_1,\ldots,\ell_N$ spanning $H^1$.
For $V_H(y)=\deg_H((\sum_i y_i\ell_i)^D)/D!$, there is a graded,
degree-preserving isomorphism
\[
 \RR[\partial_{y_1},\ldots,\partial_{y_N}]/\Ann(V_H)\simeq H,
 \qquad \partial_{y_i}\longmapsto\ell_i.
\]
\end{proposition}
\begin{proof}
Apply \cite[Theorem~9.3]{HKM24} to $\RR^N\to H^1$, $e_i\mapsto\ell_i$;
compare \cite[Theorem~1.1]{Kaveh11}. The factor $1/D!$ leaves the
annihilator unchanged and gives
$\partial^\alpha V_H=\deg_H(\ell^\alpha)$ for $|\alpha|=D$,
so the degree functional is preserved.
\end{proof}

In particular, if $A_X$ is degree-one-generated and
$\theta=c_R(\cU)\ne0$ for a generated bundle, adjoining K\"ahler
directions spanning $(A_X/\Ann\theta)^1$ to any displayed nef
directions gives a scalar polynomial whose apolar algebra is
$A_X/\Ann\theta$. Likewise, K\"ahler classes spanning $A_X^1$,
together with the $z_i$, recover the full joint moment algebra from
its degree polynomial. Both assertions follow from
\cref{hd-thm:apolar,hd-thm:geometric}; the open image of the
K\"ahler cone supplies the spanning directions.
These completions do not imply higher Hodge--Riemann for smaller
specializations, as \cref{ex:apolar-obstruction} shows. The packet
sources and their ample integral completions are treated explicitly
in \cref{sec:source-completions}.

An example makes the distinction between a Chern direction and a divisor class explicit.
For $X=\PP^2$, $h=c_1(\cO(1))$, and $\cE=\cO(1)^{\oplus2}$,
$A_X=\RR[h]/(h^3)$, $\deg(h^2)=1$, and $c_t(\cE)=1+2ht+h^2t^2$ give
\begin{equation}\label{hd-eq:P2-example}
 R_X(\cE)=\RR[h,z]/(h^3,hz-2h^2,z^2-h^2).
\end{equation}
The displayed quotient has Hilbert vector $(1,2,1)$ and pairing
$\left(\begin{smallmatrix}1&2\\2&1\end{smallmatrix}\right)$ in degree
one, of determinant $-3$. Its perfect pairing and defining moments
identify it with $R_X(\cE)$. For $L=h+z$, the primitive class $h-z$
has square $-2$. The inverse series is $1-2ht+3h^2t^2$; thus the
relations $\xi^s-2h\xi^{s-1}+3h^2\xi^{s-2}=0$, $s\ge2$, recover
the same moments algebraically, without asserting positive evaluation
kernels of every rank $s$.

In contrast, $\xi^2+h\xi=0$ has inverse array $1,-h,h^2$, giving
$x^2/2-xy+y^2/2$. The negative coefficient rules out a certificate
with both base and relative directions boundary-polarized. Monicity
alone is not positivity.

\subsection{Chern-flow networks}\label{sec:Hodge-applications}
\label{subsec:flow-Hodge}
The geometric certificate applies to the generated envelope, even when
the individual kernel bundles are not generated. The same argument
works for actual integral relations in $K^0(X)$.
\begin{proposition}\label{ec-thm:relations}
Let $X$ be integral projective, let $\cF_v,\cK_e$ be actual bundles,
and suppose $T=(t_{ve})\in\ZZ^{V\times E}$ satisfies
$[\cK_e]=\sum_vt_{ve}[\cF_v]$ in $K^0(X)$.
For $a,b\in\NN^V$, $\kappa\in\NN^E$, and $b=a+T\kappa$, put
$\cU=\bigoplus_v\cF_v^{\oplus b_v}$. Then
\begin{equation}\label{ec-eq:Crelation}
 c_t(\cU)=\prod_vc_t(\cF_v)^{a_v}
                  \prod_ec_t(\cK_e)^{\kappa_e}.
\end{equation}
For $R=\rk\cU$, its top coefficient is
\begin{equation}\label{hd-eq:network-top}
 \Theta=\prod_vc_{\rk\cF_v}(\cF_v)^{a_v}
                 \prod_ec_{\rk\cK_e}(\cK_e)^{\kappa_e}.
\end{equation}
If $\cU$ is generated, it supplies an incidence and total-Chern
realization. If also $X$ is smooth complex projective, its moment
algebra has the K\"ahler package; when $\Theta\ne0$, so does
$A_X/\Ann_{A_X}(\Theta)$, with degree $[a]\mapsto\int_X\Theta a$
and formal dimension $\dim X-R$.
\end{proposition}
\begin{proof}
Summation gives
$[\cU]=\sum_va_v[\cF_v]+\sum_e\kappa_e[\cK_e]$.
Apply the total Chern operation as in \cref{thm:chern-flow}.
Equality of ranks makes degree $R$ the product of the top classes
of the actual bundles. The remaining assertions follow from
\cref{thm:incidence-dual,thm:total-chern-arbitrary-field,hd-thm:geometric}.
\end{proof}
For a graph, the edge sequences give
$[\cK_e]=[\cF_{t(e)}]-[\cF_{s(e)}]$ and $T=B_\Gamma$,
with incoming sign positive. Thus the proposition applies to every
generated Chern-flow envelope, independently of path decomposition.
Equalities of ranks alone do not establish these $K^0$ relations.

\begin{remark}
There is one common Chern parameter in
\eqref{ec-eq:Crelation}.  It cannot be replaced by independent
parameters on local factors without another positive certificate.
For the evaluation sequence
$0\to\cO_{\PP^1}(-1)\to\cO_{\PP^1}^{\oplus2}
\to\cO_{\PP^1}(1)\to0$, the common-parameter product is $1$.
With separate parameters it is
$(1+uh)(1-vh)=1+(u-v)h$, whose integral is $u-v$.
The graph supplies an exact positive envelope; it is not itself a Hodge
structure and does not make each local factor positive.
\end{remark}

\subsection{Canonical-row source completions}\label{sec:source-completions}
Work over $\CC$, with the real diagonal Hodge algebra $A_X$ and
K\"ahler cone of \cref{sec:geometric-lefschetz}. Adding spanning source
directions completes the coefficient polynomial without changing it
on the original variables. The polynomial depends on those directions;
the weighted algebra for fixed $(X,\cV)$ does not.

\begin{corollary}
\label{hd-cor:canonical-completion}
Let $\rho:X\to B=\prod_{i=1}^n\PP^{m_i}_{\CC}$ be a canonical-row
tower with generated packet bundle $\cV$ of rank $R$.
Assume $\Xi=\rho_*c_R(\cV)\ne0$ and put
$\theta=c_R(\cV)$, $D=\dim X-R$, and $D_i=\rho^*h_i$. Then
\[
 f(x)=\Dcomp_{\mathbf m}\Nrm(\Xi)
 =\frac1{D!}\int_X\theta\left(\sum_i x_iD_i\right)^D.
\]
There are ample integral classes $L_1,\ldots,L_q$ such that
\[
 \widetilde f(x,y)=\frac1{D!}\int_X\theta
       \left(\sum_i x_iD_i+\sum_jy_jL_j\right)^D
\]
satisfies $\widetilde f(x,0)=f(x)$, belongs to $\RVC$, and has
$A_{\widetilde f}\simeq A_X/\Ann_{A_X}(\theta)$ as degree algebras,
with $\partial_{x_i}\mapsto[D_i]$ and $\partial_{y_j}\mapsto[L_j]$.
Thus $\widetilde f$ is Lorentzian and its apolar algebra has the full
K\"ahler package. For $D=0$, no extra directions are needed.
\end{corollary}
\begin{proof}
The projective-bundle formula makes $A_X$ divisor-generated with only
diagonal Hodge types. Finite-box pairing gives the formula for $f$
and detects $\theta\ne0$. Append ample integral classes to the
displayed divisors so that their combined images span
$(A_X/\Ann_{A_X}(\theta))^1$. To do so, add large multiples of a fixed
ample integral class to integral spanning classes and include the fixed
ample class.
Then \cref{hd-thm:geometric,hd-thm:apolar} identify the
apolar quotient and degree. The classes $D_i,L_j$ are boundary
polarizations, so \cref{hd-cor:degree-polynomial} gives Lorentzianity;
specialization gives $f$. For actual volume, retain these semiample
classes on the incidence projective bundle of $\cV$ and use
\eqref{eq:incidence-volume-extraction}, followed by derivative and
restriction closure, also in degree zero.
\end{proof}
Reciprocity identifies $f$, up to variable reversal, with the normalized
ordinary packet; no second normalization is applied. The apolar assertion
is for the completion, not every smaller specialization.

Differentiation in the homogenizing variable has an interpretation on the same source.
\begin{corollary}
\label{cor:canonical-Hodge-channels}
For $P=H_{w,\lambda}$, $\overline H_{w,\lambda}$, or $F_w$, use its
canonical source, with base divisors relabelled by reciprocal reversal:
\begin{equation}\label{eq:packet-source-Hodge}
 \begin{gathered}
 f(x,z)=\Nrm_{x,z}(P)
 =\frac1{D!}\int_X\theta\left(\sum_i x_iD_i+zH_z\right)^D,\\
 \theta=c_{\rk\cV}(\cV),\qquad D=\dim X-\rk\cV.
 \end{gathered}
\end{equation}
For $0\le q\le D$ with $\theta_q=\theta H_z^q\ne0$,
$\mathcal H_q=A_X/\Ann_{A_X}(\theta_q)$, with degree
$[a]\mapsto\int_X\theta_q a$, has the K\"ahler package in formal
dimension $D-q$ for the image of the K\"ahler cone of $X$.
Its displayed divisor polynomials are
\begin{align}
 \partial_z^q f(x,z)
 &=\frac1{(D-q)!}\int_X\theta_q
      \left(\sum_i x_iD_i+zH_z\right)^{D-q},
 \label{eq:packet-Hodge-derivative}\\
 \left.\partial_z^q f(x,z)\right|_{z=0}
 &=\frac1{(D-q)!}\int_X\theta_q
      \left(\sum_i x_iD_i\right)^{D-q}.
 \label{eq:packet-Hodge-layer}
\end{align}
A nonzero polynomial in either display ensures $\theta_q\ne0$.
Every nonzero layer and minimal packet therefore has a source-level
Hodge completion with ample integral extra directions and apolar algebra
$\mathcal H_q$.
\end{corollary}
\begin{proof}
Reciprocity and finite-box pairing give \eqref{eq:packet-source-Hodge};
reversal leaves $H_z$ fixed. Differentiate its intersection expansion
to obtain both formulas and their factorials. If $\theta_q=0$, both
polynomials vanish. The pulled-back hyperplane line $\cL_z$ is generated,
and $\theta_q=c_{\mathrm{top}}(\cV\oplus\cL_z^{\oplus q})$.
Apply \cref{hd-thm:geometric} and the ample spanning construction
of \cref{hd-cor:canonical-completion}. Since $A_X$ is divisor-generated,
the resulting apolar algebra is $\mathcal H_q$; the same incidence
argument gives an actual complex volume realization.
Finally, \cref{prop:RV-closure} identifies layers and padding
removal with these derivatives and restrictions.
\end{proof}
For Lascoux and atom specializations, $q=\length(w)-k$ extracts layer
$k$, including the key or Demazure atom at $q=\length(w)$.
For Grothendieck, $q=L-k$ extracts $\Nrm(A_{w,k})$, with Schubert at
$q=L$. Padding removal uses $q=\length(w)-r_\alpha$,
$\length(w)-r_\alpha^{\rm at}$, or $L-r_w$, respectively, without
setting $z=0$.

This source construction also proves Lorentzianity through
\cref{hd-cor:degree-polynomial}; arbitrary-field realizability uses
incidence. Canonicity remains relative to the source and finite cap.

\subsection{Some remarks and obstructions}\label{sec:scope-obstructions}\label{sec:limits}
Volume realizations determine intersection numbers, not the cohomology
of a source or the higher Hodge structure of the polynomial's apolar
algebra. The source-level completions of
\cref{sec:moment-algebras,sec:geometric-lefschetz} require generation
hypotheses; the following examples explain these distinctions.

\subsubsection{Apolar duality and a middle-degree obstruction}
For $0\ne f\in\RR[x_1,\ldots,x_n]_d$, write
\[
 A_f=\RR[\partial_{x_1},\ldots,\partial_{x_n}]
       /\operatorname{Ann}(f),
 \qquad \deg[D]=Df\quad(\deg D=d),
\]
where the operators are ordinary partial derivatives.
When $f$ has rational coefficients, this is the scalar extension to
$\RR$ of its rational apolar algebra. It has Poincar\'e duality: a
nonzero $Df$ has a coefficient detected by a complementary monomial
derivative. For $L\in A_f^1$, $0\le j\le d/2$, and a basis $(e_a)$
of $A_f^j$, the matrix
$M_j(L)=(\deg(e_ae_bL^{d-2j}))_{a,b}$ represents the Lefschetz map
under Poincar\'e duality. Hard Lefschetz requires its nonsingularity;
Hodge--Riemann further requires $(-1)^jM_j(L)$ to be positive definite
on $P_L^j=\ker(L^{d-2j+1}:A_f^j\to A_f^{d-j+1})$.

\begin{example}\label{ex:apolar-obstruction}
Consider the realizable-volume polynomial
\[
 \begin{gathered}
 f(x,y)=\frac{x^4+8x^3y+12x^2y^2+8xy^3+y^4}{24},\\
 f=\sum_{i=0}^4c_i x^{[4-i]}y^{[i]},\qquad(c_i)=(1,2,2,2,1).
 \end{gathered}
\]
Indeed, for $\Delta=\operatorname{conv}(0,e_1,e_2,e_3,e_4)$, take
$P=2\Delta$ and
\[
 Q=\operatorname{conv}(0,4e_1,2e_2,2e_3,e_4).
\]
The simplex formula \cite[Lemma~36]{Huh12}, normalized by
$\operatorname{MV}(\Delta^4)=1$, gives mixed volumes
$(16,32,32,32,16)$. On the split projective toric variety of the
normal fan of $P+Q$, the corresponding generated Cartier divisors satisfy
\[
 f(x,y)=\frac1{16\cdot4!}\int_X(xD_P+yD_Q)^4\in\RV.
\]
This realization works over every field. Over $\RR$, put
$L=\partial_x+\partial_y$ and
$A=4\partial_x^2-7\partial_x\partial_y+4\partial_y^2$.
Then $(LA)f=0$, while the degree-two pairing is
\[
 \begin{gathered}
 C_2=\begin{pmatrix}1&2&2\\2&2&2\\2&2&1\end{pmatrix},\qquad\det C_2=2,\\
 \deg(A^2)=(4,-7,4)C_2(4,-7,4)^{\mathsf T}=-30.
 \end{gathered}
\]
Thus $[A]\ne0$ is primitive and violates the positive degree-two
Hodge--Riemann sign. Nevertheless, the quadratic derivative Hessians
have determinants $-2,0,-2$, consistent with Lorentzianity. The apolar
quotient is not the cohomology algebra of the realizing variety.

In fact, Hard Lefschetz holds for this $L$. The apolar Hilbert
vector is $(1,2,3,2,1)$, one has $\deg(L^4)=30$, and the
Lefschetz pairing on $A_f^1$ has matrix
\[
 \begin{pmatrix}7&8\\8&7\end{pmatrix},
 \qquad \det=-15.
\]
Thus the obstruction is specifically a failure of the
middle-degree Hodge--Riemann sign, not a failure of
Hard Lefschetz.
\end{example}

\section{Finite matroid certificates and tropical Chern inequalities}\label{reorg:tropical}
Finite generating witnesses replace geometric evaluation kernels in
this section. The comparison requires neither representability nor
exact sequences of tropical bundles.

\subsection{Local Chern data and projective-bundle relations}\label{sec:finite-matroid-certificates}

Let $N$ be a lattice, $M=N^\vee$, and $\Sigma$ a smooth complete
projective fan of dimension $d$. The algebra
$A=A^\bullet(X_\Sigma)_\RR$ is the ring of piecewise polynomials
modulo global linear characters, graded by codimension. Its degree and
cone $\mathcal K_A=\operatorname{Amp}(X_\Sigma)_\RR$ give the
K\"ahler package.

A rank-$r$ Kaveh--Manon bundle $\cE$ has a finite loopless rank-$r$
matroid $\mathsf M$ and an integral piecewise-linear map into its
Bergman fan, linear in an apartment on each cone. Use the
minimum-on-circuits convention. For coordinate functions $v_e$, an
adapted basis $B_\sigma$, and characters $u_{\sigma,b}\in M$,
\begin{align}
 v_b(x)&=\langle u_{\sigma,b},x\rangle
 &&(b\in B_\sigma,\ x\in\sigma),\label{km-eq:apartment-basis}\\
 v_e(x)&=\min_{b\in C_{B_\sigma}(e)\setminus\{e\}}v_b(x)
 &&(e\notin B_\sigma,\ x\in\sigma),\label{km-eq:apartment-circuit}\\
 c_t^T(\cE)|_\sigma&=\prod_{b\in B_\sigma}(1+u_{\sigma,b}t).
 &&\label{km-eq:chern-local}
\end{align}
Here $C_B(e)$ is the fundamental circuit and $T$ means equivariant
\cite[Proposition~3.7 and Definitions~4.1, 5.3]{KM24}.
Global generation means that on every maximal cone a basis can be
chosen with
\begin{equation}\label{km-eq:ray-witness}
 \langle u_{\sigma,b},v_\rho\rangle\le v_b(v_\rho)
 \quad(\rho\in\Sigma(1),\ b\in B_\sigma),
\end{equation}
where $v_\rho$ is the primitive ray generator
\cite[Lemma~6.2, Definition~6.4 and Theorem~6.5]{KM24}.
For a basis $J$ of a matroid $Q$ on $G$, fix the sign by
\begin{equation}\label{km-eq:min-cone}
 \mathcal N_J^{\min}(Q)=
 \{w\in\RR^G:\textstyle\sum_{g\in J}w_g\le\sum_{g\in J'}w_g
                \text{ for every basis }J'\}.
\end{equation}

We first compare the normal-cone and Chern-root conventions with those in \cite{LP26,BEST23}.
\begin{lemma}\label{km-lem:sign}
In Larson--Partida's relation convention \cite{LP26}, coefficient
classes on $\mathcal N_J^{\min}(Q)$ are represented by
$k_j(w)=e_j(w_g:g\in J)$. Pullback by
$w_g=\langle u_g,x\rangle$ gives coefficients of
$\prod_{g\in J}(1+u_gt)$; the positive relative class remains $+\xi$.
\end{lemma}
\begin{proof}
For the standard coordinate torus action
\cite[Introduction and Example~1.1]{LP26}, the basis-$J$ chart has
Pl\"ucker-ratio characters $e_{J'}-e_J$. Its inequalities are
$\langle e_{J'}-e_J,w\rangle\ge0$, exactly \eqref{km-eq:min-cone}.
The tautological subspace has characters $t_g$, $g\in J$, yielding
$e_j(t_g:g\in J)$; the combinatorial definition uses the same
representatives without representability.

This agrees equivariantly with \cite{BEST23}: Remark~3.3 uses minimum
normal cones; Lemma~3.5 sends a decreasing braid chamber to the minimum
basis $B_{\bar\tau}(Q)$ under the unmodified orbit map; Definition~3.9
assigns roots $-t_g$ to the complement of a maximum basis. Thus
$\cQ_{Q^*}^\vee$ has roots $t_g$ on the minimum basis, as in \cite{LP26}.
On walls, optimal bases are connected by equal-weight exchanges in
their common face, so their elementary symmetric functions agree.
\end{proof}

In these conventions, the multiple-projective-bundle theorem of Larson and Partida \cite[Theorem~3.9]{LP26} reads as follows.
\begin{proposition}\label{km-thm:LP}
Let $Q_i$ be matroids on $G_i$ of ranks $s_i\ge2$, with characters $u_{i,g}$
whose lattice maps send each cone of $\Sigma$ into a minimum-basis
cone. For the ordinary coefficient classes of \cref{km-lem:sign},
\[
 B=A[\xi_1,\ldots,\xi_m]/
 (\xi_i^{s_i}+k_{i,1}\xi_i^{s_i-1}+\cdots+k_{i,s_i})_i
\]
has the K\"ahler package for the interior of the cone generated by
$\mathcal K_A$ and the $\xi_i$, with normalization
$\deg_B(a\prod_i\xi_i^{s_i-1})=\deg_A(a)$ for $a\in A^d$.
No matroid representability is required.
\end{proposition}
\begin{proof}
Apply \cite[Theorem~3.9, p.~17]{LP26} with \cref{km-lem:sign}. Since $s_i\ge2$,
the relative degree-one directions are independent.
The input permits zero-weight coloops.
\end{proof}

\subsection{Generating witnesses and dual matroids}
The ray inequalities in the generation condition extend to the entire fan.
\begin{lemma}\label{km-lem:global-witness}
A witness in \eqref{km-eq:ray-witness} satisfies
$\langle u_{\sigma,b},x\rangle\le v_b(x)$ for all $x\in N_\RR$,
with equality on $\sigma$.
\end{lemma}
\begin{proof}
For $x=\sum_{\rho\in\tau(1)}a_\rho v_\rho$ in a cone $\tau$,
$a_\rho\ge0$, the apartment formula is a minimum of linear functions.
Hence
$v_b(x)\ge\sum_\rho a_\rho v_b(v_\rho)
\ge\sum_\rho a_\rho\langle u_{\sigma,b},v_\rho\rangle$.
Equality on $\sigma$ is \eqref{km-eq:apartment-basis}.
\end{proof}

The apartment formula also identifies the adapted basis as a maximum-weight basis.
\begin{lemma}\label{km-lem:maximum-basis}
For $x\in\sigma$, the basis $B_\sigma$ maximizes total $v_e(x)$-weight.
\end{lemma}
\begin{proof}
Every exchange $B_\sigma-b+e$ has
$b\in C_{B_\sigma}(e)\setminus\{e\}$, so
$v_e(x)\le v_b(x)$ by \eqref{km-eq:apartment-circuit}. No edge improves
weight. Since the tangent cone of a matroid base-polytope vertex is
generated by its single-exchange edges, a higher-weight vertex would
force an improving edge, a contradiction.
\end{proof}

Take one witness basis for each maximal cone. On
$G_0=\{(\sigma,b):\sigma\in\Sigma(d),\ b\in B_\sigma\}$, set
$p(\sigma,b)=b$, $u_{(\sigma,b)}=u_{\sigma,b}$, and define
\begin{equation}\label{km-eq:label-rank}
 \operatorname{rk}_P(S)=\operatorname{rk}_{\mathsf M}(p(S)).
\end{equation}
This is restriction followed by parallel duplication: the rank axioms
follow from those of $\mathsf M$ and
$p(S\cap T)\subseteq p(S)\cap p(T)$. It has rank $r$, with bases
$I_\sigma=\{(\sigma,b):b\in B_\sigma\}$.
Add two zero-weight loops and call the ground set $G$, so $s=|G|-r\ge2$.

Passing to complements converts these maximum-weight witnesses into the minimum-basis cones required by the projective-bundle theorem.
\begin{proposition}\label{km-prop:finite-certificate}
For $Q=P^*$ and $J_\sigma=G\setminus I_\sigma$, the lattice map
$\varphi_*:x\mapsto(\langle u_g,x\rangle)_{g\in G}$ sends $\sigma$
into $\mathcal N_{J_\sigma}^{\min}(Q)$.
\end{proposition}
\begin{proof}
For a basis $H$ of $P$, $p$ is injective on $H$ and $p(H)$ is a basis
of $\mathsf M$. The preceding lemmas give on $\sigma$
\[
 \sum_{g\in H}\langle u_g,x\rangle
 \le\sum_{e\in p(H)}v_e(x)
 \le\sum_{b\in B_\sigma}v_b(x)
 =\sum_{g\in I_\sigma}\langle u_g,x\rangle.
\]
Thus $I_\sigma$ is maximum. Subtracting from the total ground-set weight
makes its complement minimum in $P^*$; restriction covers the faces.
\end{proof}

The dual presentation provides the inverse Chern series with the required positive relative class.
\begin{theorem}\label{km-thm:inverse}
The finite data above give a positive inverse-Chern certificate for a
globally generated KM bundle $\cE$ on its original fan. Its
Larson--Partida series $K(t)=1+k_1t+\cdots+k_st^s$ satisfies
\begin{equation}\label{km-eq:ordinary-inverse}
 K(t)c_t(\cE)=1\quad\text{in }A[t].
\end{equation}
The ring $B_{\cE}=A[\xi]/(\xi^s+k_1\xi^{s-1}+\cdots+k_s)$ has the
K\"ahler package for $L+t\xi$, $L\in\mathcal K_A$, $t>0$, with
$+\xi$ in the cone closure.
\end{theorem}
\begin{proof}
By sign calibration,
$K^T(t)|_\sigma=\prod_{g\in J_\sigma}(1+u_gt)$, so
\begin{equation}\label{km-eq:equivariant-inverse}
 K^T(t)c_t^T(\cE)=\prod_{g\in G}(1+u_gt).
\end{equation}
Positive-degree coefficients on the right lie in the ideal of global
linear characters and vanish in ordinary Chow. This proves the inverse
identity. Apply \cref{km-thm:LP} using the cone map of
\cref{km-prop:finite-certificate}, leaving $+\xi$ unchanged.
\end{proof}
The dual matroid $P^*$, the nontrivial equivariant product on the right,
and the sign $+\xi$ must all be retained in this comparison.
\subsection{Tropical Hodge theory and Chern inequalities}\label{sec:tropical-hodge-transport}

The multiple-bundle theorem applies to these certificates simultaneously.
\begin{proposition}\label{km-thm:joint-Hodge}
For globally generated KM bundles $\cE_1,\ldots,\cE_m$ on the same
smooth projective fan, the intrinsic algebra $R_C$ with
$C_{i,a}=c_a(\cE_i)$ has the K\"ahler package in formal dimension $d$
for the image of
$\{L+\sum_it_iz_i:L\in\mathcal K_A,\ t_i>0\}$.
For the fixed coefficient algebra $A$ with its degree map,
the moment algebra, degree map, and labelled generators depend
only on the classes $c_a(\cE_i)\in A^a$, and not on the chosen
finite generating witnesses.
For nef classes $H_j\in A^1$, the polynomial
\begin{equation}\label{km-eq:joint-polynomial}
 \sum_{|\alpha|+|\beta|=d}
 \deg_A\left(\prod_ic_{\alpha_i}(\cE_i)\prod_jH_j^{\beta_j}\right)
       \frac{x^\alpha y^\beta}{\alpha!\beta!}
\end{equation}
is Lorentzian, with zero allowed. If ample $H_j$ span $A^1$, its
apolar algebra is $R_C$, under $\partial_{x_i}\mapsto z_i$ and
$\partial_{y_j}\mapsto H_j$, and has the full K\"ahler package.
Finally, for $r_i=\rk\cE_i$ and
$\theta=\prod_ic_{r_i}(\cE_i)\ne0$, the quotient
$A/\Ann_A\theta$, with degree $[a]\mapsto\deg_A(\theta a)$, has
the package in dimension $d-\sum_i r_i$ for the image of the ample cone.
\end{proposition}
\begin{proof}
Construct each inverse certificate by \cref{km-thm:inverse}.
The multiple-bundle input \cref{km-thm:LP} gives the joint monic
presentation, with all $+\xi_i$ boundary and $K_iC_i=1$.
Apply \cref{hd-thm:moment-Hodge,hd-thm:quotient} for the algebra and
cone, and \cref{hd-cor:degree-polynomial} for the scalar polynomial.
When the ample directions span $A^1$, the toric Chow ring is
degree-one-generated, so \cref{hd-thm:apolar} identifies the apolar
algebra. The terminal assertion is \cref{hd-thm:terminal}.
\end{proof}
The spanning hypothesis is essential for the apolar assertion; an
arbitrary smaller base specialization is not covered by that conclusion.

We specialize the joint polynomial to the Chern numbers in the Kaveh--Manon conjecture.
\begin{corollary}\label{km-thm:log-concavity}
Let $\cE$ be a globally generated rank-$r$ tropical toric bundle in the
sense of \cite[Definitions~4.1 and~6.4]{KM24} on a smooth projective
$d$-dimensional toric variety, $d\ge r$. For ample $H$, let
$a_i=\deg_A(c_i(\cE)H^{d-i})$. Then $a_i\ge0$ and
\begin{equation}\label{km-eq:log-concavity}
 a_i^2\ge a_{i-1}a_{i+1}\qquad(1\le i<r).
\end{equation}
With $c_i(\cE)=0$ for $i>r$, the polynomial
\begin{equation}\label{km-eq:binary-polynomial}
 F_{\cE,H}(x,y)=\sum_{i=0}^d
 \deg_A(c_i(\cE)H^{d-i})\frac{x^iy^{d-i}}{i!(d-i)!}
\end{equation}
is nonzero Lorentzian, with an interval of nonzero coefficient indices.
For nef $H$ it remains Lorentzian, with zero allowed.
\end{corollary}
\begin{proof}
Apply \cref{km-thm:joint-Hodge}. The inverse certificate gives
\begin{equation}\label{km-eq:moments}
 \pi_*(\xi^{s-1+i})=c_i(\cE),\qquad
 a_i=\deg_{B_{\cE}}(\xi^{s-1+i}H^{d-i}).
\end{equation}
The bivariate criterion \cite[Example~2.26]{BH20} gives the signs,
log-concavity, and interval support. For ample $H$, $a_0=\deg_AH^d>0$,
including rank zero; on the nef boundary zero is permitted.
\end{proof}
This is \cite[Conjecture~1.7]{KM24}. The finite comparison
\cref{km-prop:finite-certificate,km-thm:inverse} joins its generation
hypothesis to Larson--Partida's Hodge theorem; neither external source
is claimed to contain that comparison.

\subsection{Further finite presentations}
The finite presentation can also be used to recover globally generated Chern data.
\begin{definition}\label{km-def:presentation}
A $\Sigma$-adapted weighted matroid presentation is a finite matroid
$P$ on $G$ with characters $u_g\in M$, such that each maximal cone
has a basis $I_\sigma$ maximizing
$\sum_{g\in I}\langle u_g,x\rangle$ simultaneously for every
$x\in\sigma$. Set
$C^T_{P,u}(t)|_\sigma=\prod_{g\in I_\sigma}(1+u_gt)$.
\end{definition}
Equal-weight exchanges in common optimal faces make these polynomial
restrictions agree on intersections. Loops have no effect.

The closure envelope gives a converse at the level of Chern data, without claiming uniqueness of the tropical bundle.
\begin{proposition}\label{km-thm:converse}
Deleting loops from an adapted presentation produces a globally
generated KM bundle with Chern series $C^T_{P,u}$. Thus these
presentations give exactly the Chern data of globally generated KM
bundles, not a classification of bundle isomorphism classes.
\end{proposition}
\begin{proof}
After deleting loops, rank zero is immediate. For
$w_g(x)=\langle u_g,x\rangle$, define
\begin{equation}\label{km-eq:closure-envelope}
 v_e(x)=\max_{\varnothing\ne S\subseteq G,\ e\in\cl_P(S)}
                  \min_{g\in S}w_g(x).
\end{equation}
Its upper level sets satisfy
\begin{equation}\label{km-eq:level-flats}
 \{e:v_e(x)\ge a\}=\cl_P\{g:w_g(x)\ge a\}.
\end{equation}
They are flats, so $v$ is a minimum-on-circuits Bergman vector and is
the least such majorant of $w$. In Ardila's opposite convention,
$v(w)=-(-w)^P$. His basis formula
\cite[\S3, final proposition, pp.~7--8]{Ardila04} gives, for each
maximum basis $I$, $v_b=w_b$ on $I$ and
$v_e=\min_{b\in C_I(e)\setminus\{e\}}w_b$ off $I$.
The fixed optimizer $I_\sigma$ yields integral apartment formulas
on the original fan, while $v_b\ge w_b$ globally with equality on
$\sigma$ gives generating witnesses. The local Chern series is
$C^T_{P,u}$. The reverse direction is
\cref{km-prop:finite-certificate} and \eqref{km-eq:label-rank}.
\end{proof}

Adaptation itself can be checked by finitely many single-exchange inequalities on the rays.
\begin{proposition}\label{km-prop:exchange-test}
For fixed $P,u$ and proposed $I_\sigma$, adaptation is equivalent to
\[
 \langle u_b-u_e,v_\rho\rangle\ge0
 \quad(\rho\in\sigma(1),\ e\notin I_\sigma,
       \ b\in C_{I_\sigma}(e)\setminus\{e\}).
\]
Loops impose no condition.
\end{proposition}
\begin{proof}
Necessity excludes improving exchanges. Conversely, extend the ray
inequalities linearly over each cone and apply
\cref{km-lem:maximum-basis}'s exchange criterion. Fundamental circuits
are obtained from an independence or rank oracle.
\end{proof}

Keeping all labels, including loops, makes matroid duality into an involution on these presentations.
\begin{proposition}\label{km-prop:Gale}
The involution $\mathcal G(P,u)=(P^*,-u)$ on full labelled
presentations preserves adaptation and satisfies
\begin{equation}\label{km-eq:Gale}
 C_{\mathcal G(P,u)}(t)=C_{P,u}(-t)^{-1}.
\end{equation}
\end{proposition}
\begin{proof}
The complement of a maximum $u$-basis is a maximum $-u$-basis of
$P^*$. Equivariantly,
$C^T_{P,u}(t)C^T_{P^*,-u}(-t)=\prod_{g\in G}(1+u_gt)$;
ordinary classes give the identity. Dualization and character negation
applied twice restore the presentation.
\end{proof}
Retain loops for the involution: they become coloops recording
weighted trivial summands; delete them only for the loopless envelope.
This is neither a canonical functor on all bundle isomorphism classes
nor a tropical evaluation sequence. Direct sums multiply Chern series,
and compatible integral fan pullback preserves witnesses on smooth
projective source fans. Zero-weight loops change only auxiliary ranks.
Whitney multiplication is not $z\mapsto z_1+z_2$, which introduces
binomial factors.
Chern-flow identities extend to this setting at the level of rank-marked characteristic series.
\begin{proposition}\label{km-prop:characteristic-network}
Let a finite directed graph have rank-marked vertex series $(r_v,C_v)$
and edge series $(q_e,K_e)$ in $A[t]$, with constant term one,
coefficient of $t^j$ in $A^j$, and vanishing above the marked rank.
Assume $q_e=r_{t(e)}-r_{s(e)}\ge0$ and $C_{t(e)}=C_{s(e)}K_e$.
For nonnegative integral multiplicities $a,b,\kappa$ with
$b=a+B_\Gamma\kappa$, suppose
$C_b=\prod_vC_v^{b_v}$ has an adapted presentation of rank
$R=\sum_vb_vr_v$. Its moment algebra has the K\"ahler package.
If $\Theta=\prod_vC_{v,r_v}^{a_v}\prod_eK_{e,q_e}^{\kappa_e}\ne0$,
then $A/\Ann_A\Theta$ has the weighted package of dimension $d-R$.
\end{proposition}
\begin{proof}
Positive-degree coefficients are nilpotent, so the series are invertible.
Multiplying edge identities gives
$C_b=\prod_vC_v^{a_v}\prod_eK_e^{\kappa_e}$ and
$R=\sum_va_vr_v+\sum_e\kappa_eq_e$. Its degree-$R$ coefficient is
$\Theta$, since every factor must contribute its rank-degree term.
The terminal presentation yields a generated KM bundle by
\cref{km-thm:converse}; apply \cref{km-thm:joint-Hodge}.
\end{proof}
For exact sequences the identities are Whitney multiplicativity.
Abstractly the identities and terminal presentation are hypotheses,
not existence claims for tropical kernels or cokernels; top Chern
classes are never divided.

More generally, for the projective-bundle data covered by
\cite[Theorem~3.9]{LP26} on its specified Lefschetz fan, the inverse
arrays $(1+\sum_jk_{i,j}t^j)^{-1}=\sum_{a\ge0}S_{i,a}t^a$ have
positive moment algebras by \cref{hd-thm:moment-Hodge}. No actual
bundle interpretation is required.

For a loopless rank-$r_0$ matroid $M_0$ on $N_0$ elements, the
permutohedral relation with coefficients $c_j(\cS_{M_0})$ has inverse
array $c_a(\cQ_{M_0})$, zero above $N_0-r_0$. Its terminal quotient
$A(\Sigma_{\mathrm{perm}})/\Ann(c_{N_0-r_0}(\cQ_{M_0}))$ is the
matroid Chow ring \cite[Example~3.12]{LP26}, of formal dimension $r_0-1$.
This is not an independent proof of matroid Hodge theory, already
used by Larson--Partida.

The comparison concerns the objects of \cite{KM24} on the stated
smooth projective toric bases. Without representability it produces
combinatorial Hodge algebras, not actual vector bundles or geometric
volume models. It does not cover nontrivially valuated matroids,
every notion of tropical vector bundle, or arbitrary positive sums
and pushforwards of Hodge algebras.
\section{Chern nonvanishing and coefficient inequalities}\label{sec:chern-nonvanishing}
\label{nv-sec:nonvanishing}
To recover class-valued vanishing from numerical moment support, we first
prove detection by one ample intersection in the geometric and KM settings.

\subsection{Chern products and ample detection}
For generated geometric bundles, effective representatives allow a single ample degree to detect nonvanishing.
\begin{lemma}\label{nv-lem:effective}
Let $X$ be a smooth integral complex projective $d$-fold and let
$\cE_1,\ldots,\cE_m$ be generated bundles of ranks $r_i$.
For $0\le\alpha_i\le r_i$ and $|\alpha|\le d$,
$C_\alpha=\prod_i c_{\alpha_i}(\cE_i)$ has an effective cycle representative.
For every ample line class $H$,
\begin{equation}\label{nv-eq:detection}
 C_\alpha\ne0\text{ in }\CH^{|\alpha|}(X)_\QQ
 \quad\Longleftrightarrow\quad
 \int_X C_\alpha H^{d-|\alpha|}>0.
\end{equation}
\end{lemma}
\begin{proof}
On $Y=X\times\prod_i\PP^{r_i-\alpha_i}$, the generated bundle
$\cV=\bigoplus_i p_X^*\cE_i\otimes\cO_i(1)$ has rank $R=\sum_i r_i$.
Choose a finite-dimensional generating space $W$ for $\cV$. The
affine incidence $\{(y,s)\in Y\times W:s(y)=0\}$ is the total space
of the evaluation kernel over $Y$, hence is smooth. Its projection to
$W$ is proper because $Y$ is projective. If the projection is not
dominant, a general zero scheme is empty; otherwise generic smoothness
in characteristic zero gives a smooth general zero scheme $Z$ of
codimension $R$. Hence $[Z]=c_R(\cV)\cap[Y]$.
Writing $h_i=c_1(\cO_i(1))$, expand
$c_{r_i}(\cE_i\otimes\cO_i(1))=\sum_jc_j(\cE_i)h_i^{r_i-j}$ and push
forward to obtain $p_{X*}[Z]=C_\alpha\cap[X]$.
This is effective, with lower-dimensional images contributing zero.
Every nonzero effective cycle has positive ample degree. Thus degree
zero forces this representative, and therefore the Chow class, to vanish;
a nonzero degree precludes vanishing.
\end{proof}
For these Chern products, nonvanishing in rational Chow is equivalent
to nonvanishing of the real cycle class: a zero cycle class has zero
ample degree, so the lemma gives vanishing in rational Chow.

A degree-one-generated coefficient algebra admits a numerical substitute for this effectivity argument.
\begin{lemma}\label{nv-lem:abstract-detection}
Let $A$ be a finite-dimensional Poincar\'e-duality algebra of formal
dimension $d$, generated by $A^1$. Fix $H$ in an open cone
$\mathcal A\subset A^1$ spanning $A^1$. Suppose joint Chern--divisor
degrees are nonnegative on a cone $\mathcal N\supseteq\mathcal A$,
and $MH-L\in\mathcal N$ for every $L\in\mathcal A$ and all large $M$.
Then \eqref{nv-eq:detection}, with degree in $A$, holds for these arrays.
\end{lemma}
\begin{proof}
Put $k=d-|\alpha|$. If $\deg(C_\alpha H^k)=0$, choose
$L_1,\ldots,L_k\in\mathcal A$ and a common $M$ with
$MH-L_j\in\mathcal N$. Expanding $\prod_j(L_j+(MH-L_j))$ gives
$0\le\deg(C_\alpha L_1\cdots L_k) \le M^k\deg(C_\alpha H^k)=0$.
These products span $A^k$, so Poincar\'e duality gives $C_\alpha=0$.
The converse is immediate.
\end{proof}
For generated KM bundles, \cref{km-thm:joint-Hodge} supplies mixed
nonnegativity; the ample and nef cones of the toric Chow algebra have
the required domination property. No representability is needed.

\subsection{Nonvanishing polymatroids and algebraic matroids}
The joint moment support turns these detection results into a complete criterion for products of Chern classes.
The Hall--Rado criterion for nef intersections provides a related geometric
precedent; see \cite[\S2.4]{HuXiao25}. Here the rank is computed from the
Chern classes of each direct sum, rather than from the numerical dimension
of a sum of divisor classes.

\begin{theorem}\label{nv-thm:polymatroid}
In the geometric setting of \cref{nv-lem:effective}, or for generated
KM bundles on a smooth projective toric variety, set
\begin{equation}\label{nv-eq:rank}
 \rho_C(S)=\max\left\{j:c_j\left(\bigoplus_{i\in S}\cE_i\right)\ne0\right\},
 \qquad\rho_C(\varnothing)=0.
\end{equation}
Nonvanishing is in $\CH^*(X)_\QQ$ or the toric Chow algebra, respectively.
Then $\rho_C$ is normalized, monotone, integral, and submodular.
For every $\alpha\in\NN^m$, one has
\begin{equation}\label{nv-eq:HallRado}
 \prod_i c_{\alpha_i}(\cE_i)\ne0
 \quad\Longleftrightarrow\quad
 \alpha(S)\le\rho_C(S)\quad(S\subseteq[m]).
\end{equation}
Out-of-rank Chern classes are zero.
\end{theorem}
\begin{proof}
For ample $H$, \cref{hd-cor:mixed-Chern,km-thm:joint-Hodge} make
\[
 F(x,t)=\sum_{|\alpha|\le d}
 \deg(C_\alpha H^{d-|\alpha|})
 \frac{x^\alpha t^{d-|\alpha|}}{\alpha!(d-|\alpha|)!}
\]
Lorentzian, with positive pure-slack coefficient $\deg(H^d)$.
Its support is an integral polymatroid base set with rank function $r$
\cite[Theorem~2.25]{BH20}. Every set containing the slack index $0$
has rank $d$, since $de_0$ is a base. Those inequalities impose nothing
further on nonnegative vectors of total sum $d$. Deleting the slack
coordinate therefore gives
$I(r|_{[m]})=\{\alpha\in\NN^m:\alpha(S)\le r(S)\ (S\subseteq[m])\}$;
see also \cite{Murota03}. The detection lemmas identify this with the
nonzero Chern products.

Whitney's formula expresses $c_j(\bigoplus_{i\in S}\cE_i)$ as the sum
of $C_\beta$ with $|\beta|=j$ supported in $S$. Ample detection prevents
cancellation: each nonzero summand has strictly positive degree.
Moreover, $C_\alpha\ne0$ implies $C_{\alpha|_S}\ne0$.
Thus the largest such $j$ is
$\max_{\alpha\in I(r|_{[m]})}\alpha(S)=r(S)$, identifying $r|_{[m]}$
with \eqref{nv-eq:rank} and proving \eqref{nv-eq:HallRado}.
\end{proof}

\begin{example}
For $\cO(1,0),\cO(1,0),\cO(0,1)$ on $\PP^1\times\PP^1$,
$\rho_C(S)=\mathbf1_{S\cap\{1,2\}\ne\varnothing}+\mathbf1_{3\in S}$.
Distinct rank-one labels can give the same intersection direction;
Chern rank is not the rank of a matroid of generating sections.
\end{example}

After a feasible set of Chern indices has been chosen, the remaining admissible indices are governed by the following residual rank.
\begin{proposition}\label{nv-thm:residual}
Let $\rho$ be an integral polymatroid rank function on $[m]$, and write
$I(\rho)=\{\alpha\in\NN^m:\alpha(S)\le\rho(S)\text{ for }S\subseteq[m]\}$.
For $\alpha\in I(\rho)$, set
\begin{equation}\label{nv-eq:residual}
 \rho_\alpha(S)=\min_{T\supseteq S}(\rho(T)-\alpha(T)).
\end{equation}
Then $\rho_\alpha$ is an integral polymatroid rank function. For every
$\beta\in\NN^m$, one has
$\alpha+\beta\in I(\rho)$ if and only if $\beta\in I(\rho_\alpha)$.
For such a $\beta$ with $\alpha+\beta\in I(\rho)$, moreover,
$(\rho_\alpha)_\beta=\rho_{\alpha+\beta}$.
\end{proposition}
\begin{proof}
Feasibility gives nonnegativity and normalization; the minimum gives
integrality and monotonicity. For minimizers $T_j\supseteq S_j$,
submodularity of $\rho-\alpha$ yields
\[
 \rho_\alpha(S_1)+\rho_\alpha(S_2)
 \ge(\rho-\alpha)(T_1\cup T_2)+(\rho-\alpha)(T_1\cap T_2)
 \ge\rho_\alpha(S_1\cup S_2)+\rho_\alpha(S_1\cap S_2).
\]
If $\alpha+\beta$ is feasible, then
$\beta(S)\le\beta(T)\le\rho(T)-\alpha(T)$ for $T\supseteq S$.
Conversely, $\beta(S)\le\rho_\alpha(S)\le\rho(S)-\alpha(S)$ gives
feasibility. Finally,
$(\rho_\alpha)_\beta(S)=\min_{S\subseteq T\subseteq U}
(\rho(U)-\alpha(U)-\beta(T))$; for fixed $U$, nonnegativity of $\beta$
makes $T=U$ optimal.
\end{proof}
Deletion restricts $\rho_C$; grouping disjoint blocks $J_1,\ldots,J_s$
by direct sum gives $\rho_{\rm group}(T)=\rho_C(\bigcup_{j\in T}J_j)$.
These are support operations, not identities with fixed Chern insertions.

In the geometric case, polarization realizes the nonvanishing criterion by an algebraic matroid on labelled copies of the bundle indices.
\begin{proposition}\label{nv-thm:clone}
In the geometric setting of \cref{nv-lem:effective}, choose disjoint
blocks $B_i$ of sizes $r_i$. There is a matroid $M_C$ on
$\bigsqcup_iB_i$, algebraic over $\CC$, with
\begin{equation}\label{nv-eq:clone-independence}
 I\text{ independent}\quad\Longleftrightarrow\quad
 \prod_i c_{|I\cap B_i|}(\cE_i)\ne0\text{ in }\CH^*(X)_\QQ.
\end{equation}
It is determined by labelled Chern nonvanishing and has block rank
$r_{M_C}(\bigcup_{i\in S}B_i)=\rho_C(S)$.
For feasible $\alpha$ and independent $I_\alpha$ with counts $\alpha$,
the contraction $M_C/I_\alpha$ has residual block rank $\rho_{C,\alpha}$,
which is algebraic over $\CC$.
\end{proposition}
\begin{proof}
For $d=0$ take the rank-zero matroid. Otherwise choose very ample $H$
and $\cE_0=H^{\oplus d}$, so $c_j(\cE_0)=\binom djH^j$.
Total-Chern realization gives
\[
 \Nrm\left(\sum_{|\alpha|+j=d}\binom dj
          \deg(C_\alpha H^j)x^\alpha t^j\right)\in\RVC.
\]
Polarize in the blocks $B_i$ and a slack block $B_0$ of size $d$.
By \cref{thm:algebraic-polymatroid-consequence},
the support is the base set of an algebraic matroid $\widetilde M$;
$B_0$ is itself a basis. Put $M_C=\widetilde M\setminus B_0$.
If $C_\alpha\ne0$, ample detection makes every set with counts
$(\alpha,d-|\alpha|)$ a basis of $\widetilde M$, so every clone set
with counts $\alpha$ is independent. Conversely, any independent clone
set extends to a basis using only $B_0$, by repeated augmentation with
that basis. Its coefficient and ample detection give
\eqref{nv-eq:clone-independence}. Deletion forgets coordinates of an
algebraic representation, and \cref{nv-thm:polymatroid} gives the block
rank.

A remaining clone set $J$ is independent in $M_C/I_\alpha$ exactly when
$I_\alpha\cup J$ is independent in $M_C$, or
$C_{\alpha+\beta}\ne0$ for its counts $\beta$.
Apply \cref{nv-thm:residual}, with remaining bounds $r_i-\alpha_i$.
Algebraicity of the contraction follows from
\cite[Corollary~5.6]{GHMSSW25}.
\end{proof}

\begin{remark}\label{nv-rem:two-insertions}
Contraction shifts active indices to $C_{\alpha+\beta}$, whereas fixed
insertion uses $C_\alpha C_\beta$ and \cref{fn-thm:insert}.
Since $c_{a+b}(\cE)\ne c_a(\cE)c_b(\cE)$ in general, coefficient identities
must be computed separately from support operations.
\end{remark}

\subsection{Coefficient inequalities}\label{sec:coefficient-inequalities}
\label{sec:coefficient-consequences}
These consequences apply to the normalized packets and layers of
\cref{thm:main}, and to finite complements of balanced weighted packets.
The coefficients below are those of the underlying unnormalized polynomial.

\label{subsec:hodge-consequences}
Let
\begin{equation}\label{eq:consequence-polynomial}
 F(x)=\sum_{|\alpha|=d}c_\alpha x^\alpha,
 \qquad f:=\Nrm(F)\in\RVC,
\end{equation}
be nonzero and homogeneous, with $c_\alpha=0$ outside $\NN^n$.
The coefficient matrix and its two-by-two minors give the following
inequalities in one step.
\begin{corollary}\label{cor:coeff-HR}
For $|\beta|=d-2$, the matrix
\begin{equation}\label{eq:Hbeta-consequence}
 H_\beta=(c_{\beta+e_i+e_j})_{i,j=1}^n
\end{equation}
has at most one positive eigenvalue, exactly one when nonzero. Thus
for nonempty $I\subseteq[n]$,
\begin{equation}\label{eq:principal-minor-sign}
 (-1)^{|I|-1}\det(H_\beta)_I\ge0.
\end{equation}
In particular, for $|\alpha|=d$, $i\ne j$, and
$\alpha_i,\alpha_j\ge1$,
\begin{equation}\label{eq:root-logconcavity}
 c_\alpha^2\ge
 c_{\alpha+e_i-e_j}c_{\alpha-e_i+e_j}.
\end{equation}
For a packet $P(x,z)=\sum_{k=0}^rP_k(x)z^{r-k}$ with
$P_k\in\RR[x]_{d_0+k}$ and $\Nrm_{x,z}(P)\in\RVC$, put
$c_k(\alpha)=[x^\alpha]P_k$. Then
\begin{equation}\label{eq:cross-layer-logconcavity}
 c_k(\alpha)^2\ge
 c_{k-1}(\alpha-e_i)c_{k+1}(\alpha+e_i),
\end{equation}
with inadmissible indices zero. Along every indicated root line,
including direction $e_i-e_z$, the nonzero coefficients form a
log-concave interval without internal zeros.
\end{corollary}
\begin{proof}
Apply \cref{thm:BH-criterion} and interlacing to $H_\beta$ and its
principal submatrices. A nonzero nonnegative symmetric matrix has
positive Rayleigh quotient on the all-ones vector, giving exactly one
positive eigenvalue and the determinant signs. The $i,j$ minor with
$\beta=\alpha-e_i-e_j$ gives \eqref{eq:root-logconcavity}.
Taking the direction $e_i-e_z$ at $(\alpha,r-k)$ gives
\eqref{eq:cross-layer-logconcavity}; at the boundary a neighbor is
zero. Saturated $M$-convex support gives the interval assertions.
\end{proof}

Actual volume realizations give further inequalities beyond the Lorentzian signature condition.

\begin{proposition}\label{thm:reverse-KT}
Let $\beta\in\NN^n$, put $m=d-|\beta|\ge0$, and let
$0\le e\le m$.  For any $i,j,k\in[n]$,
\begin{equation}\label{eq:coefficient-reverse-KT}
 \binom{m}{e}
 c_{\beta+(m-e)e_i+ee_j}
 c_{\beta+ee_i+(m-e)e_k}
 \ge
 c_{\beta+me_i}
 c_{\beta+ee_j+(m-e)e_k}.
\end{equation}
Coefficients with an index outside $\NN^n$ are interpreted as zero.
\end{proposition}

\begin{proof}
The endpoints $e=0,m$ are identities. Otherwise realize
$\partial_x^\beta f$ by semiample divisors $D_1,\ldots,D_n$; if this
derivative vanishes, all four coefficients are zero. For a nonzero
realization, \cite[Theorem~1.1]{JiangLi23} gives
$\binom me(D_i^{m-e}D_j^e)(D_i^eD_k^{m-e}) \ge(D_i^m)(D_j^eD_k^{m-e})$.
These are the four divided-power coefficients in
\eqref{eq:coefficient-reverse-KT}, up to one common positive factor
which cancels from the quadratic inequality.
\end{proof}

The common homogeneous packet couples different excess layers.  A one-dimensional specialization turns this coupling into a concrete log-concavity statement.

\begin{proposition}\label{prop:layer-mass-log-concavity}
Let $P(x,z)=\sum_{k=0}^rP_k(x)z^{r-k}$ with
$P_k\in\RR[x]_{d+k}$ and $\Nrm_{x,z}(P)\in\RVC$.
For $c_{k,\alpha}=[x^\alpha]P_k$ and $\lambda\in\RR_{\ge0}^n$, set
\[
 M_k(\lambda)=(d+k)!\sum_{|\alpha|=d+k}
                  \frac{c_{k,\alpha}}{\alpha!}\lambda^\alpha.
\]
Then $M_k(\lambda)^2\ge M_{k-1}(\lambda)M_{k+1}(\lambda)$, with
out-of-range terms zero. Its positive indices form an empty set or
an integer interval.
\end{proposition}

\begin{proof}
Apply the nonnegative real substitution $x_i=\lambda_i t$ to
$\Nrm_{x,z}(P)$. By \cite[Theorem~2.10]{BH20}, the result
\[
 \sum_{k=0}^rM_k(\lambda)t^{[d+k]}z^{[r-k]}
\]
is Lorentzian or zero. The bivariate criterion
\cite[Example~2.26]{BH20} gives log-concavity and the interval of
nonzero coefficients, including all boundary zero patterns of $\lambda$.
\end{proof}

For $\lambda=(1,\ldots,1)$, these are the multinomially weighted total masses of successive excess layers.  The proposition does not assert log-concavity of the unweighted coefficient sums.

\subsection{Equality in the descended pairing}\label{pr-subsec:equality}
In a descended surface pairing, the Hodge-index inequality also gives a quantitative equality criterion.
\begin{proposition}\label{pr-thm:rigidity}
Let $B$ have formal dimension two and the K\"ahler package, and let
$Q(a,b)=\deg_B(ab)$ on $B^1$.  Suppose $H,z\in B^1$ satisfy
$a=Q(H,H)>0$.  Put $b=Q(H,z)$, $c=Q(z,z)$, and
$\Delta=b^2-ac$.  Then
\[
 \Delta=a\bigl(-Q(q,q)\bigr)\ge0,\qquad q=z-(b/a)H,
\]
with equality if and only if $z=(b/a)H$ in $B^1$.  For every $D\in B^1$,
\begin{equation}\label{pr-eq:stability}
 \bigl(aQ(z,D)-bQ(H,D)\bigr)^2
 \le\Delta\bigl(Q(H,D)^2-aQ(D,D)\bigr).
\end{equation}
\end{proposition}
\begin{proof}
The signature of $Q$ is $(1,\dim B^1-1)$.  Since $Q(H,H)>0$,
$H^\perp$ is negative definite.  The vector $q$ lies in this complement,
and expansion gives $\Delta=-aQ(q,q)$, including the equality condition.
For $D_0=D-Q(H,D)H/a$, Cauchy--Schwarz for $-Q$ on $H^\perp$ gives
$Q(q,D_0)^2\le[-Q(q,q)][-Q(D_0,D_0)]$.
Multiply by $a^2$ and substitute the definitions to obtain
\eqref{pr-eq:stability}.
\end{proof}

Applying this criterion to successive Chern moments expresses equality as an annihilator relation.
\begin{corollary}\label{pr-cor:Chern-rigidity}
In a one-variable Chern moment Hodge algebra of formal dimension $d$,
assume that its Chern direction $z$ is a boundary polarization, and let
$H$ be a boundary polarization.
Set $a_i=\deg(z^iH^{d-i})$.  For $1\le i\le d-1$ with $a_{i-1}>0$,
put $\eta_i=z^{i-1}H^{d-i-1}$.  Equality
$a_i^2=a_{i-1}a_{i+1}$ holds if and only if
\begin{equation}\label{pr-eq:Chern-equality}
 \eta_i(a_{i-1}z-a_iH)=0
\end{equation}
in the moment algebra.  In addition, \eqref{pr-eq:stability} applies
to every additional divisor direction after passage to the quotient
by $\Ann(\eta_i)$.
\end{corollary}
\begin{proof}
The positive number $a_{i-1}=\deg(\eta_iH^2)$ ensures $\eta_i\ne0$.
Boundary descent gives a degree-two Hodge algebra
$B_i=R_C/\Ann(\eta_i)$.  The three entries of its pairing on $(H,z)$
are $(a_{i-1},a_i,a_{i+1})$.  Apply \cref{pr-thm:rigidity}; equality
in the quotient is exactly \eqref{pr-eq:Chern-equality}.
\end{proof}

For a base divisor $D$, the left side of \eqref{pr-eq:stability} becomes
\[
 \left[
 a_{i-1}\int_X c_i(\cE)H^{d-i-1}D
       -a_i\int_X c_{i-1}(\cE)H^{d-i}D
 \right]^2.
\]
The other factor multiplying $\Delta_i=a_i^2-a_{i-1}a_{i+1}$ is
\[
 \left(\int_Xc_{i-1}(\cE)H^{d-i}D\right)^2
       -a_{i-1}\int_Xc_{i-1}(\cE)H^{d-i-1}D^2.
\]
All exponents are nonnegative in the indicated range.  This is rigidity
of classes in a descended pairing, not a splitting criterion for
bundles.  For example, the rank-two quotient bundle on $\PP^2$ has
$c(\cQ)=1+h+h^2$ and equality sequence $(1,1,1)$, but it cannot split as
$\cO(a)\oplus\cO(b)$: the required equations $a+b=1$, $ab=1$ have no
integral solution.

\appendix

\section{Reciprocal operator calculations}\label{app:operator-conjugation}
Suppress spectator variables, write $x=x_i$, $y=x_{i+1}$, and set
$(R_{M,r}f)(x,y,z)=x^My^Mz^rf(x^{-1},y^{-1},z^{-1})$.
This is an involution on its finite box.

The sign in the reciprocal formulas is already visible for a single divided difference.
\begin{lemma}\label{lem:reciprocal-partial}
For the $x,y$ reciprocal with fixed $z$-cap,
\begin{equation}\label{eq:R-partial-R}
 R_{M,r}\partial_iR_{M,r}=-xy\partial_i.
\end{equation}
\end{lemma}
\begin{proof}
The symmetric prefactor $x^My^Mz^r$ factors out of the divided
difference. A second reciprocal cancels it and replaces the denominator
by $x^{-1}-y^{-1}=-(x-y)/(xy)$.
\end{proof}

Combining this sign with the mixed-cap multiplication identities proves the local conjugations used in the packet constructions.
\begin{proposition}\label{prop:local-reciprocal-identities}
The identities \eqref{eq:local-lascoux-conjugation} and
\eqref{eq:local-D-conjugation} hold.
\end{proposition}
\begin{proof}
Multiplication operators on the Laurent ring satisfy
\begin{equation}\label{eq:mixed-cap-multiplication}
 R_{M,r+1}xR_{M,r}=z/x,\qquad
 R_{M,r+1}yR_{M,r}=z/y,\qquad R_{M,r+1}zR_{M,r}=1.
\end{equation}
The identities
\begin{equation}\label{eq:pi-expansion}
 \pi_i=1+y\partial_i,\qquad \partial_i(xyf)=xy\partial_i f
\end{equation}
give
\begin{equation}\label{eq:K-expanded}
 \mathscr K_i^{(z)}=z+y(x+z)\partial_i.
\end{equation}
For $g=R_{M,r}f$, \cref{lem:reciprocal-partial} equivalently says
\begin{equation}\label{eq:partial-g-inverted}
 (\partial_i g)(x^{-1},y^{-1},z^{-1})
 =-x^{1-M}y^{1-M}z^{-r}\partial_i f.
\end{equation}
Thus the two terms of \eqref{eq:K-expanded}, after applying
$R_{M,r+1}$, become $f$ and $-(x+z)\partial_i f$, proving
$R_{M,r+1}\mathscr K_i^{(z)}R_{M,r}=1-(x+z)\partial_i$.
Similarly, $\partial_i(yf)=-f+x\partial_i f$ gives
\begin{equation}\label{eq:D-expanded}
 \mathscr D_i^{(z)}=-\mathrm{id}+(x+z)\partial_i.
\end{equation}
Its terms contribute $-zf$ and $-y(x+z)\partial_i f$, respectively.
Hence
\[
 R_{M,r+1}\mathscr D_i^{(z)}R_{M,r}=-\mathscr K_i^{(z)}.\qedhere
\]
\end{proof}

\begin{remark}\label{rem:appendix-monomial-check}
The final conjugations are polynomial identities, not merely identities
in a truncated Chow ring. Intermediate multiplication identities may
be Laurent; the finite-box assumptions ensure nonnegative exponents
in the capped packet expressions.
\end{remark}

\section{Weighted and relative extensions}\label{sec:weighted-relative}
The main construction already allows every balanced nonnegative integral
profile. We record here its explicit balance criterion, reflected-support
consequences, and the initial-tail conditions for relative envelopes.
Writing $\epsilon=\boldsymbol\kappa$, we use the equivalent notation
$P_{u,\boldsymbol\kappa,\mu}$ for the packet in
\eqref{eq:P-operator-word}.

\subsection{Balance and minimal terminal completion}
\label{subsec:weighted-balance}
For the nonnegative multiplicity profile
$\boldsymbol\kappa\in\NN^\ell$, write
$0<t_{j,1}<\cdots<t_{j,m_j}$ for the occurrences of colour $j<n$,
put $t_{j,m_j+1}=\ell+1$, and define
\begin{equation}\label{eq:weighted-interval-inflow}
 I_{j,k}(\boldsymbol\kappa)
 =\sum_{\substack{t_{j,k}<s<t_{j,k+1}\\i_s=j-1}}\kappa_s
 \qquad(1\le k\le m_j).
\end{equation}
\begin{proposition}\label{thm:weighted-balance}
The graph $\mathfrak C(\mathbf i,\boldsymbol\kappa,\mathbf d)$ is
balanced exactly when, for every occurring colour $j<n$,
\begin{align}
 I_{j,k}(\boldsymbol\kappa)&\ge\kappa_{t_{j,k}}
       &&(1\le k<m_j),\label{eq:weighted-internal-balance}\\
 d_j+I_{j,m_j}(\boldsymbol\kappa)&\ge\kappa_{t_{j,m_j}}.
       \label{eq:weighted-terminal-balance}
\end{align}
A completion by terminal root edges exists if and only if the internal
inequalities \eqref{eq:weighted-internal-balance} hold. When they do,
the componentwise least completion is
\[
 d_j^{\min}(\boldsymbol\kappa)
 =\max\{0,\kappa_{t_{j,m_j}}-I_{j,m_j}(\boldsymbol\kappa)\}
\]
for occurring colours, and zero for the remaining ranks, including $n$.
\end{proposition}
\begin{proof}
The state created at $t_{j,k}$ has outgoing multiplicity
$\kappa_{t_{j,k}}$. Its incoming ordinary edges are exactly the
colour-$(j-1)$ steps before the next creation of rank $j$, with total
weight $I_{j,k}$. Only the final state receives the $d_j$ root edges.
Initial states and rank $n$ have no outgoing edges. This proves the
criterion and shows that terminal supply cannot repair an internal
deficit. Under the internal inequalities, each $d_j$ is independently
bounded below by zero and the stated terminal deficit.
\end{proof}

Thus the internally admissible profiles form the cone
\[
 \mathfrak F(\mathbf i)=
 \{\boldsymbol\kappa\in\RR_{\ge0}^{\ell}:
    \kappa_{t_{j,k}}\le I_{j,k}(\boldsymbol\kappa)
                 \text{ for all }j<n,\ k<m_j\}.
\]
For canonical-row words both uniform profiles belong to this cone, by
\cref{thm:reservoir}. Each integral point gives a balanced packet by
setting $d_j=d_j^{\min}$ and $\mu_j=\sum_{r=j}^n d_r$ in
\cref{thm:canonical-packet}. The generated bundle is canonical for the
fixed cap and labelled states, not a canonical isomorphism inferred
from equality in $K^0$.

\subsection{Differential operators and reflected supports}

For a weighted packet, its realized finite complement identifies an ordinary differential preserver.
\begin{corollary}
\label{pI-cor:weighted-preserver}
Let $P=P_{u,\boldsymbol\kappa,\mu}$ be a balanced weighted packet.
For every field $\kk$, its ordinary constant-coefficient differential
operator preserves realizable-volume polynomials over $\kk$:
\[
 f\in\RV\quad\Longrightarrow\quad
 P(\partial_{x_1},\ldots,\partial_{x_n},\partial_z)f\in\RV.
\]
Additional spectator variables in $f$ are permitted. A zero output is
included. The derivative symbols here are ordinary derivatives, not
divided differences.
\end{corollary}
\begin{proof}
Write $y=(x_1,\ldots,x_n,z)$ and choose a containing exponent cap
$\mathbf m$. By \cref{thm:canonical-packet},
$P(\partial_y)y^{[\mathbf m]} =\Dcomp_{\mathbf m}\Nrm_y(P)\in\RV$.
Thus $P(\partial_y)$ is a realizable covolume operator in the sense of
\cite[Definition~1.2]{GHMSSW25}. Apply
\cite[Theorem~1.3 and Remark~1.14]{GHMSSW25}.
The zero case is immediate, and the theorem permits spectator variables.
\end{proof}

Although the original weighted coefficients are not asserted to have a volume model, their support is obtained by reflection.
\begin{corollary}
\label{pI-cor:weighted-support}\label{prop:weighted-original-support}
Every nonzero balanced weighted packet $P$ has nonnegative integral
coefficients, $M$-convex support, and a saturated Newton polytope.
More precisely, let $J=\{1,\ldots,n+1\}$ index its variables, let
$\mathbf m$ contain its support, and let $\rho$ be the rank function
of the complemented support
$\mathbf m-\Supp(P)=B(\rho)$. Then
\[
 \Supp(P)=B(\rho_{\mathbf m}^{\dagger}),\qquad
 \rho_{\mathbf m}^{\dagger}(S)
 =\mathbf m(S)-\rho(J)+\rho(J\setminus S).
\]
In particular,
$\Supp(P)=\Newt(P)\cap\ZZ^{n+1}$.
\end{corollary}
\begin{proof}
Divided differences preserve integrality. The finite complement is
realizable by \cref{thm:canonical-packet}, so its coefficients, and hence
those of $P$, are nonnegative. By
\cref{thm:algebraic-polymatroid-consequence}, the complemented support
$A=\mathbf m-\Supp(P)$ is $M$-convex. Reflection preserves symmetric
exchange: if $x_i>y_i$ in $\mathbf m-A$, apply exchange to
$\mathbf m-y$ and $\mathbf m-x$ and reflect the exchanged pair back;
see also \cite[proof of Proposition~16]{HMMSD22}. Thus $\Supp(P)$ is
$M$-convex and saturated \cite{Murota03}. Finally,
\[
 \max_{x\in\Supp(P)}x(S)
 =\mathbf m(S)-\min_{\eta\in B(\rho)}\eta(S)
 =\mathbf m(S)-\rho(J)+\rho(J\setminus S),
\]
since every base has total size $\rho(J)$.
\end{proof}

The reflected rank is independent of the cap: replacing $\mathbf m$ by
$\mathbf m+\nu$ replaces $\rho$ by $\rho+\nu$, leaving the formula
unchanged. Compare arbitrary caps through their componentwise maximum.

The next refinement concerns the support, not the coefficients. We state
its characteristic-zero reflection principle for algebraic polymatroids,
then apply it to the complemented support of a weighted packet.
The proof retains the field of representation explicitly.

\Needspace{7\baselineskip}
\begin{proposition}
\label{prop:weighted-support-charzero}
Let $\kk$ have characteristic zero, let $\rho$ be an integral polymatroid
rank function on a finite set $J$, algebraic over $\kk$, and let
$\mathbf m\in\NN^J$ satisfy $m_i\ge\rho(\{i\})$ for every $i\in J$.
Then
\[
 \rho_{\mathbf m}^{\dagger}(S)
 :=\mathbf m(S)-\rho(J)+\rho(J\setminus S),\qquad S\subseteq J,
\]
is an integral polymatroid rank function, algebraic over $\kk$ and
linearly representable over a finite algebraic extension of $\kk$.
If $\kk$ is algebraically closed, it is linearly representable over
$\kk$ itself.

In particular, the polymatroid with integral base set $\Supp(P)$ has
these properties for every nonzero balanced weighted packet $P$.
\end{proposition}
\begin{proof}
The empty case is immediate. Complementary submodularity and
$\rho_{\mathbf m}^{\dagger}(\varnothing)=0$ are direct; for $i\notin S$
the marginal is
$m_i-\rho(J\setminus S)+\rho(J\setminus(S\cup\{i\}))
\ge m_i-\rho(\{i\})\ge0$.
Choose fields $K_i$ representing $\rho$. Replacing each by the field
of a finite transcendence basis preserves all subset ranks, since the
original composita are algebraic over the replacements. Their compositum
$L/\kk$ is now finitely generated. In $\Omega_{L/\kk}$ set
$U_i=\Span_L(dK_i)$ and $U=\sum_iU_i$. Then
\begin{equation}\label{eq:weighted-differential-ranks}
 \dim_L\sum_{i\in S}U_i=\operatorname{trdeg}_\kk K_S=\rho(S).
\end{equation}
Indeed, product and quotient rules give $\sum_{i\in S}U_i=\Span_L(dK_S)$.
Extend a transcendence basis of $K_S$ to one of $L$. Characteristic
zero makes its differentials independent, while differentiating separable
minimal polynomials spans $dK_S$ by the first basis; compare
\cite[Theorem~1.7 and \S2]{RST25}.

Choose $A:E=\bigoplus_iL^{m_i}\twoheadrightarrow U$ with
$A(E_i)=U_i$ and put $N=\ker A$. Writing $E_S=\bigoplus_{i\in S}E_i$,
we have $\dim(N\cap E_{J\setminus S})=
\mathbf m(J\setminus S)-\rho(J\setminus S)$, hence
\begin{equation}\label{eq:weighted-block-kernel-rank}
 \dim_L\operatorname{pr}_S(N)
 =\mathbf m(S)-\rho(J)+\rho(J\setminus S)
 =\rho_{\mathbf m}^{\dagger}(S).
\end{equation}
The dual coordinate subspaces $V_i=\operatorname{im}(E_i^*\to N^*)$
therefore represent the reflected rank over $L$.

Take a block-column matrix spanning these subspaces and the nonzero
finitely generated $\kk$-algebra $R\subset L$ of its entries, with one
nonzero maximal-rank minor inverted for every positive-rank block union.
Larger minors vanish in $R$ and chosen minors are units, so specialization
at any maximal ideal preserves all block ranks. Its residue field
$\kk'/\kk$ is finite algebraic by
\cite[Theorem~10.34.1, Tag~00FV]{Stacks}, and equals $\kk$ when $\kk$ is
algebraically closed. This proves the linear-representation assertion.
In independent variables $t_1,\ldots,t_q$, let $K'_i$ be the field over
$\kk'$ generated by the linear forms of the specialized block $i$.
For nonempty $S$, its compositum has transcendence degree over $\kk$
equal to its degree over $\kk'$, namely
$\dim_{\kk'}\sum_{i\in S}V'_i=\rho_{\mathbf m}^{\dagger}(S)$.
For $S=\varnothing$ take the compositum to be $\kk$. This proves
algebraicity over the original field, also in rank zero.

For a packet, \cref{thm:canonical-packet,thm:algebraic-polymatroid-consequence}
give algebraicity of $\mathbf m-\Supp(P)=B(\rho)$ and
$\rho(\{i\})\le m_i$. Apply the result and
\cref{pI-cor:weighted-support}.
\end{proof}

\begin{remark}\label{rem:weighted-complement-support}
The differential-preserver and support conclusions concern completed
balanced packets; they do not make individual divided-difference
factors universal preservers. They also do not prove that $\Nrm(P)$
is realizable or Lorentzian. Incidence directly proves algebraicity
of $\mathbf m-\Supp(P)$ over every field. The separate argument in
\cref{prop:weighted-support-charzero} proves algebraicity of the
original support in characteristic zero. No unrestricted
algebraic-polymatroid duality in positive characteristic is asserted.
\end{remark}

\subsection{Relative row and co-row envelopes}
\label{pI-sec:relative-rows}
The same generation argument applies to nonsplit initial quotient
flags when the relevant initial tails are globally generated.

\begin{proposition}
\label{pI-prop:relative-row}
Let $X$ be an integral projective variety over $\kk$, and let $\cE$
be a globally generated rank-$n$ bundle with a quotient flag
$\cE=\cF_n\twoheadrightarrow\cdots\twoheadrightarrow \cF_0=0$,
$\rk \cF_i=i$. Form the quotient-flag tower of \cref{sec:canonical}
for a word decomposed into increasing consecutive rows
$B_h=(a_h,a_h+1,\ldots,b_h)$, where $1\le a_h\le b_h<n$ and
$a_1\ge\cdots\ge a_s$. Assume that each initial tail
$\ker(\cE\to \cF_{a_h-1})$ is globally generated.
Then the quotient lines created in each row are the successive factors
of a globally generated row bundle $\cG_h$ on the final tower.
One may tensor $\cG_h$ by a separately chosen globally generated line
bundle $\cM_h$ pulled back from $X$.
For any balanced nonnegative integral kernel profile and nonnegative
terminal multiplicities, the product of these twisted row top Chern
classes, kernel-root powers, and terminal top Chern classes is the top
Chern class of an actual globally generated bundle.
\end{proposition}
\begin{proof}
The tower remains integral and projective, and every prefix is a
quotient of pulled-back $\cE$. Previous colours are at least $a$, so
before a row starting at $a$ the quotient $\cE\to \cF_{a-1}$ is unchanged.
After its $j$th step, put
$\cG_{h,j}=\ker(\cF_{a+j-1}^{\rm new}\to \cF_{a-1})$, $\cG_{h,0}=0$.
Compatible quotients give
$0\to \cQ_{h,j}\to \cG_{h,j}\to \cG_{h,j-1}\to0$ and an actual surjection
$\ker(\cE\to \cF_{a-1})\twoheadrightarrow \cG_{h,j}$, with earlier objects
pulled back. Thus the row bundles are generated. For the completed
row, pulled through later stages, Whitney gives
\[
 c_{b_h-a_h+1}(\cG_h\otimes \cM_h)
 =\prod_{j=1}^{b_h-a_h+1}(q_{h,j}+c_1(\cM_h)).
\]
Apply \cref{thm:chern-flow} to the generated prefixes to absorb kernel
and terminal factors, and take the direct sum with these twisted rows.
\end{proof}

For old upper roots, preserving the additional prefix requires strictly decreasing row starts.
\begin{proposition}
\label{pI-prop:relative-corow}
Keep $X$, $\cE$, the initial quotient flag, and the row decomposition
of \cref{pI-prop:relative-row}, but replace its tail and ordering
hypotheses by $a_1>\cdots>a_s$ and that every initial tail
$\ker(\cE\to \cF_{a_h})$ is globally generated.
At the start of row $B_h=(a,\ldots,b)$, the bundles
$\cC_{h,j}=\ker(\cF_{a+j}\to \cF_a)$, $0\le j\le b-a+1$, are globally
generated and filter its old upper roots. Their independently twisted
top Chern classes, together with the terminal prefix top Chern classes,
therefore have an actual globally generated direct-sum envelope.
\end{proposition}
\begin{proof}
Strictly larger preceding row starts leave $\cF_a$ unchanged.
The initial tail $\ker(\cE\to \cF_a)$ surjects onto each $\cC_{h,j}$, and
$0\to \cR_{a+j}\to \cC_{h,j}\to \cC_{h,j-1}\to0$ is exact at row start.
Before the $j$th step, the preceding colours within the row are at most
$a+j-2$, so none changes the root of index $a+j$.
This is precisely the old upper root used at colour $a+j-1$.
Whitney multiplicativity and the globally generated terminal prefixes
complete the construction.
\end{proof}

For either relative envelope $\cV$ on $p:Y\to X$ and any projective
$f:X\to B=\prod_i\PP^{m_i}_{\kk}$, let $\Xi$ be the reduced Chow
representative of $(f\circ p)_*(c_R(\cV)\cap[Y])$, $R=\rk\cV$.
Then \cref{thm:incidence-dual} gives
$\Dcomp_{\mathbf m}\Nrm(\Xi)\in\RV$; if nonzero it has degree
$\dim Y-R$ and a smooth realization in characteristic zero.

\begin{remark}
The strict inequality in the co-row proposition is essential for this
argument. For $\cE=\cO^{\oplus2}$ and the repeated word $(1,1)$, the
first modification gives $\cF_1=\cO_{\PP^1}(1)$; the next co-row tail is
$\ker(\cO^{\oplus2}\to\cO(1))=\cO(-1)$, which is not generated.
The construction with distinct row twists is attached to its specified
row word; it does not assert a braid-independent multivariate family.

Nor does factorial normalization commute with identifying parameter
variables. For instance,
$\left.\Nrm_{u,v}((u+v)^2)\right|_{u=v=z}=3z^2, \qquad \Nrm_z(4z^2)=2z^2$.
To recover a common-twist packet, identify the line twists geometrically
before forming the coefficient normalization, or use its own coefficient
identity. One cannot infer that identity by an unqualified substitution
in a differently normalized multi-parameter polynomial.
\end{remark}

\section{Chern-flow refinements and flow cones}\label{subsec:flow-cone-geometry}
The flow identity has further consequences for characteristic classes and
for the integral cone of admissible weights.
We first recover exterior-power classes and an actual filtered kernel
from the flow identity, then describe the cone of admissible weights.

\subsection{Exterior powers and pathwise exact sequences}
The same $K^0$ identity gives, for
$\lambda_u(\cE)=\sum_j[\wedge^j\cE]u^j$,
\begin{equation}\label{eq:lambda-divergence}
 \lambda_u\!\left(\bigoplus_v\cF_v^{\oplus b_v}\right)
 =\prod_v\lambda_u(\cF_v)^{r_v}
                   \prod_e\lambda_u(\cE_e)^{\kappa_e}.
\end{equation}
Apply the multiplicative total $\lambda$-operation to
\eqref{eq:flow-K0}. Since all displayed bundles have finite rank,
the identity is polynomial and can be evaluated at $u=-1$.
Applying it first to the dual bundles gives the corresponding
$K$-theoretic Euler identity for $e_K(\cE)=\lambda_{-1}(\cE^\vee)$.
These identities do not replace the generated Chow envelope needed
for volume realization.

A path decomposition gives an actual filtered kernel, although it does not make the choice of surjection canonical.
\begin{proposition}
\label{pI-prop:pathwise-sequence}
Under the hypotheses of \cref{thm:chern-flow}, set
$\cU=\bigoplus_v\cF_v^{\oplus b_v}$ and
$\cF(r)=\bigoplus_v\cF_v^{\oplus r_v}$.
There is a generally noncanonical exact sequence
\[
 0\longrightarrow \cK_{\mathcal P}\longrightarrow \cU
   \longrightarrow \cF(r)\longrightarrow0
\]
such that $\cK_{\mathcal P}$ has a filtration with successive factors
$\cE_e$, each repeated $\kappa_e$ times, with zero factors omitted.
No acyclicity assumption is required.
\end{proposition}
\begin{proof}
On a directed cycle, the nonnegative kernel ranks sum to
$\sum_e(\rk \cF_{t(e)}-\rk \cF_{s(e)})=0$. Thus every cycle kernel is
zero and its edge maps are isomorphisms.
Add a source with supply $r_v$ and a sink with drain $b_v$ at each
vertex. The divergence identity is conservation of integral flow.
Following positive edges and subtracting the least multiplicity whenever
a cycle closes or the sink is reached decomposes it into paths and
cycles. Each subtraction removes a positive edge; after exhausting
source flow the remainder consists of cycles. Discard these, since
they change no supply or drain and have only zero kernel factors.
A retained path from $v$ to $w$ gives $\cF_w\twoheadrightarrow \cF_v$
whose kernel has the inverse-image filtration by its edge kernels.
Direct sum over paths, with their multiplicities, gives the claimed
sequence; paths with no ordinary edge contribute identity maps.
\end{proof}
The bundle $\cU$ is fixed by the divergence data. The path matching,
surjection, and filtered kernel are not asserted to be canonical.

\subsection{Path cones and primitive generators}

Let $\Gamma=(V,E)$ be a finite acyclic directed multigraph and
$S\subseteq V$ its allowed root-supply vertices. For $x\in\RR^E_{\ge0}$,
write $\operatorname{in}_x(v)=\sum_{t(e)=v}x_e$ and
$\operatorname{out}_x(v)=\sum_{s(e)=v}x_e$, and set
\begin{equation}\label{eq:source-admissible-cone}
 C(\Gamma,S)=\{x\in\RR^E_{\ge0}:\operatorname{in}_x(v)
                  \ge\operatorname{out}_x(v)\text{ for }v\notin S\}.
\end{equation}
A nonempty directed path is $S$-primitive if it starts in $S$ and
has no other nonterminal vertex in $S$. Its incidence vector $\chi_P$
distinguishes parallel edges.

The admissible weights are generated by paths, and their indecomposable integral generators can be determined explicitly.
\begin{proposition}\label{prop:path-cone-geometry}
The cone is generated by paths beginning in $S$, and
\[
 C(\Gamma,S)\cap\ZZ^E
 =\NN\{\chi_P:P\text{ is a directed path beginning in }S\}.
\]
Its extreme rays are exactly $\RR_{\ge0}\chi_P$ for $S$-primitive
paths. These vectors form the unique minimal Hilbert basis.
The integral semigroup is normal, and its algebra over any field is
Cohen--Macaulay.
\end{proposition}
\begin{proof}
For $0\ne x\in C(\Gamma,S)$, the positive-edge subgraph is a nonempty
DAG. A source in it has positive outgoing but no incoming flow and
therefore belongs to $S$. Follow positive edges to a sink, obtaining
$P$, and subtract $\lambda\chi_P$, where $\lambda=\min_{e\in P}x_e$.
Internal divergences are unchanged; the initial vertex lies in $S$,
and the terminal vertex has no positive outgoing flow. The remainder
stays in the cone and has fewer positive edges. Iteration decomposes
$x$ into paths, with integral coefficients when $x$ is integral.

A path with an intermediate vertex in $S$ splits there. Conversely,
if $P$ is $S$-primitive and $\chi_P=x+y$ in the cone, both summands
vanish off $P$. Their successive edge weights are nonincreasing, by
the internal divergence inequalities, and sum coordinatewise to one.
They are therefore constant, proving extremality. The decomposition
shows that there are no other extreme rays. Splitting at intermediate
vertices of $S$ proves integral generation by primitive paths; each
is an indecomposable primitive lattice vector on its extreme ray,
so the Hilbert basis is unique. Finally,
$\mathsf S=C(\Gamma,S)\cap\ZZ^E$ is saturated: if
$g\in\operatorname{gp}(\mathsf S)$ and $qg\in\mathsf S$ for $q>0$,
then $g\in\ZZ^E\cap C(\Gamma,S)$. Hochster's theorem gives
Cohen--Macaulayness \cite{Hoch72}.
\end{proof}

For the ordinary-edge creation-state DAG $\Gamma$ and final states
$S=\{\sigma_j(\ell)\}$, the interval inequalities in
\cref{thm:weighted-balance} identify $\mathfrak F(\mathbf i)$ with
$C(\Gamma,S)$. Thus its integral semigroup is normal, with extreme
rays and minimal Hilbert basis given by $S$-primitive paths.
The least terminal completion and an aggregated integral path
decomposition are computed in topological order; at most $|E|$ paths
have positive multiplicity, since each greedy subtraction removes a
positive edge.

With integral upper and lower bounds, total unimodularity gives the corresponding integer-decomposition property.
\begin{proposition}\label{thm:bounded-flow-idp}
Let $A$ be a row submatrix of a node--arc incidence matrix. With integral
bounds, every nonempty bounded polytope
\[
 P=\{x\in\RR^E:\ell_E\le x\le u_E,\ \ell_V\le Ax\le u_V\}
\]
has the integer-decomposition property: each
$x\in mP\cap\ZZ^E$ is a sum of $m$ points in $P\cap\ZZ^E$.
Its Ehrhart semigroup is normal and its semigroup algebra is
Cohen--Macaulay.
\end{proposition}
\begin{proof}
Every matrix formed by stacking rows of $\pm A$ and $\pm I$ is totally
unimodular, so $P$ is integral. Induct on $m$. For $m>1$ and
$x\in mP\cap\ZZ^E$, the bounded polytope
$Q_x=\{y\in P:x-y\in(m-1)P\}$ is nonempty since it contains $x/m$.
Besides the bounds for $P$, its inequalities are
\[
 \begin{aligned}
 x-(m-1)u_E&\le y\le x-(m-1)\ell_E,\\
 Ax-(m-1)u_V&\le Ay\le Ax-(m-1)\ell_V.
 \end{aligned}
\]
Its matrix is totally unimodular and its right-hand sides integral,
so it has an integral vertex $y$. Decompose $x-y$ by induction.
The Ehrhart semigroup
$\mathsf E(P)=\operatorname{cone}(\{1\}\times P)\cap(\ZZ\times\ZZ^E)$
is saturated, hence normal; the proved property says it is generated
in degree one. Apply Hochster's theorem \cite{Hoch72}.
\end{proof}

\section{Functoriality and minimal polarized completions}
\label{sec:moment-functoriality}
This appendix separates the algebraic behavior of Chern moments from
positivity certificates. The class-valued Hankel pairing controls
functoriality and nonflat base change. We then treat insertions,
compression, and minimal polarized completions, and conclude with the
distinction between positive moments and a generated bundle on a fixed
surface.

\subsection{Hankel maps and base change}\label{fn-sec:hankel}
Here $A$ is a finite-dimensional commutative graded Poincar\'e-duality
algebra over a characteristic-zero field $k$, with $A^0=k$, formal
dimension $d$, and degree $\lambda_A$, extended by zero. Choose finite
caps $r_i\ge0$ and coefficients $C_{i,a}\in A^a$, $C_{i,0}=1$,
zero outside $0\le a\le r_i$. No positivity is assumed. Define
\begin{equation}\label{fn-eq:truncated}
 S_A=A[z_1,\ldots,z_m]/(z_i^{r_i+1}:1\le i\le m),\qquad \deg z_i=1,
\end{equation}
and the $A$-linear, generally nonmultiplicative map
$T_C(az^\alpha)=a\prod_iC_{i,\alpha_i}$. The finite caps make this
compatible with out-of-range vanishing. Its Hankel map is
\begin{equation}\label{fn-eq:hankel}
 \mathsf H_C:S_A\longrightarrow S_A^*=\operatorname{Hom}_A(S_A,A),
 \qquad \mathsf H_C(f)(g)=T_C(fg).
\end{equation}
The entries are $(\mathsf H_C)_{\alpha,\beta}
=\prod_iC_{i,\alpha_i+\beta_i}$: a class-valued matrix, not a real
positive-semidefinite moment matrix. We also write $\mathcal R_C$ for $R_C$.

The class-valued pairing gives functoriality and identifies exactly the
possible failure of tensor base change.
\begin{proposition}\label{fn-thm:hankel}
The homogeneous ideal $J_C=\ker\mathsf H_C$ defines the scalar moment
algebra $R_C=S_A/J_C$, and
\begin{equation}\label{fn-eq:class-test}
 f\in J_C\quad\Longleftrightarrow\quad
 T_C(fz^\beta)=0\text{ in }A\text{ for all }0\le\beta_i\le r_i.
\end{equation}
For a unital graded map $\phi:A\to A'$ between Poincar\'e-duality
coefficient algebras, possibly of different top degrees, put
$C'_{i,a}=\phi(C_{i,a})$. There is a canonical graded map
$R_C\to R_{C'}$, $a\mapsto\phi(a)$, $z_i\mapsto z_i$, respecting
composition and injective or surjective when $\phi$ is so.
For $Q_C=\coker\mathsf H_C$, the natural surjective algebra map under
tensor base change fits into an exact sequence of graded $A'$-modules
\begin{equation}\label{fn-eq:tor}
 0\longrightarrow\Tor_1^A(Q_C,A')\longrightarrow R_C\otimes_AA'
        \longrightarrow R_{C'}\longrightarrow0.
\end{equation}
It is an isomorphism exactly when this Tor group vanishes, in
particular when $A'$ is flat over $A$.
\end{proposition}
\begin{proof}
Associativity makes $J_C$ a homogeneous ideal contained in the scalar
annihilator $I_C$. Conversely, for $f\in I_C$ and $a\in A,g\in S_A$,
$\lambda_A(aT_C(fg))=\Lambda_C(fag)=0$. Duality in $A$ gives
$T_C(fg)=0$, so $I_C=J_C$; $A$-linearity reduces tests to monomials.
The class tests commute with $\phi$, proving functoriality. Under an
injection a nonzero test remains nonzero; under a surjection all
coefficients lift.
Finally, $S_{A'}=S_A\otimes_AA'$ and
$\mathsf H_{C'}=\mathsf H_C\otimes_AA'$, using the finite free dual.
Identify $R_C=\im\mathsf H_C$ as a module and tensor
$0\to R_C\to S_A^*\to Q_C\to0$. Since $S_A^*$ is finite free,
the Tor sequence \cite[Tag~00M0]{Stacks} gives the displayed kernel.
Tensoring $S_A\twoheadrightarrow R_C$ identifies the image with
$\im\mathsf H_{C'}=R_{C'}$ and its quotient algebra map.
\end{proof}
Zero-padding the caps adds only annihilated monomials and does not
change the quotient. These algebraic statements do not prescribe
Lefschetz cones.

\begin{example}\label{fn-ex:nonflat}
Use \eqref{hd-eq:P2-example} over $\QQ$, and restrict
$A=\QQ[h]/(h^3)$ to $A'=A/(h^2)$, with $\lambda_{A'}(h)=1$.
Then $C'(t)=1+2ht$ and
$R_{C'}=\QQ[h,z]/(h^2,z-2h),\qquad R_C\otimes_AA'=\QQ[h,z]/(h^2,hz,z^2)$.
Their dimensions are two and three; the extra square-zero kernel is
$\QQ(z-2h)$, so the Tor group has dimension one. Geometrically this
is restriction from $\PP^2$ to a line, not flat base change of moments.
\end{example}
By contrast, for $A'=A[u]/(u^2)$ one has
$R_{C'}=R_C[u]/(u^2)$, with Hilbert vector $(1,3,3,1)$ in this example.
More generally, a monic extension of $A$ is free, and its defining
relation becomes the same relation over $R_C$. Freeness alone gives
no positive cone for the new generator.

\subsection{Insertions and changes of presentation}\label{fn-sec:operations}
Inserting a fixed nonzero base class has a different, exact compatibility that does not require flatness.
\begin{proposition}\label{fn-thm:insert}
For $0\ne\theta\in A^e$, let
$A_\theta=A/\Ann_A\theta$, with degree
$\lambda_\theta[a]=\lambda_A(\theta a)$, and let $\bar C$ be the
image arrays. This is a duality algebra of top degree $d-e$, and
canonically
\begin{equation}\label{fn-eq:insertion}
 \mathcal R_{A_\theta,\bar C}
 \simeq\mathcal R_{A,C}/\Ann_{\mathcal R_{A,C}}(\theta).
\end{equation}
No flatness is required.
\end{proposition}
\begin{proof}
Duality identifies the annihilator of $a\mapsto\lambda_A(\theta a)$
with $\Ann_A\theta$. A polynomial vanishes in the left side precisely
when $T_C(fz^\beta)\in\Ann_A\theta$ for every $\beta$, equivalently
$T_C(\theta fz^\beta)=0$ for every $\beta$. By the class test this
means $\theta[f]=0$ in $R_C$. Surjectivity follows from
\cref{fn-thm:hankel}; the weighted degrees agree by definition.
\end{proof}
For $I=\Ann_A\theta$, ordinary tensor base change instead gives
$R_C/(IR_C)$, and $IR_C\subseteq\Ann_{R_C}\theta$. The discrepancy
is the Tor kernel, not radicalization; for $\theta=h$ in
\cref{fn-ex:nonflat}, insertion imposes $z=2h$. When $\theta$ is a
nonzero product of boundary polarizations, descent also gives the
Hodge package. Fixed insertion and active-index shift have different
moments, $\deg_A(\theta a\prod_iC_{i,\beta_i})$ and
$\deg_A(a\prod_iC_{i,\alpha_i+\beta_i})$; no identity
$c_{a+b}(\cE)=c_a(\cE)c_b(\cE)$ is used.

For separate arrays on coefficient algebras $A_1,A_2$ over $k$,
with tensor product degree, there is also a natural identification
\[
 \mathcal R_{A_1\otimes A_2,(C^{(1)},C^{(2)})}
 \simeq\mathcal R_{A_1,C^{(1)}}\otimes_k\mathcal R_{A_2,C^{(2)}}.
\]
Indeed, the functional is a tensor product, and the tensor product of
the two duality quotients has a perfect pairing, leaving no further
annihilator. The coefficient algebra here is $A_1\otimes A_2$, not in
general the full real diagonal Hodge ring of a product of varieties.
Marked uniqueness was proved in \cref{hd-prop:intrinsic-PD}.
None of these algebraic maps automatically preserves a chosen
Lefschetz cone.

The auxiliary projective-bundle ranks can be reduced without changing the descended moments.
\begin{proposition}\label{fn-prop:compression}
Let $A$ have formal dimension $d$, set $s_0=\max\{2,d\}$, and write
$K(t)=1+k_1t+\cdots+k_dt^d$. For $S\ge s\ge s_0$, put
$p_s(\xi)=\xi^s+k_1\xi^{s-1}+\cdots+k_d\xi^{s-d}$ and
$B_s=A[\xi]/(p_s)$. There is a canonical isomorphism
\[
 B_S/\Ann_{B_S}(\xi^{S-s})\simeq B_s,
\]
preserving the weighted degree $\deg_{B_S}(\xi^{S-s}\,\cdot)$ and
the fibre-top normalization. If $B_S$ is positive with $\xi$ boundary,
$B_s$ has the induced K\"ahler package. For our geometric and KM
certificates, every auxiliary rank can thus be replaced by $s_0$
at the certificate-algebra level.
\end{proposition}
\begin{proof}
Since $p_S=\xi^{S-s}p_s$ and multiplication by the formal monomial
is injective in $A[\xi]$, one has
$\xi^{S-s}f\in(\xi^{S-s}p_s)\Longleftrightarrow f\in(p_s)$.
Multiplying the basis $1,\xi,\ldots,\xi^{s-1}$ by $\xi^{S-s}$
identifies fibre-top extraction. The descended class is nonzero since
$S-s<S$, so boundary descent gives positivity.

Evaluation presentations enlarge by zero generators, adding trivial
kernel summands; KM certificates enlarge by zero-weight coloops of
the dual matroid. Both preserve $K(t)$ and allow $S\ge s_0$.
Compress one variable at a time; the other monic relations preserve
freeness and boundary membership descends. No actual evaluation
kernel of rank $s_0$ is asserted.
\end{proof}

Inserting relative monomials instead shifts the active Chern indices; maximal shifts reduce to a base quotient.
\begin{proposition}\label{fn-prop:corners}
If $C$ has a positive inverse-Chern certificate and $z^\alpha\ne0$,
then $R_C/\Ann(z^\alpha)$ has the K\"ahler package of formal
dimension $d-|\alpha|$, with moments
$az^\beta\mapsto\deg_A(a\prod_iC_{i,\alpha_i+\beta_i})$.
If $\alpha$ is coordinatewise maximal among indices with
$\prod_iC_{i,\alpha_i}\ne0$ in $A$, then
\[
 R_C/\Ann(z^\alpha)\simeq
 A/\Ann_A\left(\prod_iC_{i,\alpha_i}\right)
\]
with weighted degree and the induced Hodge structure.
\end{proposition}
\begin{proof}
Descend by the boundary classes $z_i$; the degree formula is the moment
identity. At a maximal index every moment with an additional relative
power is zero. Hence all $z_i$ act as zero after descent, and $A$
surjects onto the quotient. Its kernel is the set of $a$ with
$\deg_A(ab\prod_iC_{i,\alpha_i})=0$ for every $b\in A$, namely the
stated annihilator by duality. The Chern product is nonzero by hypothesis.
\end{proof}

\subsection{Minimal polarized completions}\label{pr-sec:reduction}

Let $R=\bigoplus_{p=0}^dR^p$ have the K\"ahler package, and let
$S\subseteq R$ be a graded unital subalgebra containing
$L\in S^1\cap\mathcal K_R$. All duality subalgebras below have formal
dimension $d$ and the restricted degree. For $v\in P_L^p$ and
$0\le j\le d-2p$, the lowering operator is
\begin{equation}\label{pr-eq:lowering}
 \Lambda_L(L^jv)=j(d-2p-j+1)L^{j-1}v.
\end{equation}
It is linear, not generally a derivation. Set $S_0=S$ and let
$S_{n+1}$ be the graded algebra generated by $S_n\cup\Lambda_L(S_n)$.

Closure under the lowering operator constructs the smallest duality subalgebra containing the prescribed classes.
\begin{proposition}\label{pr-thm:hull}
This sequence stabilizes at the unique smallest graded
Poincar\'e-duality subalgebra $\mathcal H_R(S)$ containing $S$.
It has the K\"ahler package for
$\mathcal H_R(S)^1\cap\mathcal K_R$, is independent of the chosen
$L$, and needs at most $\dim R-\dim S$ strict enlargement steps.
\end{proposition}
\begin{proof}
Finite dimension gives stabilization. The stable algebra is preserved
by $L$, $\Lambda_L$, and grading, hence by the Lefschetz
$\mathfrak{sl}_2$ action. Complete reducibility decomposes it into
primitive strings of $R$. Restricted primitive forms remain definite;
strings of different primitive degrees are orthogonal, since the
complementary pairing contains a power of $L$ annihilating the higher
primitive degree. Strings of the same primitive degree pair through
$\deg_R(vwL^{d-2p})$. Thus the restriction has duality, Hard Lefschetz,
and Hodge--Riemann for $L$.

If a duality subalgebra $T$ contains $S$, ambient Hard Lefschetz makes
$L^{d-2p}:T^p\to T^{d-p}$ injective, and equal dimensions make it
bijective. Its primitive spaces lie in the ambient ones and inherit
the signs. Uniqueness of the Lefschetz decomposition implies
$\Lambda_L(T)\subseteq T$, so every $S_n\subseteq T$. This proves
minimality. The same argument for any interior $L'$ in the stable
algebra proves the cone assertion. Two initial choices give mutually
containing minimal algebras, hence the same hull. Each strict step
raises dimension by at least one.
\end{proof}

\begin{example}\label{pr-ex:hull}
Take
\[
 R=\RR[h_1,h_2,h_3,h_4]/(h_i^2),
\]
with top degree
$\deg_R(h_1h_2h_3h_4)=1$, $L=\sum_ih_i$, $D=h_3+2h_4$, and
$S=\RR[L,D]$. Squarefree expansion gives Hilbert vector
$(1,2,3,3,1)$; in degree three, $L^3,L^2D,LD^2$ are independent and
$D^3=0$. Here $\Lambda_L$ removes and sums factors, so
$\Lambda_L(D^2)=4(h_3+h_4)$ recovers $h_3,h_4$ together with $D$,
and then $a=h_1+h_2$. Thus
$\mathcal H_R(S)=\RR[a,h_3,h_4]\simeq\RR[a,b,c]/(a^3,b^2,c^2)$:
this containing duality algebra has $\deg_R(a^2bc)=2$ and Hilbert
vector $(1,3,4,3,1)$, of dimension $12<16$.
\end{example}
The hull is relative to $(R,S)$. Without an interior class in $S^1$,
one must first be chosen and adjoined. The hull need not be
degree-one-generated, so ordinary apolar reconstruction still requires
\cref{hd-thm:apolar}'s generation hypothesis.

Duality gives a useful projection onto the hull. More generally, for a
graded duality subalgebra $T\subseteq R$ of the same formal dimension,
containing an interior polarization, there is a unique graded map
$\Pi_T:R\to T$ characterized by
\begin{equation}\label{fn-eq:projection}
 \deg_R(\Pi_T(a)t)=\deg_R(at)\qquad(t\in T).
\end{equation}
Perfect pairings in $T$ give existence and uniqueness. Testing against
$t'\in T$ proves $\Pi_T(ta)=t\Pi_T(a)$, hence
$R=T\oplus T^\perp$ and $\Pi_T^2=\Pi_T$. Grading and
$L$-linearity preserve primitive strings, so $\Pi_T$ commutes with
$\Lambda_L$ for $L\in T^1\cap\mathcal K_R$.
For $T\subseteq U\subseteq R$, uniqueness gives
$\Pi_T^R=\Pi_T^U\Pi_U^R$.
The projection need not be multiplicative: in
$\RR[h_1,h_2]/(h_1^2,h_2^2)$, for $T=\RR[L]$ and $L=h_1+h_2$,
one has $\Pi_T(h_1)=L/2$ but $\Pi_T(h_1^2)=0\ne L^2/4$.

\subsection{Intrinsic positivity and a fixed source}
Even in dimension two, positivity of the moment algebra is weaker than the existence of a generated bundle with prescribed Chern classes.
\begin{proposition}\label{ip-thm:surface}
Let $A=\RR[h]/(h^3)$ with $\deg(h^2)=1$ and rank-marked data
$C(t)=1+ah\,t+bh^2t^2$. Its moment algebra has the K\"ahler package
for the image of $\{uh+vz:u,v>0\}$ if and only if
\[
 a\ge0,\qquad 0\le b\le a^2.
\]
\end{proposition}
\begin{proof}
The degree-one pairing is
$G=\left(\begin{smallmatrix}1&a\\a&b\end{smallmatrix}\right)$,
with $hz=ah^2$ and $z^2=bh^2$. If $b>a^2$, its two positive directions
contradict Hodge--Riemann. If $b<a^2$, it has signature $(1,1)$, so
in formal dimension two the package is equivalent to
$u^2+2auv+bv^2>0$ for all $u,v>0$. Letting $u/v\to0$ forces $b\ge0$;
if $a<0$, the choice $u=-av$ gives $(b-a^2)v^2<0$. Conversely,
$a,b\ge0$ ensure positivity. When $b=a^2$, the radical identifies
$z=ah$, giving $\RR[h]/(h^3)$. The prescribed cone consists of
$(u+av)h$ and avoids zero exactly when $a\ge0$.
\end{proof}

The distinction persists for integral data with nonnegative Schur classes.
\begin{proposition}\label{ip-prop:separation}
The data $a=3$, $b=1$ satisfy \cref{ip-thm:surface}, admit an integral
rank-two $K^0(\PP^2)$ lift, and have nonnegative Schur classes in degrees
at most two, but are not the Chern classes of a globally generated
rank-two bundle on $\PP^2$.
\end{proposition}
\begin{proof}
The degree-one eigenvalues are $4,-2$. The rank-two class
$[\cO]-[\cO(1)]+2[\cO(2)]$ has total Chern class
$(1+2h)^2/(1+h)=1+3h+h^2$; the Schur classes are
$s_1=3h$, $s_{11}=h^2$, and $s_2=8h^2$.
If a generated $\cE$ had these classes, a general section would vanish
at one reduced point $p$, giving
$0\longrightarrow\cO\longrightarrow \cE\longrightarrow I_p(3) \longrightarrow0$.
Apply $\operatorname{Hom}(-,\cO)$ to
$0\to\cO(1)\to\cO(2)^{\oplus2}\to I_p(3)\to0$.
Since $H^0(\cO(-1))=H^1(\cO(-2))=0$, one gets
$\operatorname{Ext}^1(I_p(3),\cO)=0$. The extension splits, contradicting
local freeness at $p$. This pair also lies in the classification of
\cite{Ell12}.
\end{proof}
The obstruction fixes $\PP^2$, $h$, and the rank; it does not exclude a
realization of the scalar degree polynomial on another source.

\section*{Use of generative-AI tools}
Generative-AI tools assisted with language editing, consistency checks,
computational testing, and the development and revision of proofs.
Responsibility for the mathematical content, attribution, and final
verification rests with the authors.

\end{document}